\documentclass[11pt,reqno]{amsart}
\usepackage{times}
\usepackage{amsmath,amsfonts,amstext,amssymb,amsbsy,amsopn,amsthm,eucal}
\usepackage{txfonts}
\usepackage{dsfont}
\usepackage{graphicx}   
\usepackage{hyperref}
\usepackage{accents}
\usepackage{enumerate}
\usepackage{a4wide,xcolor,mathrsfs}
\usepackage[normalem]{ulem}

\usepackage{epsfig,amsthm}
\usepackage[latin1]{inputenc}

\newcommand{\MM}{\mathscr{M}}

\newcommand{\supp}{\mathop{\rm supp}\nolimits}

\newcommand{\Lip}{\textrm{Lip}}

\newcommand{\diam}{\text{diam}}

\newcommand{\Ch}{{\rm Ch}}

\newcommand{\dom}{\rm D}

\newcommand{\mm}{\mathfrak m}
\newcommand{\nn}{\mathfrak n}
\newcommand{\sfd}{{\sf d}}
\renewcommand{\H}{{\rm H}}
\newcommand{\LL}{\mathscr{L}}
\renewcommand{\d}{{\rm{d}}}
\newcommand{\R}{\mathbb{R}}
\newcommand{\N}{\mathbb{N}}
\newcommand{\Tan}{{\rm Tan}}

\newcommand{\Prob}{\mathscr P}

\newcommand{\geo}{{\rm{Geo}}}                       
\newcommand{\e}{{\rm{e}}}                           
\newcommand{\gopt}{{\rm{OptGeo}}}                   

\newcommand{\ppi}{{\mbox{\boldmath$\pi$}}}
\newcommand{\ggamma}{{\mbox{\boldmath$\gamma$}}}
\newcommand{\sggamma}{{\mbox{\scriptsize\boldmath$\gamma$}}}

\newcommand{\CD}{{\sf CD}}
\newcommand{\RCD}{{\sf RCD}}

\newcommand{\DC}{\mathcal D_{C(\cdot)}}

\newcommand{\vare}{\varepsilon}

\newtheorem{theorem}{Theorem}[section]

\newtheorem{proposition}[theorem]{Proposition}
\newtheorem{lemma}[theorem]{Lemma}
\newtheorem{corollary}[theorem]{Corollary}
\theoremstyle{definition}
\newtheorem{definition}[theorem]{Definition}
\theoremstyle{remark}
\newtheorem{remark}{Remark}[section]
\theoremstyle{remark}

\theoremstyle{remark}

\theoremstyle{remark}

\theoremstyle{remark}

\theoremstyle{remark}

\theoremstyle{remark}

\begin{document}

\title{Structure Theory of Metric-Measure Spaces with Lower Ricci Curvature Bounds}

\author{Andrea Mondino}\thanks{A. Mondino: University of Warwick, Department of Mathematics. email: A.Mondino@warwick.ac.uk} 
\author{ Aaron Naber} \thanks{A. Naber:  Northwestern University, Department of Mathematics. email: anaber@math.northwestern.edu}

\date{\today}
\maketitle

\begin{abstract} We prove that a metric measure space $(X,\sfd,\mm)$ satisfying finite dimensional lower Ricci curvature bounds and whose Sobolev space $W^{1,2}$ is Hilbert is \emph{rectifiable}.  That is, a $\RCD^*(K,N)$-space is rectifiable, and in particular for $\mm$-a.e. point the  tangent cone is unique and euclidean of dimension at most $N$.  The proof is based on a maximal function argument combined with an original Almost Splitting Theorem via estimates on the gradient of the excess.   To this aim we also show a sharp integral Abresh-Gromoll type inequality on the excess function and an Abresh-Gromoll-type inequality on the gradient of the excess.  The argument is new even in the smooth setting. 

\end{abstract}

\tableofcontents

\section{Introduction}\label{s:introduction}

There is at this stage  a well developed structure theory for Gromov-Hausdorff limits of smooth Riemannian manifolds with lower Ricci curvature bounds, see for instance the work of Cheeger-Colding \cite{CC96, CC97,CC00a, CC00b} and  more recently \cite{CN} by Colding and the second author.  

On the other hand, in the last ten years, there has been a surge of activity on general metric measure spaces $(X,\sfd,\mm)$ satisfying a lower Ricci curvature bound in some generalized sense. This investigation began with the seminal papers of Lott-Villani \cite{Lott-Villani09}  and Sturm \cite{Sturm06I, Sturm06II}, though has been adapted considerably since the work of Bacher-Sturm \cite{BS2010} and Ambrosio-Gigli-Savar\'e \cite{AGS11a, AGS11b}.  The crucial property of any such definition is the compatibility with the smooth Riemannian case and the stability with respect to  measured Gromov-Hausdorff convergence.
While a great deal of progress has been made in this latter general framework, see for instance  \cite{AGMR2012, AGS11a, AGS11b, AGS12, AMS2013, AMSLocToGlob, BS2010,Cavalletti12, CS12,EKS2013, GaMo,Gigli12,GigliSplitting,GMS2013,GMR2013,GiMo12,R2011,R2013,Savare13,Villani09}, the structure theory on such metric-measure spaces is still much less developed than in the case of smooth limits.

The notion of  lower Ricci curvature bound on a general metric-measure space comes with two subtleties. The first is that of \emph{dimension}, and has been well understood since the work of Bakry-Emery \cite{BakryEmery_diffusions}:  in both the geometry and analysis of spaces with lower Ricci curvature bounds, it has become clear the correct statement is not that ``$X$ has Ricci curvature bounded from below by $K$'', but that ``$X$ has $N$-dimensional Ricci curvature bounded from below by $K$''. Such spaces are said to satisfy the $(K,N)$-\emph{Curvature Dimension} condition, $\CD(K,N)$ for short; a variant of this is that of \emph{reduced} curvature dimension bound, $\CD^*(K,N)$.  See \cite{BS2010, BakryEmery_diffusions, Sturm06II} and Section \ref{Sec:Prel} for more on this.

The second subtle point, which is particularly relevant for this paper, is that the classical definition of a metric-measure space with lower Ricci curvature bounds allows for Finsler structures (see the last theorem in \cite{Villani09}), which after the aforementioned works of Cheeger-Colding are known not to appear as limits of smooth manifolds with Ricci curvature lower bounds.  To address this issue, Ambrosio-Gigli-Savar\'e \cite{AGS11b} introduced a more restrictive condition which rules out Finsler geometries while retaining the stability properties under measured Gromov-Hausdorff convergence, see also \cite{AGMR2012} for the present simplified axiomatization.  In short, one studies the Sobolev space $W^{1,2}(X)$ of functions on $X$.  This space is always a Banach space, and the imposed extra condition  is that $W^{1,2}(X)$ is a Hilbert space.  Equivalently, the Laplace operator on $X$ is linear.  The notion of a lower Ricci curvature bound compatible with this last Hilbertian condition is called  \emph{Riemannian Curvature Dimension} bound, $\RCD$ for short.  Refinements of this have led to the notion of $\RCD^*(K,N)$-spaces, which is the key object of study in this paper.  See Section \ref{Sec:Prel} for a precise definition.

Remarkably, as proved by Erbar-Kuwada-Sturm \cite{EKS2013} and by Ambrosio-Savar\'e and the first author \cite{AMS2013}, the $\RCD^*(K,N)$ condition is equivalent to the dimensional Bochner inequality of Bakry-Emery \cite{BakryEmery_diffusions}.  There are various important consequences of this, and in particular the classicial Li-Yau and Harnack type estimates on the heat flow \cite{LY86}, known for Riemannian manifolds with lower Ricci bounds,  hold for $\RCD^*(K,N)$-spaces as well, see \cite{GaMo}.

More recently, an important contribution by Gigli \cite{GigliSplitting} has been to show that on $\RCD^*(0,N)$-spaces the analogue of the Cheeger-Gromoll Splitting Theorem \cite{ChGr} holds, thus providing a geometric property which fails on general $\CD(K,N)/\CD^*(K,N)$-spaces.  This was pushed by Gigli-Rajala and the first author in \cite{GMR2013} to prove that $\mm$-a.e. point in an $\RCD^*(K,N)$-space has a euclidean tangent cone; the possibility of having non unique tangent cones on a set of positive measure was conjectured to be false, but not excluded.
\\

In the present work we proceed in the investigation of the geometric properties of $\RCD^*(K,N)$-spaces by establishing their rectifiability, and consequently the $\mm$-a.e. uniqueness of tangent cones. More precisely the main result of the paper is the following:

\begin{theorem}[Rectifiability of $\RCD^*(K,N)$-spaces]\label{thm:rect}
Let $(X,\sfd, \mm)$ be an $\RCD^*(K,N)$-space, for some $K,N\in \R$ with $N> 1$. Then there exists a countable collection $\{R_j\}_{j \in \N}$ of $\mm$-measurable subsets of $X$, covering $X$ up to an $\mm$-negligible set, such that  each $R_j$ is biLipschitz to a measurable subset of $\R^{k_j}$, for some $1\leq k_j \leq N$, $k_j$ possibly depending on $j$.  
\end{theorem}
Actually, as we are going to describe below, we prove the following stronger rectifiability property: there exists $\bar{\vare}=\bar{\vare}(K,N)$ such that, if $(X,\sfd,\mm)$ is an $\RCD^*(K,N)$-space then for every $\vare\in (0, \bar \vare]$ there exists a countable collection $\{R^\vare_j\}_{j \in \N}$ of $\mm$-measurable subsets of $X$, covering $X$ up to an $\mm$-negligible set, such that  each $R_j^\vare$ is $(1+\vare)$-biLipschitz to a measurable subset of $\R^{k_j}$, for some $1\leq k_j \leq N$, $k_j$ possibly depending on $j$.

\begin{remark}
It will be a consequence of the proof that if $(X,\sfd, \mm)$ is a $\CD^*(K,N)$-space, then $X$ is $1+\vare$ rectifiable in the above sense for every $\vare \in (0,\bar{\vare}]$ if and only if $X$ is an $\RCD^*(K,N)$-space.
\end{remark}

From the constructions used to prove Theorem \ref{thm:rect} the $\mm$-a.e. uniqueness of the tangent cones follows  readily  (for the proof see  Section  \ref{SS:UniqTC}):

\begin{corollary}[$\mm$-a.e. uniqueness of tangent cones]\label{thm:UniqTC}
Let $(X,\sfd, \mm)$ be an $\RCD^*(K,N)$-space, for some $K,N\in \R$ with $N> 1$.  Then for $\mm$-a.e. $x \in X$ the tangent cone of $X$ at $x$   is unique and isometric to the $k_x$-dimensional euclidean space, for some $k_x \in \N$ with $1\leq k_x \leq N$. 
\end{corollary}

\subsection{Outline of Paper and Proof}

In the context when $X$ is a limit of smooth $n$-manifolds with $n$-dimensional Ricci curvature bounded from below, Theorem \ref{thm:rect} was first proved in \cite{CC97}.  There a key step was to prove hessian estimates on harmonic approximations of distance functions, and to use these to force splitting behavior.  In the context of general metric spaces the notion of a hessian is still not at the same level as it is for a smooth manifold, and cannot be used in such strength.  Instead we will prove entirely new estimates, both in the form of gradient estimates on the excess function and a new almost splitting theorem {\it with excess}, which will allow us to use the distance functions directly as our chart maps, a point which is new even in the smooth context.

In more detail, to prove Theorem \ref{thm:rect} we will first consider the stratification of $X$ composed by the following subsets $A_k\subset X$:
\begin{equation}\label{eq:defAkInt}
A_k:=\{x \in X \, : \, \text{there exists a tangent cone of $X$ at $x$ equal to  } \R^k \text{ but no tangent cone at } x \text{ splits } \R^{k+1}  \}.
\end{equation}
In Section   \ref{SS:Strat} it will be proved that $A_k$ is $\mm$-measurable, more precisely it is a difference of analytic subsets, and that $$\mm\left(X\setminus  \bigcup_{1\leq k \leq N} A_k\right)=0\quad .$$ Therefore Theorem \ref{thm:rect} and Corollary \ref{thm:UniqTC} will be consequences of the following more precise result, proved in Sections \ref{SS:Rect}-\ref{SS:UniqTC}.

\begin{theorem}[$\mm$-a.e. unique $k$-dimensional euclidean tangent cones and $k$-rectifiability of $A_k$]\label{thm:StructAk}
Let $(X,\sfd, \mm)$ be an $\RCD^*(K,N)$-space, for some $K,N\in \R ,N>1$ and let $A_k\subset X$, for $1\leq k \leq N$, be defined in \eqref{eq:defAkInt}. 

 Then the following holds:
\begin{enumerate}
\item For $\mm$-a.e. $x \in A_k$ the tangent cone of $X$ at $x$ is unique and isomorphic to the $k$-dimensional euclidean space.
\item  There exists $\bar{\vare}=\bar{\vare}(K,N)>0$ such that, for every $0<\vare\leq \bar{\vare}$, $A_k$ is $k$-rectifiable via   $1+\vare$-biLipschitz maps.  More precisely, for each $\vare>0$ we can cover $A_k$, up to an $\mm$-negligible subset, by a countable collection of sets $U^k_{\epsilon}$ with the property that each one is $1+\vare$-biLipschitz to a subset of $\R^k$.
\end{enumerate}
\end{theorem}

The proof of  Theorem \ref{thm:StructAk} is based on a  maximal function argument  combined with an explicit construction of Gromov-Hausdorff quasi-isometries with estimates (see Theorem \ref{lem:ConstrU}) and an original almost Splitting Theorem via  excess (see Theorem \ref{thm:AlmSplit}).

In a little more detail, given $\bar x\in A_k$ let $r>0$ such that $B_{\delta^{-1}r}(\bar x)$ is $\delta r$-close in the measured Gromov-Hausdorff sense to a ball in $\R^k$.  By the definition of $A_k$ we can find such $r>0$ for any $\delta>0$.  For some radius $r<<R<<\delta^{-1}r$ we can then pick points $\{p_i,q_i\}\in X$ which correspond to the bases $\pm R e_i$ of $\R^k$, respectively.  Let us consider the map $\vec d=\Big(\sfd(p_1,\cdot)-\sfd(p_1, \bar x),\ldots,\sfd(p_k,\cdot)-\sfd(p_k, \bar x)\Big):B_r(\bar x)\to \R^k$.  It is clear for $\delta$ sufficiently small that $\vec d$ is automatically an $\varepsilon r$-measured Gromov-Hausdorff map between $B_r(\bar x)$ and $B_r(0^k)$.  Our primary claim in this paper is that there is a set $U_\varepsilon\subseteq B_r(\bar x)$ of almost full measure such that for each $y\in U_\varepsilon$ and $s\leq r$, the restriction map $\vec d:B_{s}(y)\to\R^k$ is an $\varepsilon s$-measured Gromov-Hausdorff map.  From this we can show that the restriction map $\vec d:U_\varepsilon\to \R^k$ is in fact $1+\varepsilon$-biLipschitz onto its image.  By covering $A_k$ with such sets we will show that $A_k$ is itself rectifiable.

In order to construct the set $U_\varepsilon$ we rely on Theorem \ref{lem:ConstrU} in which it is shown that the gradient of the excess functions of the points $\{p_i,q_i\}$ is small in $L^2$.  Roughly, the set $U_\varepsilon$ is chosen by a maximal function argument to be the collection of points where the gradient of the excess remains small at all scales.  To exploit this information, in Section \ref{Sec:AlmSplit}, we  obtain an \emph{Almost Splitting Theorem via excess estimates}.  Roughly, this will tell us that at such points the $\R^k$ splitting is preserved at all scales, which is the required result to prove the main theorem.  Let us mention that the Almost Splitting Theorem in the smooth framework is  due to Cheeger-Colding \cite{CC96} and  is based on the existence of an ``almost line'';  here, the framework is the one of non smooth $\RCD^*(-\delta,N)$-spaces and the hypothesis on the existence of an ``almost line''  is replaced by an assumption on the smallness of the gradient of the excess.  Let us stress that this variant of Cheeger-Colding Almost Splitting Theorem is new even in the smooth setting. From the technical point of view our strategy is to use the estimates on the gradient of the excess  in order to construct an appropriate replacement for the Busemann function (which is a priori not available since we do not assume existence of lines) and then to adapt the arguments of the proof by Gigli  \cite{GigliSplitting}-\cite{GigliSplittingSur} of the Splitting Theorem in $\RCD^*(0,N)$-spaces.\\

In order to perform such a program, we start in Section \ref{Sec:Prel} by recalling basic notions of metric measure spaces, the measured Gromov-Hausdorff convergence, the definition of lower Ricci curvature bounds on metric-measure spaces, and a brief review of some of their basic properties.  In particular we will discuss some useful estimates and properties of the heat flow on such spaces which will be useful throughout this paper.\\

In Section \ref{S:EstHeatApprox}, inspired by the work  \cite{CN} of Colding and the second author, we regularize the distance function via the heat flow getting sharp estimates.  From a technical standpoint we also construct Lipschitz cut-off functions with $L^\infty$ estimates on the Laplacian.  Among other things this is used to obtain an improved integral Abresh-Gromoll inequality in $\RCD^*(K,N)$-spaces (see Theorem \ref{thm:ImprAbrGrom}) and an integral estimate on the \emph{gradient} of the excess function near a geodesic (see Theorem \ref{thm:GradExc}). Let us mention that the classical Abresh-Gromoll inequality was established in \cite{AG90} and then improved  to a sharp integral version in \cite{CN} in the smooth setting of Riemannian manifolds with lower Ricci curvature bounds.  In Theorem \ref{thm:ImprAbrGrom} we establish in $\RCD^*(K,N)$-spaces an analogue of the sharp  integral version of the Abresh-Gromoll inequality of  \cite{CN} and then use it to prove a new Abresh-Gromoll type inequality on the \emph{gradient} of the excess in Theorem \ref{thm:GradExc}.  This will be the starting point to construct the Gromov-Haudorff approximation with estimate, Theorem \ref{lem:ConstrU}, which is at the basis of the proof of the Rectifiability  Theorem \ref{thm:StructAk}, as explained above.\\

In Section \ref{s:Gromov_Hausdorff} we use the results established in Section \ref{S:EstHeatApprox} in order to show  that $A_k$ may be covered by distance function ``charts'' with good gradient estimates.  In particular this will rigorously construct the previously discussed sets $U_\varepsilon$.  In Section \ref{Sec:AlmSplit} we prove our {\it Almost Splitting with Excess} result in order to show that these charts have the required splitting behavior on sets of large measure.  Finally in Section \ref{Sec:MainRes} we combine these tools in order to prove our main theorems.  That is, using the almost splitting theorem we first show that the sets $U_\varepsilon$ are biLipschitz to subsets of $\R^k$, and then using a covering argument this yields the desired rectifiability of $A_k$.

\bigskip

\noindent {\bf Acknowledgment.}  The second author acknowledges the support
of the ETH Fellowship. He  wishes to express his deep gratitude to Luigi Ambrosio, Nicola Gigli and Giuseppe Savar\'e  for having introduced him to the topic of metric measure spaces with lower Ricci curvature bounds.

\section{Preliminaries and notation}\label{Sec:Prel}
\subsection{Pointed metric measure spaces and their equivalence classes}
The basic objects we will deal with throughout the paper are metric measure spaces and pointed metric measure spaces, m.m.s. and p.m.m.s. 
 for short. First of all let us recall the standard definitions. 
 
A m.m.s. is a triple $(X,\sfd,\mm)$ where $(X,\sfd)$ is a complete and separable metric space and  $\mm$ is a locally finite (i.e. finite on bounded subsets) non-negative complete Borel measure  on it.

It will often be the case that the measure 
$\mm$ is \emph{doubling}, i.e. such that
\begin{equation}\label{eq:mmdoubl}
0< \mm(B_{2r}(x)) \leq C(R) \;  \mm(B_r(x)),\qquad\forall x\in X,\ r\leq R,
\end{equation}
for some positive function $C(\cdot):[0,+\infty) \to (0,+\infty)$ which can, and will, be taken to be non-decreasing.

The bound  \eqref{eq:mmdoubl} implies that $\supp \mm=X$ and $\mm \neq 0$ and by  iteration one gets
\begin{equation}\label{eq:genDoub}
 {\mm(B_R(a))}\leq \mm(B_r(x))\big(C(R)\big)^{\log_2(\frac{r}{R})+2}, \qquad\forall 0<r\leq R,\ a\in X,\ x\in B_R(a).
\end{equation}
In particular  bounded subsets are totally bounded and hence doubling spaces are proper. 

\smallskip

A p.m.m.s is a quadruple $(X,\sfd,\mm,\bar x)$ where $(X,\sfd,\mm)$ is a metric measure space and $\bar x\in \supp(\mm)$ is a given reference point. Two p.m.m.s.  $(X,\sfd,\mm, \bar{x})$,  $(X',\sfd',\mm', \bar{x}')$ are said to be isomorphic if there exists an isometry $T:(\supp(\mm),\sfd)\to(\supp(\mm'),\sfd')$ such that $T_\sharp\mm=\mm'$ and $T(\bar{x})=\bar{x}'$.

We say that a p.m.m.s. $(X,\sfd,\mm,\bar x)$ is normalized provided $\int_{B_1(\bar x)}1-\sfd(\cdot,\bar x)\,\d\mm=1$. Obviously, given any p.m.m.s. $(X,\sfd,\mm,\bar x)$ there exists a unique $c> 0$ such that $(X,\sfd,c\mm,\bar x)$ is normalized, namely $c:=(\int_{B_1(\bar x)}1-\sfd(\cdot,\bar x)\,\d\mm)^{-1}$.

We denote by $\MM_{C(\cdot)}$ the class of (isomorphism classes of) normalized p.m.m.s. fulfilling \eqref{eq:mmdoubl} for a given non-decreasing $C:(0,\infty)\to(0,\infty)$.
 
\subsection{Pointed measured Gromov-Hausdorff topology and  measured tangents}
We will adopt the following definition of convergence of p.m.m.s (see \cite{BuragoBuragoIvanov}, \cite{GMS2013}
 and  \cite{Villani09}):
\begin{definition}[Pointed measured Gromov-Hausdorff convergence]\label{def:conv}
A sequence $(X_j,\sfd_j,\mm_j,\bar{x}_j)$ is said to converge 
in the  pointed measured Gromov-Hausdorff topology (p-mGH for short) to 
$(X_\infty,\sfd_\infty,\mm_\infty,\bar{x}_\infty)$ if there 
exists a separable metric space $(Z,\sfd_Z)$ and isometric embeddings  
$\{\iota_j:(\supp(\mm_j),\sfd_j)\to (Z,\sfd_Z)\}_{i \in \bar{\N}}$ such that
for every 
$\varepsilon>0$ and $R>0$ there exists $i_0$ such that for every $i>i_0$
\[
\iota_\infty(B^{X_\infty}_R(\bar{x}_\infty)) \subset B^Z_{\varepsilon}[\iota_j(B^{X_j}_{R+\varepsilon} (\bar{x}_j))]  \qquad \text{and} \qquad  \iota_j(B^{X_j}_R(\bar{x}_j)) \subset B^Z_{\varepsilon}[\iota_\infty(B^{X_\infty}_{R+\varepsilon} (\bar{x}_\infty))], 
\]
where $B^Z_\varepsilon[A]:=\{z \in Z: \, \sfd_Z(z,A)<\varepsilon\}$ for every subset $A \subset Z$, and 
\[
\int_Y \varphi \, \d ((\iota_j)_\sharp(\mm_j))\qquad    \to \qquad  \int_Y \varphi \, \d  ((\iota_\infty)_\sharp(\mm_\infty)) \qquad \forall \varphi \in C_b(Z), 
\]
where $C_b(Z)$ denotes the set of real valued bounded continuous functions with bounded support in $Z$.
\end{definition}
Sometimes in the following, for simplicity of notation, we will identify the spaces $X_j$ with 
their isomorphic copies $\iota_j(X_j)\subset Z$. 

It is obvious that this is in fact a notion of convergence for isomorphism classes of p.m.m.s., moreover it is induced by a metric   (see e.g. \cite{GMS2013} for details):
\begin{proposition}\label{prop:comp}
Let $C:(0,\infty)\to (0,\infty)$ be a non-decreasing function. Then there exists a distance $\mathcal D_{C(\cdot)}$ on $\MM_{C(\cdot)}$ for which converging sequences are precisely those converging in the p-mGH sense. Furthermore, the space  $(\MM_{C(\cdot)},\mathcal D_{C(\cdot)})$  is compact.
\end{proposition}
Notice that the compactness of  $(\MM_{C(\cdot)},\mathcal D_{C(\cdot)})$ follows by the standard argument of Gromov: the measures of spaces in  $\MM_{C(\cdot)}$ are uniformly doubling, hence balls of given radius around the reference points are uniformly totally bounded and thus compact in the GH-topology. Then weak compactness of the measures follows using the doubling condition again and the fact that they are normalized.

Before defining the measured tangents, let us recall that an equivalent way to define p-mGH convergence is via $\vare$-quasi isometries as follows.
\begin{proposition}[Equivalent definition of p-mGH convergence]
Let $(X_n,\sfd_n,\mm_n,\bar x_n)$, $n\in\N\cup\{\infty\}$,
be pointed metric measure spaces as above. Then  $(X_n,\sfd_n,\mm_n,\bar x_n)\to
(X_\infty,\sfd_\infty,\mm_\infty,\bar x_\infty)$ in the pmGH-sense if and only if  for any $\vare,R>0$ there exists $N({\vare,R})\in \N$ such that for all $n\geq N({\vare,R})$ there exists a Borel map $f^{R,\vare}_n:B_R(\bar x_n)\to X_\infty$ such that
\begin{itemize}
\item $f^{R,\vare}_n(\bar x_n)=\bar x_\infty$,
\item $\sup_{x,y\in B_R(\bar x_n)}|\sfd_n(x,y)-\sfd_\infty(f^{R,\vare}_n(x),f^{R,\vare}_n(y))|\leq\vare$,
\item the $\vare$-neighbourhood of $f^{R,\vare}_n(B_R(\bar x_n))$ contains $B_{R-\vare}(\bar x_\infty)$,
\item  $(f^{R,\vare}_n)_\sharp(\mm_n\llcorner{B_R(\bar x_n)})$ weakly converges to $\mm_\infty\llcorner{B_R(x_\infty)}$ as $n\to\infty$, for a.e. $R>0$.
\end{itemize}
 \end{proposition}
\bigskip

A crucial role in this paper is played by measured tangents, which are defined as follows. Let  $(X,\sfd,\mm)$ be a m.m.s.,  $\bar x\in \supp(\mm)$ and $r\in(0,1)$; we consider the rescaled and normalized p.m.m.s. $(X,r^{-1}\sfd,\mm^{\bar{x}}_r,\bar x)$ where the measure $\mm^{\bar x}_r$ is given by
\begin{equation}
\label{eq:normalization}
\mm^{\bar x}_r:=\left(\int_{B_r(\bar x)}1-\frac 1r\sfd(\cdot,\bar x)\,\d\mm\right)^{-1}\mm.
\end{equation}
Then we define:
\begin{definition}[The collection of tangent spaces $\Tan(X,\sfd,\mm,\bar{x})$]
Let  $(X,\sfd,\mm)$ be a m.m.s. and  $\bar x\in \supp(\mm)$. A p.m.m.s.  $(Y,\sfd_Y,\nn,y)$ is called a
\emph{tangent} to $(X,\sfd,\mm)$ at $\bar{x} \in X$ if there exists a sequence of radii $r_i \downarrow 0$ so that
$(X,r_i^{-1}\sfd,\mm^{\bar{x}}_{r_i},\bar{x}) \to (Y,\sfd_Y,\nn,y)$ as 
$i \to \infty$ in the pointed measured Gromov-Hausdorff topology.

We denote the collection of all the tangents of $(X,\sfd,\mm)$ at 
$\bar{x} \in X$ by $\Tan(X,\sfd,\mm,\bar{x})$. 
\end{definition}
\begin{remark}
See \cite{CN13} for basic properties of $\Tan(X,\sfd,\mm,\bar{x})$ for Ricci-limit spaces.
\end{remark}

Notice that if $(X,\sfd,\mm)$ satisfies \eqref{eq:mmdoubl} for some non-decreasing $C:(0,\infty)\to(0,\infty)$, then $(X,r^{-1}\sfd,\mm^{\bar x}_r,\bar x)\in \MM_{C(\cdot)}$ for every $\bar x\in X$ and $r\in(0,1)$ and hence   the compactness stated in Proposition \ref{prop:comp} ensures that the set $\Tan(X,\sfd,\mm,\bar{x})$ is non-empty. 

It is also worth to notice that the map
\[
\supp(\mm) \ni x\qquad\mapsto \qquad (X,\sfd,\mm^x_r,x),
\]
is (sequentially) $\sfd$-continuous for every $r>0$, the target space being endowed with the p-mGH convergence. 
 
\subsection{Cheeger energy and Sobolev Classes}
 It is out of the scope of this short subsection  to provide full details about the definition of the Cheeger energy and the associated Sobolev space $W^{1,2}(X,\sfd,\mm)$, we will instead be satisfied in recalling some  basic notions used in the paper (we refer to \cite{AGS11a}, \cite{AGS11b}, \cite{AGS12} for the basics on calculus in metric 
measure spaces).

First of all recall that on a m.m.s. there is not a canonical notion of ``differential of a function'' $f$ but at least one has an $\mm$-a.e. defined  ``modulus of the differential'', called weak upper differential and denoted with $|Df|_w$; let us just mention that this object  arises from the relaxation
in $L^2(X,\mm)$ of the local Lipschitz constant 
\begin{equation}
  \label{eq:20}
  |D f|(x):=\limsup_{y\to x}\frac{|f(y)-f(x)|}{\sfd(y,x)},\quad
  f:X\to \R,
\end{equation}
of Lipschitz functions. With this object one defines the Cheeger energy
$$\Ch(f):=\frac 1 2 \int_X |Df|_w^2 \, \d \mm.$$
The  Sobolev space $W^{1,2}(X,\sfd,\mm)$ is by definition  the space of $L^2(X,\mm)$ functions having finite Cheeger energy, and it is endowed with the natural norm  $\|f\|^2_{W^{1,2}}:=\|f\|^2_{L^2}+2 \Ch(f)$ which makes it a Banach space. We remark that, in general, $W^{1,2}(X,\sfd,\mm)$ is not Hilbert (for instance, on a smooth Finsler manifold the space $W^{1,2}$ is Hilbert if and only if the manifold is actually Riemannian); in case  $W^{1,2}(X,\sfd,\mm)$ is  Hilbert then, following the notation introduced in \cite{AGS11b} and \cite{Gigli12}, we say that $(X,\sfd,\mm)$ is \emph{infinitesimally Hilbertian}.  As explained in \cite{AGS11b}, \cite{AGS12}, the quadratic form $\Ch$ canonically 
induces a strongly regular Dirichlet
form in $(X,\tau)$, where $\tau$ is the topology induced by $\sfd$. In addition, but this fact is less elementary (see \cite[\S4.3]{AGS11b}), the formula
$$
\Gamma(f)=|D f|_w^2,\quad
\Gamma(f,g)=\lim_{\epsilon\downarrow 0}\frac{|D(f+\epsilon g)|_w^2-|Df|_w^2}{2\epsilon}
\quad\qquad f,\,g\in W^{1,2}(X,\sfd,\mm)\, , $$%
where the limit takes place in $L^1(X,\mm)$, provides an explicit expression  
of the associated \emph{Carr\'e du Champ} $\Gamma:W^{1,2}(X,\sfd,\mm) \times W^{1,2}(X,\sfd,\mm)    \to
L^1(X,\mm)$ and 
yields the pointwise
upper estimate
\begin{equation}
  \label{eq:19}
  \Gamma(f) \leq |D f|^2\quad\text{$\mm$-a.e.\ in $X$,
    whenever }f\in \Lip(X)\cap L^2(X,\mm),\quad |D f|\in L^2(X,\mm),
\end{equation}
where, of course, $\Lip(X)$ denotes the set of real valued Lipschitz  functions on $(X,\sfd)$. Observe that clearly, in a smooth Riemannian setting, the Carr\'e du Champ $\Gamma(f,g)$ coincides  with the usual scalar product of the gradients of the functions $f$ and $g$. Moreover by a nontrivial result of Cheeger \cite{Cheeger97} we have in  locally doubling \& Poincar\'e spaces that for locally Lipschitz functions the local Lipschitz constant  and the weak upper differential coincide $\mm$-a.e..  Below we will make use of the local Sobolev space $W^{1,2}_{loc}(\Omega)$, for $\Omega\subset X$ open set; by definition $W^{1,2}_{loc}(\Omega)$ is made of those Borel functions $f:\Omega\to \R$ such that for every Lipschitz function $\chi: X\to \R$ with bounded support well contained in $\Omega$ (i.e. having strictly positive distance from $X\setminus \Omega$) it holds $\chi f \in W^{1,2}(X,\sfd,\mm)$, where by definition we set $\chi f=0$ on $X\setminus \Omega$.

\subsection{Lower Ricci curvature bounds}\label{SS:defRCD}
In this subsection  we quickly recall some basic definitions and properties of spaces with lower Ricci curvature bounds that we will use later on.

We denote by $\Prob(X)$ the space of Borel probability measures on the complete and separable metric space $(X,\sfd)$ and by $\Prob_2 (X) \subset \Prob (X)$ the subspace consisting of all the probability measures with finite second moment.

For $\mu_0,\mu_1 \in \Prob_2(X)$ the quadratic transportation distance $W_2(\mu_0,\mu_1)$ is defined by
\begin{equation}\label{eq:Wdef}
  W_2^2(\mu_0,\mu_1) = \inf_{\sggamma} \int_X \sfd^2(x,y) \,\d\ggamma(x,y),
\end{equation}
where the infimum is taken over all $\ggamma \in \Prob(X \times X)$ with $\mu_0$ and $\mu_1$ as the first and the second marginals.

Assuming the space $(X,\sfd)$ is a length space, also the space $(\Prob_2(X), W_2)$ is a length space. We denote  by $\geo(X)$ the space of (constant speed minimizing) geodesics on $(X,\sfd)$ endowed with the $\sup$ distance, and by $\e_t:\geo(X)\to X$, $t\in[0,1]$, the evaluation maps defined by $\e_t(\gamma):=\gamma_t$. It turns out that any geodesic $(\mu_t) \in \geo(\Prob_2(X))$ can be lifted to a measure $\ppi \in \Prob(\geo(X))$, so that $(\e_t)_\#\ppi = \mu_t$ for all $t \in [0,1]$. Given $\mu_0,\mu_1\in \Prob_2(X)$, we denote by $\gopt(\mu_0,\mu_1)$ the space of all
$\ppi \in \Prob(\geo(X))$ for which $(\e_0,\e_1)_\#\ppi$ realizes the minimum in \eqref{eq:Wdef}. If $(X,\sfd)$ is a length space, then the set $\gopt(\mu_0,\mu_1)$ is non-empty for any $\mu_0,\mu_1\in \Prob_2(X)$.

We turn to the formulation of the $\CD^*(K,N)$ condition, coming from  \cite{BS2010}, to which we also refer for a detailed discussion of its relation with the $\CD(K,N)$ condition
 (see also \cite{Cavalletti12}).

Given $K \in \R$ and $N \in [1, \infty)$, we define the distortion coefficient $[0,1]\times\R^+\ni (t,\theta)\mapsto \sigma^{(t)}_{K,N}(\theta)$ as
\[
\sigma^{(t)}_{K,N}(\theta):=\left\{
\begin{array}{ll}
+\infty,&\qquad\textrm{ if }K\theta^2\geq N\pi^2,\\
\frac{\sin(t\theta\sqrt{K/N})}{\sin(\theta\sqrt{K/N})}&\qquad\textrm{ if }0<K\theta^2 <N\pi^2,\\
t&\qquad\textrm{ if }K\theta^2=0,\\
\frac{\sinh(t\theta\sqrt{K/N})}{\sinh(\theta\sqrt{K/N})}&\qquad\textrm{ if }K\theta^2 <0.
\end{array}
\right.
\]
\begin{definition}[Curvature dimension bounds]\label{d:CD*}
Let $K \in \R$ and $ N\in[1,  \infty)$. We say that a m.m.s.  $(X,\sfd,\mm)$
 is a $\CD^*(K,N)$-space if for any two measures $\mu_0, \mu_1 \in \Prob(X)$ with support  bounded and contained in $\supp(\mm)$ there
exists a measure $\ppi \in \gopt(\mu_0,\mu_1)$ such that for every $t \in [0,1]$
and $N' \geq  N$ we have
\begin{equation}\label{eq:CD-def}
-\int\rho_t^{1-\frac1{N'}}\,\d\mm\leq - \int \sigma^{(1-t)}_{K,N'}(\sfd(\gamma_0,\gamma_1))\, \rho_0^{-\frac1{N'}}(\gamma_{0})+\sigma^{(t)}_{K,N'}(\sfd(\gamma_0,\gamma_1))\, \rho_1^{-\frac1{N'}}(\gamma_{1}) \,\d\ppi(\gamma)
\end{equation}
where for any $t\in[0,1]$ we  have written $(\e_t)_\sharp\ppi=\rho_t\mm+\mu_t^s$  with $\mu_t^s \perp \mm$.
\end{definition}
Notice that if $(X,\sfd,\mm)$ is a $\CD^*(K,N)$-space, then so is $(\supp(\mm),\sfd,\mm)$, hence it is not restrictive to assume that $\supp(\mm)=X$. It is also immediate to establish that
\begin{equation}
\label{eq:cdinv}
\begin{split}
&\text{If }(X,\sfd,\mm)\text{ is $\CD^*(K,N)$, then the same is true for $(X,\sfd,c\mm)$ for any }c>0.\\
&\text{If }(X,\sfd,\mm)\text{ is $\CD^*(K,N)$, then for $\lambda>0$ the space $(X,\lambda \sfd,\mm)$ is $\CD^*(\lambda^{-2}K,N)$}.
\end{split}
\end{equation}

On $\CD^*(K,N)$ a natural version of the Bishop-Gromov volume growth estimate holds (see \cite{BS2010} for the precise statement), it follows that  for any given $K\in\R$, $N\in[1,\infty)$ there exists a function $C:(0,\infty)\to(0,\infty)$ depending on $K,N$ such that any $\CD^*(K,N)$-space $(X,\sfd,\mm)$ fulfills \eqref{eq:mmdoubl}. 

In order to avoid the Finsler-like behavior of spaces with a curvature-dimension bound, the $\CD^*(K,N)$ condition may been strengthened  by requiring also that the Banach space $W^{1,2}(X,\sfd,\mm)$ is Hilbert. Such spaces are said to satisfy the \emph{Riemannian $\CD^*(K,N)$} condition  denoted with $\RCD^*(K,N)$.

Now we state three fundamental properties of  $\RCD^*(K,N)$-spaces (the first one is proved in \cite{Gigli12}, the second in  \cite{GMS2013} and  the third in \cite{GigliSplitting}). Let us first introduce the coefficients $\tilde{\sigma}_{K,N}(\cdot):[0,\infty)\to\R$ defined by
 \[
\tilde{\sigma}_{K,N}(\theta):=\left\{
\begin{array}{ll}
\theta \sqrt{\frac{K}{N}} \, {\rm cotan} \left(\theta \sqrt{\frac{K}{N}} \right),&\qquad\textrm{ if }K>0,\\
1 &\qquad\textrm{ if }K=0 ,\\
\theta \sqrt{-\frac{K}{N}} \, {\rm cotanh} \left(\theta \sqrt{-\frac{K}{N}} \right),&\qquad\textrm{ if }K<0.
\end{array}
\right.
\]
Recall that given an open subset $\Omega\subset X$, we say that a Sobolev function $f \in W^{1,2}_{loc}(\Omega, \sfd, \mm\llcorner \Omega )$ is in the domain of the Laplacian and write $f \in \dom(\Delta^\star, \Omega)$, if there exists a Radon measure $\mu$ on $\Omega$ such that for every $\psi\in \Lip(X)\cap L^1(\Omega, |\mu|)$ with compact support in $\Omega$ it holds
$$-\int_{\Omega} \Gamma(f,\psi) \, \d \mm = \int_{\Omega} \psi \, \d \mu \quad. $$
In this case we write $\Delta^\star f |_{\Omega}:=\mu$; to avoid cumbersome notation, if $\Omega=X$ we simply write $\Delta^\star f$. If moreover $\Delta^\star f$ is absolutely continuous with respect to $\mm$ with $L^2_{loc}$ density, we denote by $\Delta f$ the unique function such that: $\Delta^\star f = (\Delta f) \, \mm$, $\Delta f \in L^2_{loc}(X,\mm)$. In this case,  for every $\psi \in W^{1,2}(X,\sfd,\mm)$ with compact support, the  following integration by parts formula holds:
$$-\int_{X} \Gamma(f,\psi) \, \d \mm= \int_{X} \Delta f \, \psi \, \d \mm \quad. $$
Finally, if $\Delta f \in L^{2}(X,\mm)$,  we write $f \in \dom(\Delta)$.

\begin{theorem}[Laplacian comparison for the distance function]\label{thm:LapComp}
Let $(X,\sfd,\mm)$ be an $\RCD^*(K,N)$-space for some $K\in \R$ and $N\in(1,\infty)$. For $x_0\in X$ denote by $\sfd_{x_0} : X \to [0,+\infty)$  the
function $x \mapsto \sfd(x, x_0)$. Then
$$\frac{\sfd^2_{x_0}}{2}\in \dom(\Delta^\star) \quad \text{with} \quad \Delta^\star \frac{\sfd^2_{x_0}}{2} \leq N\, \tilde{\sigma}_{K,N}(\sfd_{x_0}) \,\mm \quad \forall x_0 \in X  $$ 
and
$$\sfd_{x_0} \in \dom(\Delta^\star, X\setminus\{x_0\})  \quad \text{with} \quad \Delta^\star \sfd_{x_0} |_{X\setminus \{x_0\}}\leq \frac{ N \, \tilde{\sigma}_{K,N}(\sfd_{x_0})-1}{\sfd_{x_0}}\, \mm \quad \forall x_0 \in X.   $$
\end{theorem}

\begin{theorem}[Stability]\label{thm:stab}
Let $K\in \R$ and $N\in[1,\infty)$. Then the class of normalized p.m.m.s $(X,\sfd,\mm,\bar x)$ such that $(X,\sfd,\mm)$ is $\RCD^*(K,N)$ is closed (hence compact) w.r.t. p-mGH convergence.
\end{theorem}
\begin{theorem}[Splitting]\label{thm:splitting}
 Let $(X,\sfd,\mm)$ be an $\RCD^*(0,N)$-space with $1 \le N < \infty$. Suppose that
 $\supp(\mm)$ contains a line. Then $(X,\sfd,\mm)$ is isomorphic to $(X'\times \R, \sfd'\times \sfd_E,\mm'\times \LL_1)$,
where $\sfd_E$ is the Euclidean distance, $\LL_1$ the Lebesgue measure and  $(X',\sfd',\mm')$ is an $\RCD^*(0,N-1)$-space if $N \ge 2$ and a singleton if $N < 2$.
\end{theorem}

Notice that for the particular case $K=0$ the $\CD^*(0,N)$ condition is the same as the $\CD(0,N)$ one. Also, in the statement of the splitting theorem, by line we intend an isometric embedding of $\R$.

Observe that Theorem \ref{thm:stab} and properties \eqref{eq:cdinv} ensure that for any $K,N$ we have that
\begin{equation}
\label{eq:tancd}
\begin{split}
&\text{If }(X,\sfd,\mm)\text{ is an  $\RCD^*(K,N)$-space and $x\in X$ then}\\
&\text{ every $(Y,\sfd,\nn,y)\in\Tan(X,\sfd,\mm,x)$ is  $\RCD^*(0,N)$.}
\end{split}
\end{equation}
By iterating Theorem \ref{thm:stab} and Theorem \ref{thm:splitting},  in \cite{GMR2013} the following result has been established.
\begin{theorem}[Euclidean Tangents]\label{thm:euclideantangents}
 Let $K \in \R$, $1 \le N < \infty$ and  $(X,\sfd,\mm)$  a $\RCD^*(K,N)$-space.
 Then at $\mm$-almost every $x \in X$ there exists $k \in \N$, $1\le k \le N$, such that 
 $$
  (\R^k,\sfd_{E},\LL_k,0) \in \Tan(X,\sfd,\mm,x),  
 $$
 where $\sfd_{E}$ is the Euclidean distance and $\LL_k$ is the $k$-dimensional
 Lebesgue measure normalized so that $\int_{B_1(0)}1-|x|\,\d\LL_k(x)=1$.
\end{theorem}
Let us remark that the normalization of the limit measure expressed in the statement plays little role  and depends only on the choice of renormalization of rescaled measures in the process of taking limits. Let us also mention  that a fundamental ingredient in the proof of Theorem \ref{thm:euclideantangents} was a crucial idea of Preiss \cite{P1987} (adapted to doubling metric spaces by Le Donne \cite{LD2011} and to doubling metric measure spaces in \cite{GMR2013})
stating that  ``tangents of tangents are tangents'' almost everywhere. We report here the statement (see \cite[Theorem 3.2]{GMR2013} for the proof) since it will be useful also in this work.
\begin{theorem}[``Tangents of tangents are tangents'']\label{thm:iteratedtangents}
Let  $(X,\sfd,\mm)$ be a m.m.s. satisfying \eqref{eq:mmdoubl} for some $C:(0,\infty)\to(0,\infty)$.

Then for  $\mm$-a.e.  $x \in X$ the following holds:  for any $(Y,\sfd_Y,\nn,y) \in \Tan(X,\sfd,\mm,x)$ and any $y' \in Y$ we have
 \[
  \Tan(Y,\sfd_Y,\nn^{y'}_1,y') \subset \Tan(X,\sfd,\mm,x),
 \]
 the measure $\nn^{y'}_1$ being defined as in \eqref{eq:normalization}.
\end{theorem}

\subsection{Convergence of functions defined on varying spaces}\label{SubSec:ConvFunct} 
In this subsection we recall some basic facts about the convergence of functions defined on m.m.s.  which are themselves converging to a limit space (for more material the interested reader is referred to \cite{GMS2013} and the references therein).  

Let $(X_j, \sfd_j, \mm_j, \bar{x}_j)$ be a sequence of p.m.m.s.  in $\MM_{C(\cdot)}$, for some nondecreasing $C(\cdot):(0,+\infty)\to (0,+\infty)$,  p-mGH converging to  a limit p.m.m.s $(X_\infty, \sfd_\infty, \mm_\infty, \bar{x}_\infty)$. Following Definition \ref{def:conv}, let  $(Z,\sfd_Z)$ be an ambient Polish  metric space and let $\iota_j:(X_j, \sfd_j) \to (Z,\sfd_Z)$, $j \in \N\cup \{\infty\}$ be isometric immersions realizing the convergence.  
First we define  pointwise and uniform convergence of functions defined on varying spaces.

\begin{definition}[Pointwise and uniform convergence of functions defined on varying spaces]
Let $(X_j, \sfd_j, \mm_j, \bar{x}_j)$, $j\in \N\cup \{\infty\}$, be a p-mGH converging sequence of p.m.m.s. as above and let   $f_j:X_j \to \R$, $j \in \N \cup \{\infty\}$, be  a sequence of functions. We say that $f_j\to f_\infty$ \emph{pointwise}  if 
\begin{equation}\label{eq:DefPointConv}
f_j(x_j)\to f_\infty(x_\infty) \quad \text{for every sequence of points $x_j \in X_j$ such that } \iota_j(x_j)\to \iota_\infty(x_\infty) .
\end{equation}
If moreover for every $\vare>0$ there exists $\delta>0$ such that 
\begin{equation}\label{eq:DefPointConv}
|f_j(x_j)- f_\infty(x_\infty)| \leq \vare \quad \text{ for every $j\geq \delta^{-1}$ and every   $x_j \in X_j, x_\infty \in X_\infty$ with  $\sfd_{Z}(\iota_j(x_j),\iota_\infty(x_\infty)) \leq \delta$\quad,    } 
\end{equation}
then we say that $f_j\to f_\infty$ \emph{uniformly}. 
\end{definition}

By using the separability of the metric spaces, one can repeat the classic proof of Arzel\'a-Ascoli Theorem based on extraction of  diagonal subsequences  and get the following proposition.

\begin{proposition} [Arzel\'a-Ascoli Theorem for varying spaces] \label{Prop:AscArz}
Let $(X_j, \sfd_j, \mm_j, \bar{x}_j), j\in \N\cup \{\infty\}$, be a p-mGH converging sequence of proper p.m.m.s. as above and let   $f_j:X_j \to \R$, $j \in \N$, be  a sequence of $L$-Lipschitz   functions, for some uniform $L \geq 0$, which satisfy $\sup_{j \in \N} |f_j(\bar{x_j})|< \infty$.  Then there exists a limit $L$-Lipschitz function $f_\infty: X_\infty \to \R$ such that, up to subsequences,  $f_j|_{B_R(\bar{x}_j)} \to f_\infty|_{B_R(\bar{x}_\infty)}$ uniformly for every $R>0$. 
\end{proposition}

By recalling that $\RCD^*(K,N)$-spaces satisfy doubling \& Poincar\'e with constant depending just on $K,N$ and moreover, since $W^{1,2}$ is Hilbert (so in particular reflexive),  one can repeat the proof of the lower semicontinuity of the slope given in \cite[Theorem 8.4]{ACD}, see also the previous work of Cheeger \cite{Cheeger97}, in order to obtain the following variant for p-mGH converging spaces.

\begin{proposition}[Lower semicontuity of the slope in $\RCD^*(K,N)$-spaces] \label{Prop:LSCSlope}

Let $(X_j, \sfd_j, \mm_j, \bar{x}_j)$, $j\in \N\cup \{\infty\}$, be a p-mGH converging sequence of $\RCD^*(K,N)$-spaces as above and  let   $f_j:X_j \to \R$, $j \in \N\cup \{\infty\}$, be  a sequence of locally Lipschitz   functions such that  $f_j|_{B_R(\bar{x_j})} \to f_\infty|_{B_R(\bar{x}_\infty)}$ uniformly for some $R>0$.

Then, for every $0<r<R$ one has
\begin{equation}\label{eq:lscSlope}
\int_{B_r(\bar{x}_\infty)} |Df_\infty|^2 \, \d \mm_\infty \leq \liminf_{j\to \infty}  \int_{B_r(\bar{x}_j )} |Df_j|^2 \, \d \mm_j\quad. 
\end{equation} 
\end{proposition} 

\subsection{Heat flow on $\RCD^*(K,N)$-spaces}\label{SS:Heat}
Even if many of the results in this subsection hold in higher generality (see for instance \cite{AGMR2012}, \cite{AGS11a}, \cite{AGS11b}), as in this paper we will deal with $\RCD^*(K,N)$-spaces we focus the presentation to this case.

Since  $\Ch$ is a convex and lower semi-continuous functional on $L^2(X,\mm)$, applying the classical theory of gradient flows of  convex functionals in Hilbert spaces (see for instance \cite{Ambrosio-Gigli-Savare08} for a comprehensive presentation) one can study its gradient flow  in the space $L^2(X,\mm)$. More precisely one obtains that for every $f \in L^2(X,\mm)$ there exists a continuous curve $(f_t)_{ \in [0,\infty)}$ in $L^2(X,\mm)$, locally absolutely continuous in $(0, \infty)$ with $f_0=f$ such that $$f_t \in \dom(\Delta) \quad \text{ and } \quad  \frac{\d^+}{\d t} f_t = \Delta f_t \quad, \quad \forall t>0. $$ 
This produces a semigroup $(\H_t)_{t\geq 0}$ on $L^2(X,\mm)$ defined by $\H_t f= f_t$, where $f_t$ is the unique $L^2$-gradient flow of $\Ch$.

An important property of the heat flow is the maximum (resp. minimum) principle, see   \cite[Theorem 4.16]{AGS11a}: if $f\in L^2(X,\mm)$ satisfies $f \leq C$ $\mm$-a.e. (resp. $f \geq C$ $\mm$-a.e.), then also $\H_t f \leq C$ $\mm$-a.e.  (resp. $\H_t f \geq C$ $\mm$-a.e.) for all $t \geq 0$. Moreover the heat flow preserves the mass:  for every $f \in L^2(X,\mm)$ 
$$\int_X \H_t f \, \d \mm=\int_X f \, \d \mm, \quad \forall t \geq 0.$$ 
A nontrivial property of the heat flow proved for $\RCD(K,\infty)$-spaces in \cite[Theorem 6.8]{AGS11b} (see also \cite{AGMR2012} for the generalization to $\sigma$-finite measures) is the Lipschitz regularization; namely if $f\in L^2(X,\mm)$ then $\H_t f \in \dom(\Ch)$ for every $t>0$ and 
$$2 \,I_{2K}(t) \; \Gamma(\H_t f) \leq \H_t(f^2) \quad \mm\text{-a.e. in } X,$$
where $I_{2K}(t):=\int_0^t e^{2Ks} \, \d s= \frac{e^{2Kt}-1}{2K}$; in particular, if $f \in L^\infty(X,\mm)$ then $\H_t f$ has a Lipschitz representative for every $t>0$ and 
\begin{equation}\label{eq:LipRegHeat}
\sqrt{2\, I_{2K}(t)} \,  | D\H_t f| \leq \|f\|_{L^\infty(X,\mm)} \quad \forall t>0, \quad \text{ everywhere on } X\quad.
\end{equation}
\\Let us also recall that since $\RCD^*(K,N)$-spaces are locally doubling \& Poincar\'e, then as showed by Sturm \cite[Theorem 3.5]{Sturm96}, the heat flow satisfy the following Harnack inequality: let $Y\subset \subset X$ be a compact subset of $X$, then there exists a constant $C_H=C_H(Y)$ such that for all balls $B_{2r}(x) \subset Y$, all $t\geq 4r^2$ and all $f\in \dom(\Ch)$ with $f\geq 0 \, \mm$-a.e. on $X$ and $f=0 \, \mm$-a.e. on $X\setminus Y$  it holds
\begin{equation}\label{eq:ParHarnack}
\sup_{(s,y) \in Q^-} \H_s f (y) \leq C_H \cdot \inf_{(s,y)\in Q^+} \H_s f(y),
\end{equation}
where $Q^-:=]t-3r^2, t-2r^2[ \times B_r(x)$ and $Q^+:=]t-r^2,t[\times B_r(x)$. We wrote $\sup$ and $\inf$ instead of ess sup and ess inf because in this setting the evolved functions $\H_s f$, $s>0$, have continuous representatives \cite[Proposition 3.1]{Sturm96} given by the formula
\begin{equation}
\H_t f(x)= \int_X H_t(x,y) f(y) \, \d\mm(y)
\end{equation}
where $H_t(x,y)\geq 0$ is the so called heat kernel; recall also that $H_t(\cdot,\cdot)$ is jointly continuous on $X\times X$, symmetric and bounded for $t>0$ see \cite[Section 4]{Sturm96}. Since the flow commutes with its generator we also have that $\Delta (\H_t f)=\H_t(\Delta f)$ and in particular $\Delta (\H_t f)\in W^{1,2}(X,\sfd,\mm)$. Thanks to the $L^\infty$-to-Lipschitz regularization proved in $\RCD(K,\infty)$-spaces in \cite[Theorem 6.8]{AGS11b}, see also \cite{AGMR2012} for the generalizations to  sigma finite reference measures, it follows in particular that $H_t(\cdot,\cdot)$ is  Lipschitz on each variable. 


By using directly the $\RCD^*(K,N)$ condition one gets sharper information. For instance, in their  recent paper \cite[Theorem 4.3]{EKS2013}, Erbar-Kuwada-Sturm proved the dimensional Bakry-Ledoux $L^2$-gradient-Laplacian estimate \cite{BakryLedoux}: if $(X,\sfd,\mm)$ is a $\RCD^*(K,N)$-space, then for every $f \in \dom(\Ch)$ and every $t>0$, one has
\begin{equation}\label{eq:BakryLedoux}
\Gamma(\H_t f)+\frac{4Kt^2}{N(e^{2Kt}-1)} |\Delta \H_t f|^2 \leq e^{-2Kt} \H_t\left( \Gamma(f)\right) \quad \mm\text{-a.e.} \; .
\end{equation}
As a consequence, they showed (see Proposition 4.4) under the same assumption on $X$ that if  $\Gamma(f)\in L^\infty(X,\mm)$ then $\H_t f$ is Lipschitz and $\H_t (\Gamma f)$,$\Delta \H_t f$ have continuous representatives satisfying \eqref{eq:BakryLedoux} everywhere in $X$. In particular, thanks to the above discussion, this is true for the heat kernel $H_t(\cdot,\cdot)$.  In the sequel, if this is the case, we will always tacitly assume we are dealing with the continuous representatives. Finally let us mention that the classical Li-Yau \cite{LY86} estimates on the heat flow hold on $\RCD^*(K,N)$-spaces as well, see \cite{GaMo}.

\section{Sharp estimates  for heat flow-regularization of distance function and applications}\label{S:EstHeatApprox}
Inspired by \cite{CN}, in this section we regularize the distance function via the heat flow obtaining  sharp estimates.  We are going to follow quite closely their scheme of arguments, but the proofs of any individual lemma may sometimes differ in order to generalize the statements to the non smooth setting of $\RCD^*(K,N)$-spaces. 
\\

Throughout the section $(X,\sfd,\mm)$ is an  $\RCD^*(K,N)$-space for some $K \in \R$ and $N \in (1,+\infty)$ and $p,q\in X$ are points in $X$ satisfying $\sfd_{p,q}:=\sfd(p,q)\leq 1$ (of course, by applying the estimates recursively, one can also consider points further apart). Often we will work with the following functions:
\begin{eqnarray}
\sfd^-(x)&:=&\sfd(p,x), \label{def:d-} \\
 \sfd^+(x)&:=& \sfd(p,q)-\sfd(q,x) \label{def:d+} \\
 e(x)&:=& \sfd(p,x)+\sfd(x,q)-\sfd(p,q)=\sfd^-(x)-\sfd^+(x) \quad, \label{def:exc} 
\end{eqnarray} 
the last one being the so called \emph{excess function}. We start by proving existence of good cut-off functions with quantitative estimates, and  then we establish  a $L^1$-Harnack inequality which will imply an improved integral Abresch-Gromoll type inequality on the excess and its gradient.

\subsection{Existence of good cut-off functions on $\RCD^*(K,N)$-spaces with gradient and laplacian estimates} 
The existence of  good cut-off functions is a key technical ingredient in the theory of GH-limits of Riemannian manifolds with lower Ricci bounds, see for instance \cite{CC96,CC97,CC00a,CC00b, Col96a, Col96b,  Col97, CN}. The existence  of  regular cut-off function (i.e. Lipschitz with $L^\infty$ laplacian, but without quantitative estimates)  in $\RCD^*(K,\infty)$-spaces was proved in  \cite[Lemma 6.7]{AMSLocToGlob}; since for the sequel we need quantitative estimates on the gradient and the laplacian of the cut-off function we give here a construction for the  finite dimensional case.  

\begin{lemma}\label{lem:cutoff}
Let $(X,\sfd,\mm)$ be a $\RCD^*(K,N)$-space for some $K \in \R$ and $N\in (1,+\infty)$. Then for every $x\in X$, $R>0$, $0<r<R$  there exists a Lipschitz function $\psi^r:X \to \R$ satisfying:
\\(i) $0\leq \psi^r \leq 1$ on $X$, $\psi^r\equiv 1$ on $B_r(x)$ and $\supp (\psi^r) \subset B_{2r}(x)$;
\\(ii) $r^2|\Delta \psi^r|+r |D \psi^r| \leq C(K,N,R)$. 
\end{lemma}

\begin{proof}
First of all we make the construction with estimate in case $r=1$, the general case will follow by a rescaling argument. 
\\Fix $x\in X$ and let $\tilde{\psi}$ be the 1-Lipschitz function defined as $\tilde{\psi}\equiv 1$ on $B_1(x)$, $\tilde{\psi}\equiv 0$ on $X\setminus B_2(x)$ and $\tilde{\psi}(y)=2-\sfd(x,y)$  for $y \in B_2(x)\setminus B_1(x)$. Consider the heat flow regularization $\tilde{\psi}_t:=\H_t\tilde{\psi}$ of $\tilde{\psi}$. By the results recalled in Subsection \ref{SS:Heat} we can choose continuous representatives of $\tilde{\psi}_t, |D\tilde{\psi}_t|, |\Delta \tilde{\psi}_t|$ and moreover everywhere on $X$ it holds
\begin{equation}\label{eq:estildepsi}
|D \tilde{\psi}_t|^2+\frac{4Kt^2}{N(e^{2Kt}-1)} |\Delta \tilde{\psi}_t|^2 \leq e^{-2Kt} \H_t\left( |D \tilde{\psi}|^2 \right) \leq e^{-2Kt}.
\end{equation}
It follows that
$$|\tilde{\psi}_t-\tilde{\psi}|(y)\leq \int_0^t |\Delta \tilde{\psi}_s|(y) \, \d s \leq \int_0^t \sqrt{\frac{N(e^{2Ks}-1)}{e^{2Ks} 4Ks^2}} \, \d s = F_{K,N}(t), \quad \forall y \in X,$$
where $F_{K,N}(\cdot):\R^+\to \R^+$ is continuous, converges to $0$ as $t\downarrow 0$ and to $+\infty$ as $t\uparrow +\infty$.  Therefore there exists $t_{N,K}>0$ such that $\tilde{\psi}_{t_{N,K}}(y) \in [3/4,1] $ for every $y \in B_1(x)$ and $\tilde{\psi}_{t_{N,K}}(y) \in [0,1/4]$ for every $y \notin B_2(x)$. We get now the desired cut-off function $\psi$ by composition with a $C^2$-function $f:[0,1]\to [0,1]$ such that $f\equiv 1$ on $[3/4,1]$ and $f\equiv 0$ on $[1/4,0]$; indeed  $\psi:=f \circ \tilde{\psi}_{t_{N,K}}$ is now identically equal to one on $B_1(x)$, vanishes identically on $X\setminus B_2(x)$ and, using \eqref{eq:estildepsi} and Chain Rule, it satisfies the estimate $|D\psi|+|\Delta \psi| \leq C(K,N)$ as desired.

To obtain the general case, let $r\in (0,R)$ and consider the rescaled distance $\sfd_r:=\frac{1}{r}\sfd$ on $X$. Thanks to \eqref{eq:cdinv},  the rescaled space $(X,\sfd_r,\mm)$ satisfies the $\RCD^*(r^2 K, N)$ condition and since $r^2 K\geq \hat{K}(R,K)$ we can construct a cut-off function $\psi^r$ such that $\psi^r\equiv 1$ on $B^{\sfd_r}_1(x)$,  $\psi^r\equiv 0$ on $X \setminus B^{\sfd_r}_2(x)$ and satisfying  $|D^{{\sfd_r}}\psi|+|\Delta^{\sfd_r} \psi| \leq C(K,N,R)$, where the quantities with up script  ${\sfd_r}$ are computed in rescaled metric ${\sfd_r}$. By obvious rescaling properties of the lipschitz constant and of the laplacian we get the thesis for the original metric $\sfd$.     
\end{proof}
 
In the sequel it will be useful to have good cut-off functions on annular regions. More precisely for a closed subset $C\subset X$ and $0<r_0<r_1$, we define the annulus $A_{r_0,r_1}(C):=T_{r_1}(C)\setminus T_{r_0}(C)$, where $T_r(C)$ is the $r$-tubular neighborhood of $C$. Using Lemma \ref{lem:cutoff} and the local doubling property of $\RCD^*(K,N)$-spaces one can follow verbatim the proof of \cite[Lemma 2.6]{CN} (it is essentially a covering argument) and establish the following useful result.

\begin{lemma}\label{lem:cutoffAnn}
Let $(X,\sfd,\mm)$ be a $\RCD^*(K,N)$-space for some $K \in \R$ and $N\in (1,+\infty)$. Then for every closed subset $C\subset X$, for every $R>0$ and $0<r_0<10 r_1\leq R$   there exists a Lipschitz function $\psi:X \to \R$ satisfying:
\\(1) $0\leq \psi \leq 1$ on $X$, $\psi \equiv 1$ on $A_{3r_0, r_1/3}(C)$ and $\supp (\psi) \subset A_{2r_0, r_1/2}(C)$;
\\(2) $r_0^2 |\Delta \psi|+r_0 |D \psi| \leq C(K,N,R)$ on $A_{2r_0,3r_0}(C)$;
\\(3) $r_1^2 |\Delta \psi|+r_1 |D \psi| \leq C(K,N,R)$ on $A_{r_1/3,r_1/2}(C)$. 
\end{lemma}
 
 \subsection{$L^1$-Harnack and improved integral Abresch-Gromoll type inequalities}
We start with an estimate on the heat kernel similar in spirit to the one proved by Li-Yau \cite{LY86} in the smooth setting (for the framework of $\RCD^*(K,N)$-spaces see \cite{GaMo}) and by Sturm \cite{Sturm96} for doubling \& Poincar\'e spaces; since we need a little more general estimate we will give a different proof, generalizing to the non smooth setting  ideas of \cite{CN}. 

\begin{lemma}[Heat Kernel bounds]\label{lem2.3}
Let $(X,\sfd,\mm)$ be an  $\RCD^*(K,N)$-space for some $K \in \R$, $N \in (1,+\infty)$ and let $H_t(x,y)$ be the heat kernel for some $x\in X$. Then for every $R>0$, for all $0<r<R$ and $t\leq R^2$, we have
\begin{enumerate}
\item if $y \in B_{10\sqrt{t}}(x)$, then $\frac{C^{-1}(N,K,R)}{\mm(B_{10\sqrt{t}}(x))} \leq H_t(x,y) \leq   \frac{C(N,K,R)}{\mm(B_{10\sqrt{t}}(x))}$
\item $\int_{X \setminus B_r(x)} H_t(x,y) \, \d\mm(y) \leq C(N,K,R) r^{-2} t$.
\end{enumerate}
\end{lemma}

\begin{proof}
One way to get the first estimate is to directly apply the upper and lower bounds on the fundamental solution of the heat flow obtained by Sturm in \cite[Section 4]{Sturm96}, but we prefer to  give here a more elementary argument \cite{CN} based on the existence of good cut-off functions since we will make use of these estimates for the second claim and later on. 

Thanks to Lemma \ref{lem:cutoff} there exists a cut-off function $\psi^r:X \to [0,1]$ with $\psi^r\equiv 1$ on $B_{10r}(x)$, $\psi^r\equiv 0$ on $X\setminus B_{20r}(x)$ and satisfying the estimates $r|D\psi^r|+r^2|\Delta \psi^r|\leq C(K,N,R)$. Let us consider the heat flow regularization $\psi^r_t(y):=\H_t \psi^r (y)=\int_X H_t(y,z) \, \psi^r \, \d \mm(z)$ of $\psi^r$.  Using the symmetry  of the heat kernel, the bound on $|\Delta \psi^r|$, with an integration by parts ensured by the fact that $\psi^r$ has compact support we estimate
$$|\Delta \psi^r_t|(y)= \left|\int_{X} \Delta_y H_t(y,z) \, \psi^r(z) \, \d \mm(z)\right| = \left|\int_{X} \Delta_z H_t(y,z) \, \psi^r(z) \, \d \mm(z) \right|=\left|\int_{X} H_t(y,z) \, \Delta \psi^r(z) \, \d \mm(z)\right|\leq C(K,N,R) r^{-2}. $$
Therefore
$$\left| \psi^r_t-\psi^r \right|(y)\leq \int_0^t |\Delta \psi^r_s|(y) \, \d s \leq C(K,N,R) r^{-2} t. $$
By choosing $t_r:=\frac{1}{2 C(K,N,R)} r^2$ we obtain
\begin{eqnarray}
&&\int_{B_{20r}(x)} H_{2t_r}(x,z)\, \d\mm(z) \leq \int_{X} H_{2t_r}(x,z)\, \d\mm(z) \leq 1 \quad, \label{eq:pflem2.3a}\\
&&\frac{3}{4}\leq \psi^{r}_{\frac{1}{2}t_r}(x)= \int_{B_{20r}(x)} H_{\frac{1}{2}t_r}(x,z)\, \psi^r(z)\, \d\mm(z) \leq  \int_{B_{20r}(x)} H_{\frac{1}{2}t_r}(x,z) \, \d\mm(z) \quad. \label{eq:pflem2.3b}
\end{eqnarray}
From \eqref{eq:pflem2.3a} we infer that $\inf_{B_{20 r}(x)} H_{2t_r}(x,\cdot) \leq  \mm(B_{20 r}(x))^{-1}$ thus, by  the parabolic Harnack inequality \eqref{eq:ParHarnack} we get
\begin{equation}\label{eq:estSupKernel}
\sup_{B_{20 r}(x)} H_{t_r}(x,\cdot) \leq \frac{C(K,N,R)}{\mm(B_{20 r}(x))}.
\end{equation}
On the other hand, \eqref{eq:pflem2.3b} implies that  $\sup_{B_{20 r}(x)} H_{\frac{1}{2}t_r}(x,\cdot)\geq \frac{3}{4} \mm(B_{20 r}(x))^{-1}$ and again by the parabolic Harnack inequality \eqref{eq:ParHarnack} we obtain
\begin{equation}\label{eq:estInfKernel}
\inf_{B_{20 r}(x)} H_{t_r}(x,\cdot) \geq \frac{1}{\mm(B_{20 r}(x)) C(K,N,R)}.
\end{equation}
Combining \eqref{eq:estSupKernel} and  \eqref{eq:estInfKernel} together with local doubling property of the measure $\mm$ gives claim \emph{(1)}.

In order to prove the second claim let $\phi(y):=1-\psi^r(y)$, where now $\psi^r$ is the cut-off function with $\psi^r\equiv 1$ on $B_{r/2}(x)$, $\psi^r\equiv 0$ on $X\setminus B_{r}(x)$ and satisfying   $r|D\psi^r|+r^2|\Delta \psi^r|\leq C(K,N,R)$.
Denoting with $\phi_t:=\H_t \phi$, the same argument as above gives that
$$\phi_t(x) \leq C(K,N,R) r^{-2} t, $$
which yields 
$$\int_{X\setminus B_r(x)} H_t(x,z) \, \d \mm(z) \leq \int_X  H_t(x,z) \, \phi(z) \, \d \mm(z)=\phi_t(x)\leq C(K,N,R) r^{-2} t \quad,  $$
as desired.
\end{proof}

By the above sharp bounds on the heat kernel, repeating verbatim the proof of \cite[Lemma 2.1 and Remark 2.2]{CN} the following useful $L^1$-Harnack inequalities  hold.

\begin{lemma}[$L^1$-Harnack inequality]\label{lem2.1}
Let $(X,\sfd,\mm)$ be an  $\RCD^*(K,N)$-space for some $K \in \R$, $N \in (1,+\infty)$ and let $0<r<R$. If $u:X \times [0,r^2]\to \R$, $u(x,t)=u_t(x)$,  is a nonnegative  continuous function with compact support for each fixed $t\in \R$ satisfying $(\partial_t - \Delta) u \geq -c_0$ in the weak sense, then 
\begin{equation}\label{eq:MeanValue}
\fint_{B_r(x)} u_0 \leq C(K,N,R) \left[ u_{r^2}(x)+ c_0 r^2 \right].
\end{equation}
More generally the following $L^1$-Harnack inequality holds
\begin{equation}\label{eq:L1Harnack}
\fint_{B_r(x)} u_0 \leq C(K,N,R) \left[\inf_{y\in B_r(x)} u_{r^2}(y)+ c_0 r^2 \right].
\end{equation}
\end{lemma}

\noindent Applying Lemma \ref{lem2.1} to  a function constant in time gives the following classical $L^1$-Harnack estimate which will be used in the proof of the improved integral Abresh-Gromoll inequality.

\begin{corollary}\label{cor2.4}
Let $(X,\sfd,\mm)$ be an  $\RCD^*(K,N)$-space for some $K \in \R$, $N \in (1,+\infty)$.  If $u:X\to \R$ is a nonnegative Borel function with compact support with $u \in \dom(\Delta^\star)$ and satisfying $\Delta^\star u \leq c_0 \mm$ in the sense of measures, then for each $x \in X$ and $0<r\leq R$, we have
\begin{equation}\label{eq:MeanValueConst}
\fint_{B_r(x)} u \leq C(K,N,R) \left[ u(x)+ c_0 r^2 \right] \quad.
\end{equation}
\end{corollary}

Before continuing we remark that the verbatim techniques used to prove Lemma \ref{lem2.1} may be used to prove the following mean value estimate:

\begin{lemma}[Mean Value Inequality]
Let $(X,\sfd,\mm)$ be an  $\RCD^*(K,N)$-space for some $K \in \R$, $N \in (1,+\infty)$ and let $0<r<R$. If $u:X \times [0,r^2]\to \R$, $u(x,t)=u_t(x)$,  is a nonnegative  continuous function with compact support for each fixed $t\in \R$ satisfying $(\partial_t - \Delta) u \geq -c_0$ in the weak sense, then for $r<R$
\begin{equation}
\sup_{B_r(x)}u_{r^2} \leq C(K,N,R) \left[ \fint_{B_R(x)}u_{0}+ c_0 R^2 \right].
\end{equation}
In particular, if $u:X\to \R$ is a nonnegative Borel function with compact support with $u \in \dom(\Delta^\star)$ and satisfying $\Delta^\star u \leq c_0 \mm$ in the sense of measures, then
\begin{equation}\label{eq:MeanValueConst}
\sup_{B_r(x)} u \leq C(K,N,R) \left[ \fint_{B_R(x)}u+ c_0 R^2 \right] \quad.
\end{equation}
\end{lemma}
\begin{remark}
Note that combining the last two results gives rise to the classical harnack inequality for harmonic functions.  Also note that for none of the above results is compact support needed, simply one want sufficient growth conditions so that the convolution with the heat kernel is well defined.\\
\end{remark}

We conclude this subsection with a proof of the improved integral Abresch-Gromoll inequality for the excess function $e_{p,q}(x):=\sfd(p,x)+\sfd(x,q)-\sfd(p,q)\geq 0$ relative to a couple of points $p,q\in X$.  Observe that if $\gamma(\cdot)$ is a minimizing geodesic connecting $p$ and $q$, then $e_{p,q}$ attains its minimum value $0$ all along $\gamma$. Therefore, in case  $(X,\sfd,\mm)$ is a smooth Riemannian manifold with uniform estimates on sectional curvature and injectivity radius, since $e_{p,q}$ would be a smooth function near the interior of $\gamma$, one would expect  for $x \in B_r(\gamma(t))$ the estimate $e(x)\leq Cr^2$. In case of lower Ricci bounds and more generally in $\RCD^*(K,N)$-spaces,  this is a lot to ask for. However, an important estimate by Abresh and Gromoll  \cite{AG90}  (see \cite{GiMo12} for the generalization to the $\RCD^*(K,N)$ setting) states that
$$e(x)\leq C r^{1+\alpha(K,N)}, $$
where $\alpha(K,N)$ is a small constant and $x\in B_r(\gamma(t))$.  The next theorem,  which generalizes a result of \cite{CN} proved for smooth Riemannian manifolds with lower Ricci curvature bounds, is an improvement of this statement: indeed even if we are not able to take $\alpha\equiv 1$, this is in fact the case at \emph{most} points.

\begin{theorem}[Improved Integral Abresh-Gromoll inequality]\label{thm:ImprAbrGrom} 
Let $(X,\sfd,\mm)$ be an  $\RCD^*(K,N)$-space for some $K \in \R$ and $N \in (1,+\infty)$; let $p,q \in X$ with $\sfd_{p,q}:=\sfd(p,q)\leq 1$ and  fix $0<\varepsilon<1$. 
Then there exists  $\bar{r}=\bar{r}(N,K,\varepsilon)\in (0,1]$ so that the following holds: if $x \in A_{\varepsilon \, \sfd_{p,q}, 2 \sfd_{p,q}}(\{p,q\})$ satisfies $e_{p,q}(x)\leq r^2 \sfd_{p,q}$ for some $r \in (0,\bar{r}]$, then
$$\fint_{B_{r \sfd_{p,q}}(x)} e_{p,q}(y) \, \d \mm(y) \leq C(K,N,\varepsilon) r^2 \sfd_{p,q} \quad.$$
\end{theorem}  

\begin{proof}
Let $\psi$ be the cut-off function given by Lemma \ref{lem:cutoffAnn} relative to $C:=\{p,q\}$ with  $\psi\equiv 1$ on $A_{\varepsilon \, \sfd_{p,q}, 2 \sfd_{p,q}}(\{p,q\})$, $\psi\equiv 0$ on  $X\setminus A_{\varepsilon \, \sfd_{p,q}/2 , 4 \sfd_{p,q}}(\{p,q\})$, and satisfying $ \varepsilon \, \sfd_{p,q} |D \psi|+  \varepsilon^2 \sfd_{p,q}^2|\Delta \psi| \leq C(K,N)$. Setting $\bar{e}:= \psi e_{p,q}$, using the Laplacian comparison estimate of Theorem \ref{thm:LapComp}, we get that $\bar{e}\in \dom(\Delta^\star)$ and
$$\Delta^\star \bar{e}= \left(\Delta \psi\; e_{p,q} \right) \mm+ \left( 2 \Gamma(\psi,e_{p,q}) \right) \mm + \psi \; \Delta^\star e_{p,q} \leq \frac{C(K,N,\varepsilon)}{\sfd_{p,q}} \mm \quad\text{as measures}.$$
The claim follows then by applying Corollary \ref{cor2.4}.
\end{proof}
\noindent
Clearly, Theorem \ref{thm:ImprAbrGrom} implies the standard Abresh-Gromoll inequality:

\begin{corollary}[Classical Abresh-Gromoll inequality]\label{cor:AbrGrom}
Let $(X,\sfd,\mm)$ be an  $\RCD^*(K,N)$-space for some $K \in \R$ and $N \in (1,+\infty)$; let $p,q \in X$ with $\sfd_{p,q}:=\sfd(p,q)\leq 1$ and  fix $0<\varepsilon<1$.
Then there exists  $\bar{r}=\bar{r}(N,K,\varepsilon)\in (0,1]$ so that the following holds: if $x \in A_{\varepsilon \, \sfd_{p,q}, 2 \sfd_{p,q}}(\{p,q\})$ satisfies $e_{p,q}(x)\leq r^2 \sfd_{p,q}$ for some  $r \in (0,\bar{r}]$, then there exists $\alpha(N)\in (0,1)$ such that
$$ e_{p,q}(y)\leq  C(K,N,\varepsilon)\,  r^{1+\alpha(N)} \, \sfd_{p,q}\; , \quad \forall y \in B_{r \sfd_{p,q}}(x)\quad .$$
\end{corollary}

\begin{proof}
Theorem \ref{thm:ImprAbrGrom} combined with Bishop-Gromov  estimate on volume growth of metric balls \cite{BS2010,CS12,CM16} gives that
for every ball $B_{s \sfd_{p,q}}(z_0)\subset B_{r\sfd_{p,q}}(x)$ it holds
$$\fint_{B_{s \sfd_{p,q}}(z_0)}  e_{p,q} \, \d \mm \leq \frac{\mm(B_{r \sfd_{p,q}}(z_0))}{\mm(B_{s \sfd_{p,q}}(z_0))} \fint_{B_{r \sfd_{p,q}}(z_0)}  e_{p,q} \, \d \mm \leq C(K,N,\varepsilon) \frac{r^{N}}{s^{N}} r^2 \sfd_{p,q}=  C(K,N,\varepsilon) \frac{r^{N+2}}{s^{N}} \, \sfd_{p,q};$$
in particular there exists a point $z \in B_{s \sfd_{p,q}}(z_0)$ such that $e_{p,q}(z)\leq  C \, \frac{r^{N+2}}{s^{N}} \, \sfd_{p,q}$. Since $|D e_{p,q}|\leq 2$ we infer that 
$$e_{p,q}(y) \leq   C \, \left[ \frac{r^{N+2}}{s^{N}}+ 2 s \right]  \sfd_{p,q} \;, \quad \forall y \in B_{s \sfd_{p,q}}(z).$$
Minimizing in $s$ the right hand side and using the arbitrarity of the initial ball $B_{s \sfd_{p,q}}(z_0)\subset B_{r\sfd_{p,q}}(x)$,   we obtain the thesis with $\alpha(N)=\frac{1}{N+1}$.   
\end{proof}
\noindent
Combining   Theorem \ref{thm:ImprAbrGrom}  and Corollary \ref{cor:AbrGrom} with the Laplacian comparison estimate of Theorem \ref{thm:LapComp}, via an integration by parts we get the following crucial gradient estimate on the excess function (which, to our knowledge, is original even in the smooth setting).

\begin{theorem}[Gradient estimate of the excess]\label{thm:GradExc}
Let $(X,\sfd,\mm)$ be an  $\RCD^*(K,N)$-space for some $K \in \R$ and $N \in (1,+\infty)$; let $p,q \in X$ with $\sfd_{p,q}:=\sfd(p,q)\leq 1$ and  fix $0<\varepsilon<1$. \\ If $x \in A_{\varepsilon \, \sfd_{p,q}, 2 \sfd_{p,q}}(\{p,q\})$ satisfies $e_{p,q}(x)\leq r^2 \sfd_{p,q} \leq \bar{r}^2(N,K,\varepsilon) \sfd_{p,q}$ and $B_{2r \sfd_{p,q}}(x)\subset  A_{\varepsilon \, \sfd_{p,q}, 2 \sfd_{p,q}}(\{p,q\})$, then there exists $\alpha(N)\in (0,1)$ such that
$$\fint_{B_{r\, \sfd_{p,q}}(x)} |D \,e_{p,q}|^2\, \d \mm \leq C(K,N,\varepsilon) \, r^{1+\alpha}  \quad.$$
\end{theorem}

\begin{proof} Let $\varphi$ be the cut-off function given by Lemma \ref{lem:cutoff} with $\varphi\equiv 1$ on $B_{r\, \sfd_{p,q}}(x)$,  $\supp \varphi \subset B_{2r \,\sfd_{p,q}}(x)\subset A_{\varepsilon \, \sfd_{p,q}, 2 \sfd_{p,q}}(\{p,q\})$ and satisfying $r \,\sfd_{p,q} |D \varphi|+r^2 \sfd_{p,q}^2 |\Delta \varphi|\leq C(K,N)$. By iterative integrations by parts, recalling that by the Laplacian comparison \ref{thm:LapComp} we have $e_{p,q} \in {\dom}(\Delta^\star, B_{2r \sfd_{p,q}}(x))$ with upper bounds in terms of $\mm$,  we get (we write shortly $e$ in place of $e_{p,q}$)
\begin{eqnarray}
\fint_{B_{r\, \sfd_{p,q}}(x)} |D \,e|^2\, \d \mm &\leq& C(K,N) \fint_{B_{2r\, \sfd_{p,q}}(x)} |D \,e|^2\, \varphi \, \d \mm  \nonumber \\
&=&  C(K,N) \left[- \fint_{B_{2r\, \sfd_{p,q}}(x)} \Gamma(e,\varphi) \,e \, \d \mm - \frac{1}{\mm(B_{2r\, \sfd_{p,q}}(x))} \int_{B_{2r \sfd_{p,q}}(x)}  e \, \varphi \,  \d(\Delta^\star e) \right] \nonumber \\
&=& C(K,N) \Big[ -\frac{1}{2} \fint_{B_{2r\, \sfd_{p,q}}(x)} \Gamma(e,\varphi) \,e  \, \d\mm +  \frac{1}{2} \fint_{B_{2r\, \sfd_{p,q}}(x)} \Delta \varphi \,e^2  \, \d\mm + \frac{1}{2} \fint_{B_{2r\, \sfd_{p,q}}(x)} \Gamma(\varphi, e) e  \, \d\mm \nonumber \\
&& \quad \quad \quad \quad - \frac{1}{\mm(B_{2r\, \sfd_{p,q}}(x))} \int_{B_{2r \sfd_{p,q}}(x)}  e \, \varphi \,  \d(\Delta^\star e) \Big] \nonumber\\
&=& C(K,N) \Big[ \frac{1}{2} \fint_{B_{2r\, \sfd_{p,q}}(x)} \Delta \varphi \,e^2  \, \d\mm  +  \frac{1}{\mm(B_{2r\, \sfd_{p,q}}(x))} \int_{B_{2r\, \sfd_{p,q}}(x)} \big((\sup_{B_{2r\, \sfd_{p,q}}(x)} e)-e\big) \, \varphi \, \d(\Delta^\star e) \nonumber\\
&&\quad \quad \quad \quad  -   \frac{1}{\mm(B_{2r\, \sfd_{p,q}}(x))} \big(\sup_{B_{2r\, \sfd_{p,q}}(x)} e\big)     \int_{B_{2r\, \sfd_{p,q}}(x)} \varphi \, \d(\Delta^\star e)  \Big] \nonumber\\
&\leq& C(K,N)     \Big[ \big(\sup_{B_{2r\, \sfd_{p,q}}(x)} e\big) \fint_{B_{2r\, \sfd_{p,q}}(x)} |\Delta \varphi| e \, \d\mm +  \frac{1}{\mm(B_{2r\, \sfd_{p,q}}(x))} \int_{B_{2r\, \sfd_{p,q}}(x)} \big((\sup_{B_{2r\, \sfd_{p,q}}(x)} e)-e\big) \, \varphi \, \d(\Delta^\star e) \nonumber\\
&&\quad \quad \quad \quad  +  \big(\sup_{B_{2r\, \sfd_{p,q}}(x)} e \big)  \fint_{B_{2r\, \sfd_{p,q}}(x)}  e \, |\Delta \varphi |\, \d\mm  \Big] \nonumber\\
&\leq& C(K,N,\vare)     \left[ r^{1+\alpha}\, \sfd_{p,q} (r \sfd_{p,q})^{-2} r^2 \sfd_{p,q}+   r^{1+\alpha}  + r^{1+\alpha}\, \sfd_{p,q}   r^2 \sfd_{p,q} (r \sfd_{p,q})^{-2} \right] \nonumber \\
&\leq&  C(K,N,\vare)  \, r^{1+\alpha}\quad,  
\end{eqnarray}
where in the second to last estimate we used Theorem \ref{thm:ImprAbrGrom}, Corollary \ref{cor:AbrGrom} and the Laplacian comparison estimate \ref{thm:LapComp}.
\end{proof}

\subsection{Estimates on the Heat-flow regularization of the distance function}\label{SS:estHeatReg}
The goal of the present subsection is to prove the  first order estimates for heat-flow regularization of the distance function that will be used later.
Throughout  the subsection $(X,\sfd,\mm)$ is  an  $\RCD^*(K/r_{1}^{2},N)$-space for some $K \in \R$, $N \in (1,+\infty)$ and $r_{1}\geq 1$. We fix  points $\bar{x}, p,q\in X$ such that $\sfd(p,\bar{x}), \sfd(q, \bar{x})\in [r_{1}, 2 r_{1}]$.
Let $\psi$ be the cutoff function given by Lemma \ref{lem:cutoff} (note that in the last estimate we have $C(K,N)$ instead of $C(K,N,R)$ since we are assuming the space to be $\RCD^{*}(K/r_{1}^{2},N)$ instead of  $\RCD^{*}(K,N)$, the proof of   Lemma \ref{lem:cutoff} gives the claim) such that 
$$\psi\equiv 1 \; \text{on $B_{r_{1}/4}(\bar{x})$, \; } \;   \psi\equiv 0 \text{ on $X\setminus B_{r_{1}/2}(\bar{x})$  and } r_{1}^2|\Delta \psi^r|+r_{1} |D \psi^r| \leq C(K,N). $$
Throughout the subsection we will deal with the heat flow regularizations of the distance and excess functions defined in \eqref{def:d-}, \eqref{def:d+} and \eqref{def:exc}:
\begin{equation}\label{def:htet}
h_t^-:= \H_t (\psi\, \sfd^-), \quad  h_t^+:= \H_t (\psi \, \sfd^+) \; \text{ and } e_t:= \H_t(\psi \, e).
\end{equation} 
In particular $h_0^{\pm}=\sfd^{\pm}, e_0=e$ on $B_{r_{1}/4}(\bar{x})$ and by uniqueness of the heat flow $e_t=h^-_t-h^+_t$.  Observe also that, since of course $h_0^{\pm}$ and $e_0$ are Lipschitz, the results recalled in Section \ref{SS:Heat} imply that $\Gamma(h^{\pm}_t), \Gamma(e_t), \Delta h^{\pm}_t, \Delta e_t$ have continuous representatives. We start with the following easy consequence of the Laplacian comparison Theorem \ref{thm:LapComp}.

\begin{lemma}\label{lem:Laphtet}
Let $h^{\pm}_t, e^{\pm}_t$ be defined in  \eqref{def:htet}.  Then there exists $C(K,N)>0$ such that
\begin{equation} \label{eq:Laphtet}
\Delta h^-_t,  - \Delta h^+_t,  \Delta e_t \leq \frac{C(K,N)}{r_{1}}. 
\end{equation}
\end{lemma}

\begin{proof}
We show the claim for $e_t$, the proof of the others is completely analogous. First of all,  by the Laplacian comparison Theorem  \ref{thm:LapComp} we have  $e_0\in \dom(\Delta^\star)$ and
$$\Delta^\star e_0= \left(\Delta \psi\; e \right) \mm+ \left( 2 \Gamma(\psi,e ) \right) \mm + \psi \; \Delta^\star e \leq \frac{C(K,N)}{r_{1}} \mm \quad\text{as measures}.$$
Notice that, since $e_0$ has compact support we can test the above inequality of compactly supported measures on functions that may not have compact support.  Then recalling the symmetry of the heat kernel and that by definition $e_t(x)=\int_{B_{r_{1}/2}(\bar{x})} H_t(x,y) \, \psi(y) \, e(y) \, \d\mm(y)$, we get via an integration by parts
\begin{eqnarray}
\Delta e_t (x) &=& \int_{B_{r_{1}/2}(\bar{x})} \Delta_x H_t(x,y) \, \psi(y) \, e(y) \, \d\mm(y) =  \int_{B_{r_{1}/2}(\bar{x})} \Delta_y H_t(x,y) \, \psi(y) \, e(y) \, \d\mm(y) \nonumber \\
&=&    \int_{B_{r_{1}/2}(\bar{x})} H_t(x,y) \,  \d (\Delta^\star_y  e_0 (y))  \leq    \frac{C(K,N)}{r_{1}}   \int_{B_{r_{1}/2}(\bar{x})} H_t(x,y) \,  \d \mm(y) \leq 
  \frac{C(K,N)}{r_{1}} \quad.  \nonumber
 \end{eqnarray}
\end{proof}

\begin{lemma}\label{lem:EstHtReg}
There exists $C=C(K,N)$ such that the following hold
\\(i) For all $x \in B_{r_{1}/4}(\bar{x})$ one has
\begin{equation}\label{eq:LinfEstet}
e_{t}(x) \leq  e(x) +  \frac{C(K,N)} {r_{1}} t.
\end{equation}
(ii) For all $x \in B_{r_{1}/4}(\bar{x})$ one has
\begin{equation}\label{eq:LinfEstht}
|h^{\pm}_{t}-\sfd^{\pm}|(x) \leq e(x)+ \frac{C(K,N)}{r_1} t.
\end{equation}
(iii) For all $x \in X$ one has
\begin{equation}
|D h^{\pm}_t|(x)+ t  |\Delta h^{\pm}_t |^2(x) \leq   1+ C(K,N) \frac{ t}{r_{1}^{2}}.
\end{equation}
\\(iv) For every $r_{0} \leq \frac{r_{1}}{8}$ and $t\in [1,r_{0}]$ it holds
\begin{equation}\label{eq:L2Esthd}
\int_{B_{r_{0}}(\bar{x})} |D(h^{\pm}_{t}-\sfd^{\pm})|^{2} \, \d\mm \leq C(K,N,r_{0}) \, \left[ \sup_{x \in B_{2 r_{0}}(\bar{x})} e(x) +  \frac{1}{r_1} \right].
\end{equation}

\end{lemma}

\begin{proof}
(i). From Lemma \ref{lem:Laphtet} we know that $\Delta e_t$ is continuous and satisfies the estimate \eqref{eq:Laphtet}. Since by definition $e_t$ solves the heat equation, we get
$$e_t(x)=e_0(x)+\int_0^t \Delta e_s(x) \, \d s \leq e(x) + \frac{C(K,N)}{r_{1}} t,   \quad \forall x \in B_{r_{1}/4}(\bar{x}).\, $$
(ii). To get the second claim observe that exactly as above, using Lemma \ref{lem:Laphtet}, we get  
$$h^-_{t}(x)\leq \sfd^-(x) +  \frac{C(K,N)}{r_1} t  \quad \text{and} \quad \sfd^+(x) -  \frac{C(K,N)}{r_1} t  \leq  h^+_{t}(x) \quad,   $$
or,  in other words,
$$h^-_{t}(x)- \sfd^-(x)\leq  \frac{C(K,N)}{r_1} t    \quad \text{and} \quad   -(h^+_{t}(x)-\sfd^+(x)) \leq  \frac{C(K,N)}{r_1} t. $$
The reverse inequalities follow from the first claim \eqref{eq:LinfEstet} combined with the identity
$$h^-_{t}(x)- \sfd^-(x)= h^+_{t}(x)- \sfd^+(x)+e_{t}(x)- e(x)\quad.$$
(iii). 
First of all observe that by the Bakry-Ledoux estimate \eqref{eq:BakryLedoux} we have  
\begin{equation}\label{eq:BL}
|D h^{\pm}_t|+\frac{4Kt^2}{N  r_{1}^{2} (e^{2Kt/r_{1}^{2}}-1)} |\Delta h^{\pm}_t |^2 \leq e^{-2Kt/r_{1}^{2}} \, \H_t\left( |D h^{\pm}_0|^2\right)  \quad \text{pointwise.}
\end{equation}
Moreover, the very definition of $h_0^\pm$ yields 
$$|D h_0^{\pm}| \leq |D\psi| \, \sfd^{\pm}+ \psi \, |D \sfd^{\pm}| \leq C(K,N), \quad |Dh_0^{\pm}|\equiv 1 \;\text{ on } B_{r_{1}/4}(\bar{x}), \quad  |Dh_0^{\pm}| \equiv 0 \;\text{ outside } B_{r_{1}/2}(\bar{x}).$$
Therefore we can estimate
\begin{eqnarray}
\H_t\left( |D h^{\pm}_0|^2\right)(x)&=& \int_{B_{\frac{r_{1}}{2}}(\bar{x})} H_t(x,y) |D h_0^{\pm}|^2(y) \, \d \mm(y) \leq \int_{B_{\frac{r_{1}}{4}}(\bar{x})} H_t(x,y) \, \d \mm(y) + C  \int_{B_{\frac{r_{1}}{2}}\setminus B_{\frac{r_{1}}{4}}(\bar{x})} H_t(x,y) \, \d \mm(y) \nonumber \\
& \leq& 1+ C(K,N) \frac{ t}{r_{1}^{2}} \quad, \label{eq:estHtDh0}
\end{eqnarray}
where in last inequality  we  used the second part of  Lemma \ref{lem2.3} (note that we can replace $C(K,N,R)$ by $C(K,N)$ since we are assuming $\RCD^{*}(K/r_{1}^{2},N)$ instead of $\RCD^{*}(K,N)$).  The thesis follows by  the combination of   \eqref{eq:BL} and \eqref{eq:estHtDh0}. 
\\(iv).  Let $\phi:X\to [0,1]$ be a $1/r_{0}$-Lipschitz cut-off function with $\phi\equiv 1$ on  $B_{r_{0}}(\bar{x})$ and $\phi\equiv 0$ on  $X\setminus B_{2r_{0}}(\bar{x})\supset X\setminus B_{r_{1}/4}(\bar{x})$. By using the previous items (ii) and (iii) together with the Laplacian comparison Theorem  \ref{thm:LapComp}, for $t\in [1,r_{0}]$ we get
 \begin{align}
\int_{B_{r_{0}}(\bar{x})} |D (h^{-}_{t} -\sfd^{-})|^2 \,	\d \mm& \leq \int_{B_{2r_{0}}(\bar{x})}  |D (h^{-}_{t} -\sfd^{-})|^2 \, \phi \, \d \mm  \nonumber  \\
&= -  \int_{B_{2r_{0}}(\bar{x})}   (h^{-}_{t} -\sfd^{-}) \, \Gamma \left(\phi, h^{-}_{t} -\sfd^{-}\right)  \, \d \mm - \int_{B_{2r_{0}}(\bar{x})}  (h^{-}_{t} -\sfd^{-}) \, \phi  \, \d \left({\Delta^{\star}}(h^{-}_{t} -\sfd^{-}) \right)   \nonumber  \\
&\leq C(K,N,r_{0})\, \|h^{-}_{t}-\sfd^{-} \|_{L^{\infty}(B_{2r_{0}}(\bar{x}))}   +  \int_{B_{2r_{0}}(\bar{x})}   \|h^{-}_{t}-\sfd^{-} \|_{L^{\infty}(B_{2r_{0}}(\bar{x}))}\,  \phi  \, \d \left( {\Delta^{\star}} ( h^{-}_{t}- \sfd^{-} ) \right)              \nonumber  \\
& +   \int_{B_{2r_{0}}(\bar{x})}  \big( \|h^{-}_{t}-\sfd^{-} \|_{L^{\infty}(B_{2r_{0}}(\bar{x}))} + (h^{-}_{t}-\sfd^{-}) \big) \, \phi  \, \d \left( {\Delta^{\star}}  (\sfd^{-}-h^{-}_{t}) \right)      \nonumber  \\
&\leq C(K,N,r_{0})\, \|h^{-}_{t}-\sfd^{-} \|_{L^{\infty}(B_{2r_{0}}(\bar{x}))}        -    \|h^{-}_{t}-\sfd^{-} \|_{L^{\infty}(B_{2r_{0}}(\bar{x}))} \int_{B_{2r_{0}}(\bar{x})} \,  \Gamma\left(\phi,  ( h^{-}_{t}- \sfd^{-} ) \right) \, \d \mm    \nonumber  \\
& +   2\|h^{-}_{t}-\sfd^{-} \|_{L^{\infty}(B_{2r_{0}}(\bar{x}))}   \int_{B_{2r_{0}}(\bar{x})}  \, \phi  \, \d \left( {\Delta^{\star}} ( \sfd^{-}-h^{-}_{t}) \right)           \nonumber  \\
&\leq   C(K,N,r_{0})\, \|h^{-}_{t}-\sfd^{-} \|_{L^{\infty}(B_{2r_{0}}(\bar{x}))} \leq  C(K,N,r_{0}) \, \left[ \sup_{x \in B_{2 r_{0}}(\bar{x})} e(x) +  \frac{1}{r_1} \right].   \nonumber
\end{align}

\end{proof}

\section{Construction of Gromov-Hausdorff approximations with  estimates}\label{s:Gromov_Hausdorff}
Thanks to Theorem \ref{thm:euclideantangents} we already know that  at $\mm$-almost every $x \in X$ there exists $k \in \N$, $1 \le k \le N$, such that 
 $$
  (\R^k,\sfd_{E},\LL_k,0) \in \Tan(X,\sfd,\mm,x) \quad.  
 $$
The goal of the present section is to prove an explicit Gromov-Hausdorff approximation with estimates at such points by using the results of Section \ref{S:EstHeatApprox}. More precisely we prove the following.

 \begin{theorem}\label{lem:ConstrU}
Let $(X,\sfd,\mm)$ be an $\RCD^*(K,N)$-space for some  $K \in \R$, $N \in (1,\infty)$ and let $\bar{x} \in X$ be such that  $(\R^k,\sfd_{E},\LL_k,0) \in \Tan(X,\sfd,\mm,\bar{x})$ for some $k \in \N$. Then, for every  $0<\varepsilon_2<<1$ there exist $\bar{R}=\bar{R}(\vare_2)>>1$ such that for every $\tilde{R}\geq \bar{R}$ there exists $R\geq \tilde{R}$  and there exists $0<r=r(\bar x, \vare_2,R)<< 1$ such that the following holds.

Call  $(X,  \sfd_r, \mm^{\bar{x}}_r, \bar{x})$  the rescaled p.m.m.s., where $\sfd_r(\cdot,\cdot):=r^{-1} \sfd(\cdot, \cdot)$ and  $ \mm^{\bar{x}}_r$ was defined in \eqref{eq:normalization}. There exist points $\{p_i,q_i\}_{i=1,\ldots,k}\subset \partial B_{R^\beta}^{\sfd_r}(\bar{x})$ and $\{p_i+p_j\}_{1\leq i<j\leq k}\in  B_{2R^\beta}^{\sfd_r}(\bar{x})\setminus B_{R^\beta}^{\sfd_r}(\bar{x})$, for some $\beta=\beta(N)>2$, such that \footnote{``$p_i+p_j$'' is just a symbol indicating a point of $X$, no affine structure is assumed}
\begin{eqnarray}
\sum_{i=1}^k \fint _{B^{\sfd_r}_R(\bar{x})} |D e_{p_i,q_i}|^2 \, \d \mm^{\bar{x}}_r &+& \sup_{y \in B^{\sfd_r}_R(\bar{x})} e_{p_i,q_i} (y) \leq \varepsilon_2\quad, \label{eq:estUexc} \\
 \sum_{1\leq i< j\leq k} \fint_{B^{\sfd_r}_R(\bar{x})}  \left|D \left( \frac{\sfd_r^{p_i}+\sfd_r^{p_j}} {\sqrt 2} - \sfd_r^{p_i+p_j}\right) \right|^2 \, \d \mm^{\bar{x}}_r &+& \sup_{B^{\sfd_r}_R(\bar{x})}  \left| \frac{\sfd_r^{p_i}+\sfd_r^{p_j}} {\sqrt 2} - \sfd_r^{p_i+p_j}\right|   \leq \varepsilon_2\quad, \label{eq:estUdidj} 
\end{eqnarray} 
where $\sfd^{p_i}_r(\cdot):=\sfd_r(p_i, \cdot):= r^{-1}\sfd(p_i,\cdot),$ the excess $e_{p_i,q_i}$ is defined by $e_{p_i,q_i}(\cdot):=\sfd^{p_i}_r(\cdot)+\sfd^{q_i}_r(\cdot)-\sfd_r(p_i,q_i)$ and  the slope $|D\cdot|$ is intended to be computed with respect to the rescaled structure $(X, r^{-1} \sfd, \mm^{\bar{x}}_r)$.
 \end{theorem}

  
 \begin{proof} Since by assumption $(\R^k,\sfd_{E},\LL_k,0^k) \in \Tan(X,\sfd,\mm,\bar x)$ then, for every $0<\vare_1\leq 1/10$ there exists $r=r(\vare_1, \bar{x})>0$ such that
\begin{equation}\label{eq:XRkeps1}
\DC \left( \left(X, \sfd_r, \mm^{\bar{x}}_r, \bar{x} \right), \left( \R^k,\sfd_{E},\LL_k,0^k  \right) \right) \leq  \varepsilon_1\quad . 
\end{equation}
In particular we can find $\delta$-quasi isometries from $ B^{\sfd_r}_{1}(\bar{x})\subset X$ to $B^{\sfd_{E}}_{1}(0^{k})\subset \R^{k}$, with $\delta=\delta(\vare_{1})\to 0$ as $\vare_{1}\to 0$.
 Observe that, by the rescaling property \eqref{eq:cdinv}, $(X, \sfd_r, \mm^{\bar{x}}_r)$ is an $\RCD^*(r^2 K,N)$-space. Now let 
 \begin{equation}\label{eq:defpiqi}\nonumber
 p_i, q_i, p_i+p_j\in  B^{\sfd_r}_{1/2}(\bar{x})\subset X \text{ be the points corresponding  to } -\frac{\vec{e}_i}{4}, \frac{\vec{e}_i}{4}, -\frac{\vec{e}_i+\vec{e}_j}{4}  \in  B^{\sfd_{E}}_{\sqrt{2}/4}(0^k)\subset \R^k 
 \end{equation}
 respectively via the $\delta$-quasi isometry ensured by  \eqref{eq:XRkeps1}, for every $1\leq i,j \leq k$, where $(\vec{e}_i)$ is the standard basis of $\R^k$. Let us explicity note that $1/4 \leq \sfd_r(p_i,q_i) \leq 1$, for every $i=1, \ldots, k$.
 
 Consider a minimizing geodesic $\gamma_i$ connecting $p_i$ and $q_i$. Then, combining  the excess estimate Corollary \ref{cor:AbrGrom}  and the excess gradient estimate Theorem \ref{thm:GradExc}, called $\xi_i:=\gamma_i\left(\frac{\sfd_r(p_i,q_i)}{2}\right)\in X$ (so, in  particular, $e_{p_i,q_i}(\xi_i)=0$), we have 
 \begin{equation}\label{eq:epiqi}\nonumber
\left( \sup_{y \in B^{\sfd_r}_{2\vare}(\xi_i)} e_{p_i,q_i} (y) \right)+  \fint_{B^{\sfd_r}_{2\vare}(\xi_i)} |D \, e_{p_i,q_i}|^2 (y) \, \d \mm^{\bar x}_r(y) \leq C(K,N)\, \vare^{1+\alpha(N)}, \quad \forall 0<\vare\leq \bar{\vare}(K,N). 
 \end{equation} 
 Moreover, for $\vare_1=\vare_1(K,N,\vare)>0$ small enough in \eqref{eq:XRkeps1} we also have $\xi_i \in B^{\sfd_r}_{\vare}(\bar{x})$, so $B^{\sfd_r}_{\vare}(\bar{x})\subset B^{\sfd_r}_{2\vare}(\xi_i)$ and, by the doubling property of $\mm^{\bar x}_r$, we infer that
 \begin{equation}\label{eq:Depiqi}
\left( \sup_{y \in B^{\sfd_r}_{\vare}(\bar{x})} e_{p_i,q_i} (y) \right)+\fint_{B^{\sfd_r}_{\vare}(\bar{x})} |D \, e_{p_i,q_i}|^2 (y) \, \d \mm^{\bar x}_r(y) \leq C(K,N)\, \vare^{1+\alpha(N)}, \quad \forall 0<\vare\leq \bar{\vare}(K,N). 
 \end{equation}

Now we do a second rescaling of the metric, namely we consider the new metric $\sfd_{R^{-\beta} r}:=R^\beta \sfd_r=\frac{R^\beta}{r} \sfd$, for $\beta\geq \beta(N)>\max\{ 1+\frac{1}{\alpha(N)}, 2\}$ and observe that by obvious rescaling properties, having chosen $\vare=R^{-\beta+1}$, estimate \eqref{eq:Depiqi} implies
 \begin{equation}\label{eq:DepiqiR}
 \sup_{B^{\sfd_{R^{-\beta} r}}_{R}(\bar{x})} e_{p_i,q_i} \leq C(K,N)\frac{1}{R^{\alpha(\beta-1)-1}} \quad \text{ and } \quad   \fint_{B^{\sfd_{R^{-\beta} r}}_{R}(\bar{x})} |D \, e_{p_i,q_i}|^2 (y) \, \d \mm^{\bar x}_{R^{-\beta}r}(y) \leq C(K,N)\, \frac{1}{R^{(\beta-1)(1+\alpha)}}.
  \end{equation} 
The proof of \eqref{eq:estUexc} is therefore complete once we choose $R\geq \bar{R}(K,N,\vare_2)>>1$.
\\

Now we prove \eqref{eq:estUdidj}. Again, since by assumption  $(\R^k,\sfd_{E},\LL_k,0^k) \in \Tan(X,\sfd,\mm,\bar x)$,  for every $0<\theta \leq 1/10$ there exists $R\geq \bar{R}(K,N,\vare_2)>>1$ such that
\begin{equation}\label{eq:XRketa}
\DC \left( \left(X, \sfd_{R^{-\beta} r}, \mm^{\bar{x}}_{R^{-\beta} r}, \bar{x} \right), \left( \R^k,\sfd_{E},\LL_k,0^k  \right) \right) \leq  \theta \quad . 
\end{equation}
In particular, for some $\eta=\eta(\theta)\to 0$ as $\theta\to 0$, we have that $\sfd_{R^{-\beta} r}(p_i, \cdot), \sfd_{R^{-\beta} r}(p_i+p_j, \cdot)$ are $\eta$-close in $L^\infty(B_{R}^{\sfd_{R^{-\beta} r}}(\bar x))$ norm (via composition with a GH quasi-isometry) to  $\sfd_{E}( -R^\beta \vec{e}_i/4 , \cdot), \sfd_{E}(-R^\beta(\vec{e}_i+\vec{e}_j)/4 , \cdot)$ respectively. Moreover, in euclidean metric, we have that 
\begin{equation}\label{eq:eucleta}
\left| \frac{\sfd_{E}( - R^\beta \vec{e}_i/4 , \cdot)+ \sfd_E( - R^\beta \vec{e}_j/4 , \cdot)}{\sqrt 2} -\sfd_{E}(-R^{\beta}(\vec{e}_i+\vec{e}_j)/4 , \cdot)\right| \leq \eta \quad \text{on } B_R(0^k), \quad \forall 1\leq i <j \leq k,
\end{equation}
for $R\geq  R(\eta)$ large enough; an easy way to see it is to observe that $$\left| \frac{\sfd_E( -\vec{e}_i/4, \cdot)+ \sfd_E( -\vec{e}_j/4 , \cdot)}{\sqrt 2} -\sfd_{E}(-(\vec{e}_i+\vec{e}_j)/4 , \cdot)\right| \leq \bar{C} \vare^2 \quad \text{on } \; B_{\vare}(0^k) $$ 
 by a second order Taylor expansion at $0^k$, then rescale by $R^\beta$, choose $\vare^{-1}:=R^{\beta-1}$ and $R\geq \left(\frac{\bar{C}}{\eta}\right)^{\frac{1}{\beta-2}}$.
Combining \eqref{eq:XRketa}   and \eqref{eq:eucleta} we get 
\begin{equation} \label{eq:4eta}
\left| \frac{\sfd_{R^{-\beta} r}( p_i, \cdot)+ \sfd_{R^{-\beta} r}( p_j, \cdot)}{\sqrt 2} -\sfd_{R^{-\beta} r}( p_i+p_j, \cdot)\right| \leq 4 \eta \quad \text{on } B^{\sfd_{R^{-\beta} r}}_R(\bar x), \quad \forall 1\leq i <j \leq k.
\end{equation}
In order to conclude the proof we next show that 
\begin{equation}\label{eq:ClaimOrt}
{\mathfrak {I}}:= \fint_{B^{\sfd_{R^{-\beta} r}}_R(\bar{x})} \left|D \left( \frac{\sfd_{R^{-\beta} r}^{p_i}+\sfd_{R^{-\beta} r}^{p_j}} {\sqrt 2} - \sfd_{R^{-\beta} r}^{p_i+p_j}\right) \right|^2 \, \d \mm_{R^{-\beta} r}^{\bar{x}} \leq C(K,N)\; \frac{\eta}{R}    \quad ,
\end{equation}
which, together with  \eqref{eq:4eta}, will give \eqref{eq:estUdidj}  by choosing $\bar{R}=\bar{R}(\vare_2)$ large enough. \\
To this aim let $\varphi$ be a $1/R$-Lipschitz cut-off function with $0\leq \varphi \leq 1$, $\varphi\equiv 1$ on $B^{\sfd_{R^{-\beta} r}}_R(\bar{x})$ and $\supp \varphi \subset B^{\sfd_{R^{-\beta} r}}_{2R}(\bar{x})$; in order to simplify the notation let us denote 
\begin{equation}\label{eq:defuijvij}
u^{ij}:= \frac{\sfd_{R^{-\beta} r}^{p_i}+\sfd_{R^{-\beta} r}^{p_j}} {\sqrt 2} \quad  \text{and} \quad  v^{ij}:=\sfd_{R^{-\beta} r}^{p_i+p_j}\quad .
\end{equation}
 With an integration by parts together with \eqref{eq:4eta} and the Laplacian comparison Theorem \ref{thm:LapComp} (which in particular gives that $\Delta^\star v^{ij} \llcorner B^{\sfd_{R^{-\beta} r}}_{2R}(\bar{x}),\Delta^\star u^{ij}\llcorner B^{\sfd_{R^{-\beta} r}}_{2R}(\bar{x}) \leq  C(K,N) \frac{1}{R^{\beta-1}}\, \mm\llcorner B^{\sfd_{R^{-\beta} r}}_{2R}(\bar{x})$) yields
\begin{eqnarray}
{\mathfrak {I}} &\leq & C(K,N) \fint_{B^{\sfd_{R^{-\beta} r}}_{2R}(\bar{x})} \left|D (u^{ij}-v^{ij}) \right|^2 \, \varphi \, \d \mm_{R^{-\beta} r}^{\bar{x}} \nonumber \\
&=& - C(K,N)\Big[ \fint_{B^{\sfd_{R^{-\beta} r}}_{2R}(\bar{x})} \Gamma \Big(u^{ij}, \big(\sup_{{B^{\sfd_{R^{-\beta} r}}_{2R}(\bar{x})} } |u^{ij}-v^{ij}| \big)-(u^{ij}-v^{ij}) \Big) \, \varphi \, \d \mm_{R^{-\beta} r}^{\bar{x}}  \nonumber \\
&& \quad \quad \quad \quad  + \fint_{B^{\sfd_{R^{-\beta} r}}_{2R}(\bar{x})} \Gamma\Big(v^{ij}, \big(\sup_{{B^{\sfd_{R^{-\beta} r}}_{2R}(\bar{x})} } |u^{ij}-v^{ij}| \big)+(u^{ij}-v^{ij}) \Big) \, \varphi \, \d \mm_{R^{-\beta} r}^{\bar{x}}   \Big]\nonumber \\
&\leq& C(K,N) \Big(\sup_{{B^{\sfd_{R^{-\beta} r}}_{2R}(\bar{x})} } |u^{ij}-v^{ij}| \Big) \left[  \frac{1}{R^{\beta-1}}+\frac{1}{R}  \right] \leq C(K,N)\; \eta \;  \left[  \frac{1}{R^{\beta-1}}+\frac{1}{R}  \right] \quad , \nonumber
\end{eqnarray}
which proves our claim \eqref{eq:ClaimOrt}. 
\end{proof}

\section{Almost  splitting via excess}\label{Sec:AlmSplit}
The interest of the almost splitting theorem we prove in this section is that the condition on the existence of an {\it almost line} is replaced by an assumption on the smallness of the excess and its derivative;  this will be convenient in the proof of  the rectifiability thanks to estimates on the Gromov-Hausdorff approximation  proved in Theorem \ref{lem:ConstrU}.   From the technical point of view our strategy is to argue by contradiction and to construct an appropriate replacement for the Busemann function, which is a priori not available since we do not assume the  existence of a long geodesic.  Then in the limit we may rely on the arguments used in Gigli's proof of the Splitting Theorem in the non smooth setting (see \cite{GigliSplitting}-\cite{GigliSplittingSur}) in order to construct our splitting.

\begin{theorem}[Almost splitting via excess]\label{thm:AlmSplit}
Fix $N\in (1, +\infty)$ and $\beta>2$. For every $\vare>0$ there exists a $\delta=\delta(N,\vare)>0$  such that if the following hold

i) $(X,\sfd,\mm)$ is an $\RCD^*(-\delta^{2\beta},N)$-space,

ii) there exist points  $\bar{x}, \{p_i,q_i\}_{i=1,\ldots, k}, \{p_i+p_j\}_{1\leq i< j \leq k}$ of $X$, for some $k\leq N$,  such that \footnote{``$p_i+p_j$'' is just a symbol indicating a point of $X$, no affine structure is assumed} $\sfd(p_i,\bar{x}), \sfd(q_i,\bar{x}),  \sfd(p_i+p_j,\bar{x})\geq \delta^{-\beta}$,
\begin{equation} \nonumber
\sum_{i=1}^{k} \sup_{B_R(\bar{x})} e_{p_{i},q_{i}} + \sum_{i=1}^k \fint _{B_R(\bar{x})} |D e_{p_i,q_i}|^2 \, \d \mm + \sum_{1\leq i< j\leq k} \fint_{B_R(\bar{x})}  \left|D \left( \frac{\sfd^{p_i}+\sfd^{p_j}} {\sqrt 2} - \sfd^{p_i+p_j}\right) \right|^2 \, \d \mm   \leq \delta  \quad,  
\end{equation} 
for every $R\in [1,\delta^{-1}]$.
 Then there exists a p.m.m.s. $(Y,\sfd_Y, \mm_Y, \bar{y})$ such that
 $$\DC \left( \left(X, \sfd, \mm, \bar{x} \right), \left( \R^k\times Y,\sfd_{\R^k\times Y},\mm_{\R^k\times Y}, (0^k, \bar{y})  \right) \right) \leq  \vare \quad . 
 $$
 More precisely
 
 1) if $N-k<1$ then $Y=\{\bar{y}\}$ is a singleton,  and if $N-k\in [1,+\infty)$ then  $(Y,\sfd_Y, \mm_Y)$ is an $\RCD^*(0,N-k)$-space,
 
 2) there exist maps $v:X\supset B_{\delta^{-1}}(\bar{x})\to Y$ and $u:X\supset B_{\delta^{-1}}(\bar{x}) \to \R^k$ given by $u^i(x)=\sfd(p_i,x)-\sfd(p_i,\bar{x})$ such that the product map
 $$(u,v):X \supset B_{\delta^{-1}}(\bar{x})\to Y\times \R^k \quad \text{is a measured GH $\vare$-quasi isometry on its image}. $$ 
\end{theorem}

\begin{proof}
We argue by contradiction. If the thesis is not true then  for any $n \in \N$ there exists an  $\RCD^*(-n^{-2\beta}, N)$-space $(X_n, \sfd_n, \mm_n)$  and points  $\bar{x}_n, \{p^i_n,q^i_n\}_{i=1,\ldots, k}, \{p^i_n+p^j_n\}_{1\leq i< j \leq k}$ of $X_n$ such that $\sfd(p^i_n,\bar{x}_n), \sfd(q^i_n,\bar{x}_n),  \sfd(p^i_n+p^j_n,\bar{x}_n)\geq n^{\beta}$ with
\begin{equation}\label{eq:assn}
\sum_{i=1}^{k} \sup_{B_R(\bar{x})} e_{p_{i},q_{i}} + \sum_{i=1}^k \fint _{B_R(\bar{x}_n)} |D e_{p^i_n,q^i_n}|^2 \, \d \mm_n +\sum_{1\leq i< j\leq k} \fint_{B_R(\bar{x}_n)}  \left|D \left( \frac{\sfd^{p^i_n}+\sfd^{p^j_n}} {\sqrt 2} - \sfd^{p^i_n+p^j_n}\right) \right|^2 \, \d \mm_n   \leq \frac{1}{n}  \quad,   
\end{equation}
for every  $R\in[1,n]$.
To begin with, by p-mGH compactness Proposition \ref{prop:comp}  combined with the $\RCD^*(K,N)$-Stability Theorem \ref{thm:stab} (recall also that the $\RCD^*(0,N)$ condition is equivalent to the $\RCD(0,N)$ condition), we know that there exists an $\RCD^*(0,N)$-space $(X,\sfd,\mm, \bar{x})$ such that, up to subsequences,  we have
 \begin{equation}\label{eq:XntoX}
 \DC\left( (X_n, \sfd_n, \mm_n, \bar{x}_n), (X, \sfd, \mm, \bar{x})\right) \to  0 \quad.
 \end{equation}
Our goal is to prove that $(X, \sfd, \mm, \bar{x})$ is isomorphic to a product $\left(\R^k\times Y, \sfd_{\R^k\times Y}, \mm_{\R^k\times Y}, (0^k,\bar{y})\right)$ for some $\RCD^*(0,N-k)$-space $(Y,\sfd_Y, \mm_Y)$. The strategy is to use the distance functions together with the excess estimates in order to construct in the limit space $X$ a kind of ``affine'' functions which play an analogous role of the Busemann functions in the proof of the splitting theorem.  To this aim call 
 \begin{equation}\label{def:fngn}
 \begin{split}
&f^i_n: B_n(\bar{x}_n)\to \R, \, f^i_n(\cdot):=\sfd_n(p^i_n, \cdot)-\sfd_n(p^i_n, \bar{x}_n) \quad\text{and} \quad g^i_n: B_n(\bar{x}_n)\to \R, \,g^i_n(\cdot):=\sfd_n(q^i_n, \cdot)-\sfd_n(q^i_n, \bar{x}_n) \quad,   \\
&\quad \quad f^{ij}_n: B_n(\bar{x}_n)\to \R, \, f^{ij}_n(\cdot):=\sfd_n(p^i_n+p^j_n, \cdot)-\sfd_n(p^i_n+p^j_n, \bar{x}_n)   \quad     \forall n \in N. 
\end{split}
\end{equation}
Of course $f^i_n,g^i_n, f^{ij}_n$ are 1-Lipschitz so by Arzel\'a-Ascoli Theorem \ref{Prop:AscArz} we have that there exist 1-Lipschitz functions $f^i,g^i,f^{ij}:X\to \R$ such that
\begin{equation}\label{fngntofg}
f^i_n|_{B_R(\bar{x}_n)}\to f^i|_{B_R(\bar{x})}, g^i_n|_{B_R(\bar{x}_n)}\ \to g^i|_{B_R(\bar{x})} , f^{ij}_n |_{B_R(\bar{x})}\to f^{ij}|_{B_R(\bar{x})} \; \text{ uniformly } \forall R>0 \text{ and }  f^i(\bar{x})=g^i(\bar{x})=f^{ij}(\bar{x})=0 
\end{equation}
 as $n \to \infty$,  for every  $i,j=1,\ldots,k$. 
 
 As it will be clear in a moment, the maps $f^i$ will play the role of the Busemann functions  in proving the isometric splitting of X. To this aim  we  now proceed by successive claims about properties  of the functions $f^i,g^i$ which represent the  cornerstones to apply the arguments by Gigli \cite{GigliSplitting}-\cite{GigliSplittingSur} of the Cheeger-Gromoll splitting Theorem.
 \\
 
CLAIM 1: \emph{$f^i=-g^i$ everywhere on $X$ for every $i=1,\ldots, k$.}
\\From the very definition of the excess we have
\begin{eqnarray}
e_{p^i_n,q^i_n}(\cdot)&=&\sfd_n(p^i_n, \cdot)+\sfd_n(q^i_n, \cdot)-\sfd_n(p^i_n,q^i_n)=f^i_n(\cdot)+g^i_n(\cdot)+\sfd_n(p^i_n,\bar{x}_n)+\sfd_n(q^i_n,\bar{x}_n)-\sfd_n(p^i_n,q^i_n)\nonumber \\
&=& f^i_n(\cdot)+g^i_n(\cdot)+e_{p^i_n,q^i_n}(\bar{x}_n) \quad , \label{eq:efg}
\end{eqnarray}
which gives  in particular that
\begin{equation}\label{eq:DeDfg}
|D e_{p^i_n,q^i_n}| \equiv |D(f^i_n+g^i_n)| \quad \text{on } B_n(\bar{x}_n)\quad. 
\end{equation}
Fix now $R>0$ and observe that, since $f^i_n+g^i_n|_{B_R(\bar{x}_n)}\to f^i+g^i|_{B_R(\bar{x})} $ uniformly  we have by the lowersemicontinuity of the slope Proposition \ref{Prop:LSCSlope}
$$
\int_{B_R(\bar{x})} \left|D\left( f^i+g^i   \right) \right|^2 \d\mm \leq  \liminf_n  \int_{B_{R+1}(\bar{x}_n)} \left|D\left( f^i_n+g^j_n \right) \right|^2 \d\mm_n =0 \quad \text{for every fixed } R\geq 1 \quad, $$
thanks to  \eqref{eq:assn}. This gives $|D(f^i+g^i)|=0$ $\mm$-a.e. and therefore the claim follows from  \eqref{fngntofg} since $f^i$ and $g^i$ are continuous.
\\

CLAIM 2: \emph{ $\Delta^\star f^i = 0$ on $X$ as a measure, i.e. $f^i$ is harmonic, for every $i=1,\ldots,k$.}
\\Fixed any $R>0$, let  $\varphi:X \to [0,1]$ be a $1/R$-Lipschitz cut-off function with $\varphi\equiv 1$ on $B_R(\bar{x})$ and $\varphi\equiv 0$ outside $B_{2R}(\bar{x})$.  From the technical point of view it is convenient here to see the convergence \eqref{eq:XntoX} realized by isometric immersions $\iota_n,\iota$ of the spaces $X_n,X$  into an ambient Polish space $(Z,\sfd_Z)$ as in Definition \ref{def:conv} and define $\varphi:Z \to [0,1]$ be a $1/R$-Lipschitz cut-off function with $\varphi\equiv 1$ on $B_R(\iota(\bar{x}))$  and $\varphi\equiv 0$ outside $B_{2R}(\iota(\bar{x}))$ so that $\varphi$ can be used as  $1/R$-Lipschitz cut-off function also for the spaces $X_n$; but let us not complicate the notation with the isometric inclusions here. 
\\
We first claim that 
\begin{equation}\label{eq:intXGammafivarphi}
\int_X \Gamma(f^i,\varphi)\, \d\mm= \lim_n \int_{X_n} \Gamma(f^i_n,\varphi)\, \d\mm_n .
\end{equation}
To this aim we make use of the estimates on the heat flow approximation of the distance function proved in Subsection \ref{SS:estHeatReg}. Let $\psi_{n}:X_{n}\to [0,1]$ be   cut off functions satisfying 
$$\psi_{n}\equiv 1 \; \text{on $B_{n^{\beta}/4}(\bar{x}_{n})$, \; } \;   \psi\equiv 0 \text{ on $X\setminus B_{n^{\beta}/2}(\bar{x}_{n})$  and } n^{2\beta}|\Delta \psi^r|+n^{\beta} |D \psi^r| \leq C(K,N), $$
and let $\tilde{f}^{i}_{n}:=\H_{1}(\psi_{n} f^{i}_{n})$. From Lemma \ref{lem:EstHtReg} combined with \eqref{eq:assn}, we know that for all fixed $R>0$ and $n\geq \max\big((16R)^{1/\beta}, R\big)$ it holds
\begin{align}
&\int_{B_{2R}(\bar{x}_{n})} |D(\tilde{f}^{i}_{n}-{f}^{i}_{n})|^{2} \, \d\mm \leq C(K,N,R) \, \left[ \sup_{x \in B_{2 R}(\bar{x}_{n})} e_{p^i_n,q^i_n}(x) +  \frac{1}{n^{\beta}} \right] \leq C(K,N,R) \left(\frac{1}{n}+\frac{1}{n^{\beta}}\right)  . \label{eq:Gammatildeff} \\
&\int_{B_{2R}(\bar{x}_{n})}  |\Delta \tilde{f}^{i}_{n}|^{2} \, \d \mm  \leq C(K,N,R). \nonumber
\end{align}
From \cite[Corollary 5.5]{AH-N} (see also \cite[Corollary 6.10]{GMS2013}), the above estimates ensure that  $\tilde{f}^{i}_{n}|_{B_{2R}(\bar{x}_{n})}$ converge strongly in $W^{1,2}$ sense to $f^{i}|_{B_{2R}(\bar{x})}$, and then by using  again \eqref{eq:Gammatildeff} the claimed  \eqref{eq:intXGammafivarphi} follows.
\\Now that \eqref{eq:intXGammafivarphi} is proved,   via integration by parts we get
\begin{equation}\nonumber
-\int_X \Gamma(f^i,\varphi)\, \d\mm=- \lim_n \int_{X_n} \Gamma(f^i_n,\varphi)\, \d\mm_n= \lim_n \int_{X_n}  \varphi \, \d\Delta^\star f^i_n \leq \lim_{n} \frac{C(N)}  {n^{\beta}-2R-1} \mm_n(B_{2R+1}(\bar{x}_n))=0,
\end{equation}
where in the inequality we used the Laplacian comparison Theorem \ref{thm:LapComp} to infer that $\Delta^\star f^i_n \leq  \frac{C(N)}{n^{\beta}-2R-1} \mm_n$ on $B_{2R+1}(\bar{x}_n)$ for $n$ large enough.
\\It follows (see for instance \cite[Proposition 4.14]{Gigli12} ) that $f^i$ admits a measure-valued Laplacian on $X$ satisfying $\Delta^\star f^i \leq 0$ on $X$. On the other hand,  by completely analogous arguments we also get    $\Delta^\star g^i \leq 0$ on $X$ as a measure. The combination of these last two facts with CLAIM 1 gives CLAIM 2.
\\

CLAIM 3:\emph{ for every $a\in \R$ the function $af^i$ is a Kantorovich potential, $i=1,\ldots,k$. More precisely we show that $af^i$ is $c$-concave and satisfies }
\begin{equation}\label{eq:afc}
(af^i)^c = -a f^i- \frac{a^2}{2} \quad \text{and} \quad (-af^i)^c=a f^i -\frac{a^2}{2}.
\end{equation}
Though our situation is a bit different, our proof of this claim is inspired by the ideas of \cite{GigliSplitting}.  To simplify the notation let us drop the index $i$ in the arguments below. Since by construction  $f$ is $1$-Lipschitz, then for every $a \in \R$ the function $af$ is $|a|$-Lipschitz and 
$$a f(x)- a f(y) \leq |a| \sfd (x,y) \leq \frac{a^2}{2}+\frac{\sfd^2(x,y)}{2} \quad \forall x,y \in X \quad, $$
which gives $\frac{\sfd^2(x,y)}{2}-a f(x) \geq - a f(y)-\frac{a^2}{2}$ for every $x,y \in X$. Therefore, by the very definition of $c$-transform, we get
\begin{equation}\label{eq:afc>}
(af)^c(y):=\inf_{x\in X} \left(\frac{\sfd^2(x,y)}{2}-a f(x)\right) \geq - a f(y)-\frac{a^2}{2} \quad \forall y \in X \quad.
\end{equation}
To prove the converse inequality fix $y \in X$ and consider first the case $a\leq0$. For $n$ large enough let $y_n\in X_n$ be the point corresponding to $y$ via a GH-quasi isometry,  let $\gamma^{y_n,p_n}:[0, \sfd_n(y_n,p_n)]\to X_n$ be a unit speed minimizing geodesic from $y_n$ to $p_n$ and let $y^a_n:=\gamma^{y_n,p_n}_{|a|}$. In this way we have
\begin{equation}\label{eq:dnqna}
\sfd_n(p_n,y^a_n)=\sfd_n(p_n,y_n)+a \quad.
\end{equation}
From \eqref{eq:XntoX} and  since the space $(X,\sfd)$ is proper, we have that there exists  $y^a\in X$, a limit point of  $y^a_n$, such that 
\begin{equation}\label{eq:dyya}
\sfd(y,y^a)=|a|
\end{equation}
and, by using \eqref{fngntofg} and \eqref{eq:dnqna}, we  obtain 
\begin{equation} \label{eq:fyafy}
f(y^a)-f(y)= \lim_{n} \left[ f_n(y^a_n)-f_n(y_n) \right] =  \lim_{n} \left[ \sfd_n(p_n,y^a_n)-\sfd_n(p_n,y_n) \right] =a   \quad.
\end{equation}
By choosing $y^a$ as a competitor in the definition of $(af)^c$, thanks to  \eqref{eq:dyya} and  \eqref{eq:fyafy}, we infer
\begin{equation}\label{eq:afcpf}
(af)^c(y):= \inf_{x\in X} \left(\frac{\sfd^2(x,y)}{2}-a f(x)\right) \leq \frac{\sfd^2(y^a,y)}{2}-a f(y^a)=\frac{a^2}{2}- a^2- a f(y) = -\frac{a^2}{2}-a f(y) 
\end{equation}
as desired. To handle the case $a>0$ repeat the same arguments for $g^i,g^i_n$ in place of $f^i,f^i_n$: by considering this time $y^a_n:=\gamma^{y_n,q_n}_{a}$, where $\gamma^{y_n,q_n}:[0, \sfd_n(y_n,q_n)]\to X_n$ is a unit speed minimizing geodesic from $y_n$ to $q_n$, and passing to the limit as $n\to +\infty$ we get a point $y^a$ such that
$$f(y^a)-f(y)=-\left[g(y^a)-g(y)\right]=-  \lim_{n} \left[ g_n(y^a_n)-g_n(y_n) \right] = - \lim_{n} \left[ \sfd_n(q_n,y^a_n)-\sfd_n(q_n,y_n) \right] =a. $$
At this stage one can repeat verbatim \eqref{eq:afcpf} to conclude the proof of the first identity of  \eqref{eq:afc}; the second one follows by choosing $-a$ in place of $a$. The $c$-concavity of $af$ is now a direct consequence of \eqref{eq:afc}, indeed 
$$(af)^{cc}= \left(-a f- \frac{a^2}{2}  \right)^c= (-af)^c+ \frac{a^2}{2}=af \quad. $$

CLAIM 4: \emph{$|Df^i|\equiv 1$ everywhere on $X$, for every $i=1,\ldots,k$.}
\\Again, for sake of simplicity of notation, we  drop  the index $i$ for the proof. We already know that $|Df|\leq 1$ everywhere on $X$; to show the converse recall that by \eqref{eq:dyya}-\eqref{eq:fyafy} for every $a<0$ and $y \in X$ there exists $y^a \in X$ such that $\sfd(y,y^a)=|a|$ and $f(y^a)-f(y)=a$. Therefore it follows that
$$|Df|(y)=\limsup_{z\to y} \frac{|f(z)-f(y)|}{\sfd(z,y)}\geq \lim_{a\uparrow 0} \frac{|f(y^a)-f(y)|}{\sfd(y^a,y)}=1\quad,$$
as desired.
\\

CLAIM 5: \emph{$\Gamma(f^i,f^j)=0$ $\mm$-a.e. on $X$.}
\\First of all observe that it is enough to prove 
\begin{equation}\label{eq:claim5pf}
\Big|D\Big( \frac{f^i+f^j}{\sqrt{2}}-f^{ij}\Big) \Big|=0 \quad \mm\text{-a.e. on } X \quad.
\end{equation} 
Indeed, since $|Df^i|, |Df^{ij}|\equiv 1$ on $X$ (the proof for $f^{ij}$ can be performed along the same lines of  CLAIM 4), by polarization we get
\begin{eqnarray}
\left|\Gamma\Big(f^i,f^j \Big) \right| &=& \left|\,  \left|D \left( \frac{f^i+f^j }{\sqrt 2}\right)\right|^2 -1\,  \right|= \left| \, \left|D \left( \frac{f^i+f^j }{\sqrt 2}\right)\right|^2-  \left| D f^{ij}\right|^2 \right| \nonumber \\
&=&\left| \, \left|D  \left( \frac{f^i+f^j }{\sqrt 2}\right) \right|+  \left| D f^{ij}\right| \right| \; \cdot \; \left|\,   \left|D  \left( \frac{f^i+f^j }{\sqrt 2}\right) \right|-  \left| D f^{ij}\right| \right| \nonumber \\
&\leq&   10  \,  \, \left|D \left( \frac{f^i+f^j }{\sqrt 2}- f^{ij}\right) \right| \quad \mm\text{-a.e. on } X \quad . \nonumber
\end{eqnarray}
So let us establish \eqref{eq:claim5pf}. Observe that since $\left( \frac{f^i_n+f^j_n}{\sqrt{2}}-f^{ij}_n \right) \Big|_{B_R(\bar{x}_n)} \to  \left( \frac{f^i+f^j}{\sqrt{2}}-f^{ij}\right) \Big|_{B_R(\bar{x})}$ uniformly, and since they are  uniformly Lipschitz,  we have the lowersemicontinuity of the slope Proposition \ref{Prop:LSCSlope} which yields
$$
\int_{B_R(\bar{x})} \left|D\left( \frac{f^i+f^j}{\sqrt{2}}-f^{ij}\right) \right|^2 \d\mm \leq  \liminf_n  \int_{B_{R+1}(\bar{x}_n)} \left|D\left( \frac{f^i_n+f^j_n}{\sqrt{2}}-f^{ij}_n\right) \right|^2 \d\mm_n =0 \quad \text{for every fixed } R\geq 1 \quad, $$
thanks to  \eqref{eq:assn}, as desired.
\\

Using CLAIMS 2-3-4 we will argue by combining the ideas of Cheeger-Gromoll and Gigli \cite{GigliSplitting} with an induction argument.  Indeed, CLAIMS 2-3-4 are precisely the ingredients required to applying the arguments of  \cite{GigliSplitting} to obtain the following:

\begin{lemma}\label{l:splitting}
Each mapping $f^i:X\to \R$ is a splitting map.  That is, there exists a $\RCD^*(0,N-1)$ space $Y^i$ and an isomorphism $X\to Y^i\times\R$ such that $f^i(y,t)=t$.  If $N<2$ then $Y^i$ is exactly a point.
\end{lemma}
\begin{proof}
Since, with CLAIMS 2-3-4 in hand, the proof of the above is verbatim as in \cite{GigliSplitting},  we will not go through the details except to mention the main points.  Namely, CLAIM 3 first allows us to define the optimal transport gradient flow $\Phi_t:X\to X$ of $f^i$.  By CLAIM 2 this flow preserves the measure.  If $X$ were a smooth manifold, one would then use CLAIM 4 as by Cheeger-Gromoll to argue that $|\nabla^2 f^i|=0$, which would immediately imply that the flow map was a splitting map as claimed.  In the general case one argues  as in \cite{GigliSplitting} to use CLAIM 2 and CLAIM 4 to show the induced map $\Phi^*_t:W^{1,2}(X)\to W^{1,2}(X)$ is an isomorphism of Hilbert spaces, which forces $f^i$ to be the claimed splitting map.
\end{proof}

To finish the proof of the almost splitting theorem via excess we need to see that each $f^i$ induces a distinct splitting.  That is, we want to know that the mapping $f=(f^1,\ldots,f^k):X\to \R^k$ is a splitting map.  We will proceed by induction on $k$:\\

CONCLUSION \emph{of the proof of the almost splitting via excess. } 
\\If $k=1$ then the proof is complete from Lemma \ref{l:splitting}.  Now let us consider $k\geq 2$ and let us assume the mapping $(f^1,\ldots,f^{k-1}):X\to \R^{k-1}$ is a splitting map.  That is, there exists a $\RCD^*(0,N-k+1)$ space $X'$ such that $X=X'\times \R^{k-1}$ and such that under this isometry we have  $(f^1,\ldots,f^{k-1})(x',t_1,\ldots,t_{k-1}) = (t_1,\ldots,t_{k-1})$.  We will show the mapping $(f^1,\ldots,f^{k}):X\to \R^{k}$ is a splitting map.

To this aim let us consider the function
$$
\tilde{f}^k:=f^k\circ \iota=f^k((\cdot, 0)): X'\to \R\, ,
$$
where $\iota:X'\to X$ is the inclusion map of $X'$ into $X$ as the $0$-slice.
\\

CLAIM A: \emph{$f^k(x',t_1, \ldots, t_{k-1})=f^k(x',s_1,\ldots, s_{k-1}) \;$  for every $(s_1,\ldots, s_{k-1}),(t_1,\ldots,t_{k-1}) \in \R^{k-1}$ and every $x'\in X'$}.\\
Let $\Phi_t$ be the flow map induced by $f^{k-1}$. By following verbatim the proof of \cite[Proposition 2.17]{GigliSplittingSur} (or equivalently the proof of \cite[Corollary 3.24]{GigliSplitting}), which is based on the trick ``Horizontal-vertical derivative''  introduced in \cite{AGS11b},  we have that for every $g \in L^1\cap L^\infty(X,\mm)$ with bounded support it holds
$$
\lim_{t\downarrow 0} \int_X \frac{f^k \circ \Phi_t- f^k}{t} \, g \, \d \mm= - \int_{X} \Gamma(f^{k-1},f^k) \, g \, \d \mm =0  \quad,
$$
where in  the last equality we used Claim 5 above. Observing that for $\mm'\times \LL^{k-2}$-a.e. $(x', s')\in X'\times \R^{k-2}$ the map $t\mapsto f^k\circ \Phi_{t}((x',s',0))$ is 1-Lipschitz from $\R$ to $\R$, so in particular ${\mathcal L}^1$-a.e. differentiable, by the  Dominated Convergence Theorem  we infer that
\begin{eqnarray}
\int_{X} \frac{d}{dt} (f^k \circ \Phi_t )\, g \, \d \mm&=& \lim_{h\downarrow 0}\int_{X} \frac{f^k \circ \Phi_{t+h}- f^k  \circ \Phi_{t}}{h} \, g \, \d \mm \, \nonumber \\
&=&\lim_{h\downarrow 0} \int_X \frac{f^k \circ \Phi_h- f^k}{h} \, g\circ(\Phi_t)^{-1} \, \d \mm=0  \quad .\nonumber
\end{eqnarray}
It follows that for  $\mm'\times \LL^{k-2}$-a.e. $(x', s')\in X'\times \R^{k-2}$ we have $\frac{d}{dt} \Big(f^k\circ \Phi_t\Big) ((x', s',0))=0$ for $\LL^1$-a.e. $t \in \R$ and therefore, since again the function $t\mapsto f^k \circ \Phi_t$ is $1$-Lipschitz (so in particular absolutely continuous), 
\begin{equation}\label{eq:ficonst}
f^k(x',s',s)-f^k(x',s',0)=\int_{0}^{s}  \frac{d}{dt} \Big(f^k\circ \Phi_t((x',s',0))\Big) \, \d t=0, \quad \text{for } \mm'\times \LL^{k-2}\text{-a.e. } (x', s')\in X'\times \R^{k-2}, \, \forall s \in \R.   
\end{equation}
But, since $f^k:X \to \R$ is $1$-Lipschitz, the  identity \eqref{eq:ficonst} holds for every $(x',s')\in X'\times \R^{k-2}$. In other words the function $f^k(x',t_1,\ldots,t_{k-1})$ does not depend on the variable $t_{k-1}$. Repeating verbatim the argument above with $f^{i}$, $i=1,\ldots, k-2$, in place of $f^{k-1}$ gives CLAIM A.
\\

CLAIM B: \emph{ $|D\tilde{f}^k|_{X'}=1$} $\mm'$-a.e. on $X'$.
\\This claim is an easy consequence of a nontrivial result proved in \cite{AGS11b} (see  \cite[Theorem 3.13]{GigliSplittingSur} for the statement we refer to) stating that the Sobolev space of the product $X'\times \R$  splits isometrically into  the product of the corresponding Sobolev spaces.  More precisely, we already know that for $\mm'$-a.e. $x'$ the  function $t\mapsto f^k(x',t)$ from $\R^{k-1}$ to $\R$ is 1-Lipschitz (and then in particular locally of Sobolev class) as well as the function $x'\to f^k(x',t)$, the result then states the following orthogonal splitting
\begin{equation}\label{eq:SobolevSplit}
|Df^k|^2_{X'\times \R^{k-1}} (x',t)=|D f^k(\cdot, t)|^2_{X'} (x')+ |D f^k(x',\cdot)|^2_{\R^{k-1}} (t)  \quad \text{for $\mm'\times {\mathcal L}^{k-1}$-a.e. $(x',t)\in X'\times \R^{k-1}$} \quad.
\end{equation}
CLAIM B follows then by the combination of  CLAIM 4, CLAIM A and \eqref{eq:SobolevSplit}.
\\

CLAIM C: \emph{$\Delta^\star_{X'} \tilde{f}^k =0$ as a measure on $X'$}.
\\For every Lipschitz function $\tilde{\varphi}\in X'$ with bounded support, called $\varphi(x',t)=\tilde{\varphi}(x')$, thanks to CLAIM 2, using CLAIM A and that  (by polarization of identity \eqref{eq:SobolevSplit})
$$\Gamma_X(f^k,\varphi)((x',t))=\Gamma_{X'} (f^k(\cdot,t), \varphi(\cdot,t))(x')+ \Gamma_{\R^{k-1}}(f^k(x',\cdot), \varphi(x',\cdot))(t) \quad \text{for $\mm'\times {\mathcal L}^{k-1}$-a.e. $(x',t)\in X'\times \R^{k-1}$,}$$
 we have that
\begin{eqnarray}
0&=&\int_X  \varphi \, \d(\Delta^\star f^k)=- \int_X \Gamma_X(f^k,\varphi)\, \d \mm = -\int_{X'\times \R^{k-1}} \left[ \Gamma_{X'}(\tilde{f}^k,\tilde{\varphi})+\Gamma_{\R^{k-1}}(f^k,\varphi) \right] \, \d (\mm'\times{\mathcal L}^{k-1}) \nonumber \\
&=&- \int_{X'\times \R^{k-1}} \Gamma_{X'}(\tilde{f}^k,\tilde{\varphi})  \, \d (\mm'\times{\mathcal L}^{k-1})\quad. \nonumber
\end{eqnarray}
Therefore  $ \int_{X'} \Gamma_{X'}(\tilde{f}^k,\tilde{\varphi})  \, \d \mm'=0 $ and CLAIM C follows.
\\

CLAIM D: \emph{for every $a\in \R$ the function $a\tilde{f}^k$ is  a Kantorovich potential in $X'$. More precisely we show that $a\tilde{f}^k$ is $c$-concave and satisfies} 
\begin{equation}\label{eq:aftildec}
(a\tilde{f}^k)^c = -a \tilde{f}^k- \frac{a^2}{2} \quad \text{and} \quad (-a\tilde{f}^k)^c=a \tilde{f}^k -\frac{a^2}{2} \quad\text{everywhere on }X' \, .
\end{equation}
First of all observe that by CLAIM A we have 
$$\inf_{(y',t)\in X'\times \R^{k-1}} \frac{\sfd_{X'}^2(x',y')+|t|^2}{2}\mp a f^k((y',t))= \inf_{y'\in X'} \frac{\sfd_{X'}^2(x',y')}{2} \mp a f^k((y',0)) \quad \forall x' \in X'. $$
\\By the definition of $c$-transform, using  the  above identity,  recalling CLAIM 3 and  that we identified $X$ with $X'\times \R^{k-1}$, we have
\begin{eqnarray}
(\pm a \tilde{f}^k)^c(x')&=&\inf_{y' \in X'} \frac{\sfd^2_{X'}(x',y')}{2}\mp a \tilde{f}^k(y') =  \inf_{(y',t)\in X'\times \R^{k-1}} \frac{\sfd_{X'}^2(x',y')+|t|^2}{2}\mp a f^k((y',t)) \nonumber \\
&=& \inf_{y \in X} \frac{\sfd^2_X((x',0),y)}{2}\mp a f^k(y)= (\pm a f^k)^c((x',0))=\mp a f^k((x',0))-\frac{a^2}{2}\nonumber\\
&=&\mp a \tilde{f}^k(x')-\frac{a^2}{2} \nonumber 
\end{eqnarray}
which gives  \eqref{eq:aftildec}. The $c$-concavity then easily follows:
$$(a \tilde{f}^k)^{cc}=\left(-\frac{a^2}{2}- a \tilde{f}^k\right)^c=\frac{a^2}{2}+(- a \tilde{f}^k)^c=\frac{a^2}{2}- \frac{a^2}{2}+ a \tilde{f}^k= a \tilde{f}^k\quad.$$

Now we can apply Lemma \ref{l:splitting} to $\tilde f^k$, which completes the induction step and hence the proof.
\end{proof}

\section{Proof of the main results} \label{Sec:MainRes}

\subsection{Different stratifications coincide $\mm$-a.e.}\label{SS:Strat}

In this subsection we will analyze the following a priori different stratifications of the $\RCD^*(K,N)$-space $(X,\sfd,\mm)$; after having proved that they are made of $\mm$-measurable subsets we will show that different stratifications coincide $\mm$-a.e. .  We start by the definition. For every $k \in \N$, $1\leq k\leq N$ consider
\begin{eqnarray}
A_k&:=&\{x \in X\; : \; \R^k \in \Tan(X,x) \text{ but } \nexists  (Y,\sfd_Y,\mm_Y, \bar{y}) \text{ with } \diam(Y)>0 \text{ s.t. }  \R^k \times Y  \in \Tan(X,x) \}   \label{eq:defAk} \\
A_k'&:=&\{x \in X\; : \; \R^k \in \Tan(X,x) \text{ but  no } (\tilde{X}, \sfd_{\tilde{X}}, \mm_{\tilde{X}}, \tilde{x}) \in  \Tan(X,x) \text{ splits  } \R^{k+1}     \}   \label{eq:defAk'} \\
A_k''&:=&\{x \in X\; : \; \R^k \in \Tan(X,x) \text{ but   } \R^{k+j} \notin  \Tan(X,x) \text{ for every } j\geq 1     \}  \quad ,  \label{eq:defAk''} 
\end{eqnarray} 
where we wrote (and will sometimes write later on when there is no ambiguity on the meaning) $\Tan(X,x)$ in place of $\Tan(X,\sfd,\mm,x)$,   $\R^k$ in place of $(\R^k, \sfd_E, \LL_k,0 )$ and $\R^k \times Y$ in place of $(\R^k\times Y,\sfd_E\times \sfd_Y,\LL_k \times \mm_Y, (0,\bar{y}))$ in order to keep the notation short. It is clear from the definitions that 
\begin{equation}\label{eq:AkIncl}
A_k\subset A_k'\subset A_k''\quad,
\end{equation}
moreover, from Theorem \ref{thm:euclideantangents} we already know that 
\begin{equation}\label{eq:Ak''All}
\mm\left(X\setminus \bigcup_{1\leq k \leq N } A_k''\right)=0.
\end{equation}  
As preliminary step, in the next lemma we establish the measurability of $A_k, A_k', A_k''$. Let us point out that a similar construction was performed in \cite{GMR2013}.

\begin{lemma}[Measurability of the stratification]\label{lem:AkMeas}
For every $k\in \N, 1\leq k\leq N$, the sets $A_k, A_k', A_k''$ can be written as difference of couples of analytic sets so they are $\mm$-measurable.
\end{lemma}

\begin{proof}
We prove the statement for $A_k$, the argument for the others being analogous.
Define ${\mathcal A}\subset X \times \MM_{C(\cdot)}$ by
\begin{equation}\label{eq:defcalA}
{\mathcal A}:= \{(x,(Y,\sfd_Y,\mm_Y, y)) \; : \; (Y,\sfd_Y,\mm_Y, y) \in \Tan(X,x)\}. 
\end{equation}
Recall that for every $r\in \R$, the map $X\ni x \mapsto (X, \sfd_r, \mm^x_r, x) \in (\MM_{C(\cdot)}, \mathcal D_{C(\cdot)})$ is continuous, so the set 
$$\bigcup_{x \in X} \left[ \{x\}\times B_{\frac{1}{i}} \big( (X, \sfd_r, \mm^x_r, x) \big) \right] \subset X\times  \MM_{C(\cdot)} \quad \text{is open for every } i\in \N$$
and therefore
$${\mathcal A}= \bigcap_{i \in \N} \bigcap_{j \in \N}  \bigcup_{r<\frac{1}{j}} \bigcup_{x \in X}   \left[ \{x\}\times B_{\frac{1}{i}} \big( (X, \sfd_r, \mm^x_r, x) \big) \right] \quad \text{is Borel}. $$
Now let ${\mathcal B}\subset X \times \MM_{C(\cdot)}$ be defined by 
$${\mathcal B}:=\bigcup_{x \in X} \{x\}\times (\R^k,\sfd_E,\LL_k,0^k) =\bigcap_{i \in \N}  \bigcup_{x \in X} \left[ \{x\}\times  B_{\frac{1}{i}} \big((\R^k,\sfd_E,\LL_k,0^k)\big)  \right].$$
Clearly also ${\mathcal B}$ is Borel, as well as the below defined set ${\mathcal C}\subset X \times \MM_{C(\cdot)}$ 
$${\mathcal C}:=\bigcup_{\{Y\in  \MM_{C(\cdot)}, \diam(Y)>0 \}} \bigcup_{x \in X} \left[ \{x\}\times (\R^k \times Y)  \right] =\bigcap_{j \in \N} \bigcup_{i \in \N} \bigcup_{\{Y\in  \MM_{C(\cdot)}, \diam(Y)\geq \frac{1}{i} \}} \bigcup_{x \in X} \left[ \{x\}\times B_{\frac{1}{j}}
\big( (\R^k \times Y)  \big) \right] .$$
Calling $\Pi_1: X\times \MM_{C(\cdot)}\to X$ the projection on the first factor we have that $\Pi_1({\mathcal A}\cap {\mathcal B} \cap {\mathcal C})$ is analytic as well as $\Pi_1({\mathcal A}\cap {\mathcal B})$, as projections of Borel subsets. But 
\begin{eqnarray}
&\Pi_1({\mathcal A}\cap {\mathcal B} \cap {\mathcal C})=\{x \in X\;:\; \R^k \in \Tan(X,x) \text{ and } \exists Y\in \MM_{C(\cdot)} \text{ with } \diam(Y)>0 \text{ s.t. } \R^k \times Y \in \Tan(X,x) \}  \nonumber \\
&\Pi_1({\mathcal A}\cap {\mathcal B})=\{x \in X\;:\; \R^k \in \Tan(X,x) \}  \quad, \nonumber
\end{eqnarray}
so that $A_k=\Pi_1({\mathcal A}\cap {\mathcal B})\setminus \Pi_1({\mathcal A}\cap {\mathcal B} \cap {\mathcal C})$ is a difference of analytic sets and therefore is measurable with respect to any Borel measure; in particular $A_k$ is $\mm$-measurable.
\end{proof}

In the next lemma we prove that the a priori  different stratifications $A_k, A_k',A_k''$ essentially coincide.
\begin{lemma}[Essential equivalence of the different stratifications]\label{lem:AkEq}
Let $(X,\sfd,\mm)$ be an $\RCD^*(K,N)$-space and recall the definition of $A_k, A_k', A_k''$ in \eqref{eq:defAk}, \eqref{eq:defAk'}, \eqref{eq:defAk''} respectively.  Then
$$\mm(A_k''\setminus A_k)=0 \quad \text{ for every } 1\leq k \leq N\quad, $$
so in particular, thanks to \eqref{eq:AkIncl}, we have that $A_k=A_k'=A_k''$ up to sets of $\mm$-measure zero and 
$$\mm\left(X\setminus \bigcup_{1\leq k \leq N} A_k\right)=0\quad.$$
\end{lemma}

\begin{proof}
First recall that thanks to Theorem \ref{thm:iteratedtangents}, for $\mm$-a.e. $x \in X$, for every $(\tilde{X},\sfd_{\tilde{X}}, \tilde{\mm},\tilde{x})\in \Tan(X,\sfd,\mm,x)$ and for every $\tilde{x}'\in \tilde{X}$ we have 
\begin{equation}\label{eq:itertang}
\Tan(\tilde{X},\sfd_{\tilde{X}}, \tilde{\mm}_1^{\tilde{x}'},\tilde{x}') \subset  \Tan(X,\sfd,\mm,x).
\end{equation}
Fix $1\leq k \leq N$; we argue by contradiction by assuming that $\mm(A_k''\setminus A_k)>0$.  It follows that   there exists a point $\bar{x}\in A_k''\setminus A_k$ where the above property \eqref{eq:itertang} holds. By definition, if $\bar{x}\in A_k''\setminus A_k$ then $\R^k \in \Tan(X,\sfd,\mm,\bar{x})$ and there exists a p.m.m.s. $(Y,\sfd_Y, \mm_Y,y)$ with $\diam(Y)>0$ such that $\tilde{X}:= \R^k \times Y \in \Tan(X,\sfd,\mm,x)$. Notice that, since every element in $\Tan(X,\sfd,\mm,x)$ is an $\RCD^*(0,N)$-space, by applying $k$ times the Splitting Theorem \ref{thm:splitting} to $\tilde{X}$ we get that $Y$ is an $\RCD^*(0,N-k)$ space, in particular $Y$ is a geodesic space. Since by assumption $\diam(Y)>0$ then $Y$ contains at least two points and a geodesic $\gamma:[0,1]\to Y$ joining them. It follows that any blow up of $(Y, \sfd_Y, \mm_Y)$ at $\gamma(1/2)$ is an $\RCD^*(0,N-k)$ space containing a line and so it splits off an $\R$ factor by the Splitting Theorem \ref{thm:splitting}. Therefore there exists an $\RCD^*(0,N-k-1)$ space $(\tilde{Y},\sfd_{\tilde{Y}}, \mm_{\tilde{Y}},\tilde{y})$ such that 
$$\R^{k+1}\times \tilde{Y} \in \Tan\left(\tilde{X},\sfd_{\tilde{X}}, \tilde{\mm}_1^{(0,\gamma(1/2))}, \big(0,\gamma(1/2)\big) \right)$$
and, thanks to \eqref{eq:itertang} and our choice of $\bar{x}$, this yields
$$\R^{k+1}\times \tilde{Y}  \in \Tan(X,\sfd,\mm,\bar{x}) \quad .$$
If $\tilde{Y}$ is a singleton we have finished, indeed in this case we would have proved that $\R^{k+1} \cong \R^{k+1}\times \{\tilde{y}\} \in  \Tan(X,\sfd,\mm,\bar{x})$, contradicting the assumption that $\bar{x}\in A_k''$.
\\If instead $\tilde{Y}$  contains at least two points  then it contains a geodesic joining them and we can repeat the arguments above to show that $\R^{k+2}\times \tilde{{\tilde{Y}}}  \in \Tan(X,\sfd,\mm,\bar{x})$ for some $\RCD^*(0,N-k-2)$ space   $\tilde{{\tilde{Y}}}$.  Recalling that an $\RCD^*(0,\tilde{N})$ space for $0\leq \tilde{N} < 1$ is a singleton, after a finite number of iterations of the above procedure we get that $\R^{k+j}\in \Tan(X,\sfd,\mm,\bar{x})$ for some $1\leq j \leq N-k$, contradicting that $\bar{x}\in A_k''$.
\end{proof}

In order to establish the rectifiability, the following easy lemma will also be useful.  

\begin{lemma}\label{lem:delta}
i) For every $x \in A_k$ and for every $\vare>0$ there exists $\delta=\delta(x,\vare)>0$ such that for every $0<r<\delta$ the following holds: If for some p.m.m.s. $(Y,\sfd_Y, \mm_Y, y)\in \MM_{C(.)}$ one has
\begin{equation}\label{eq:SmallDY}
\mathcal D_{C(\cdot)} \left((X,\sfd_r, \mm^x_r, x), \big(Y\times \R^k, \sfd_{Y}\times \sfd_E, \mm_Y \times \LL_k, (y,0^k) \big) \right) \leq \delta \quad \text{then} \quad  \diam{Y}\leq \vare.
\end{equation}

ii) Define the function $\bar{\delta}(\cdot, \vare):A_k \to \R^+$ by
\begin{equation} \label{eq:defbardelta}
\bar{\delta}(x, \vare):= \sup \{ \delta(x, \vare) \text{ such that  \eqref{eq:SmallDY} holds} \} \quad .
\end{equation}
Then, for every fixed $\vare_1, \delta_1>0$, the set of points 
\begin{equation}\label{eq:d1meas}
\{x \in A_k: \bar{\delta}(x, \vare_1)\geq \delta_1\} \subset X \text{ is the complementary  in $A_k$ of an analytic subset of $X$ and therefore $\mm$-measurable.}
\end{equation}
\end{lemma}

\begin{proof}
i) If by contradiction \eqref{eq:SmallDY} does not hold then there exist $\vare>0$, a sequence $r_j\downarrow 0$ and a sequence of p.m.m.s. $(Y_j,\sfd_{Y_j}, \mm_{Y_j}, y_j)\in \MM_{C(.)}$ s.t. 
$$\mathcal D_{C(\cdot)} \left((X,\sfd_{r_j}, \mm^x_{r_j},x), \big(Y_j\times \R^k, \sfd_{Y_j}\times \sfd_E, \mm_{Y_j} \times \LL_k, (y_j,0^k) \big) \right) \to 0  \quad \text{and} \quad  \diam{Y_j}\geq \vare. $$
But by p-mGH compactness Proposition \ref{prop:comp} there exists a p.m.m.s.  $(Y,\sfd_{Y}, \mm_{Y}, y)\in \MM_{C(.)}$ such that, up to subsequences,  $Y_j\to Y$ in p-mGH sense. It follows that $\diam Y>0$ and $Y\times \R^k \in \Tan(X,\sfd,\mm,x)$ contradicting that $x \in A_k$.
\\

ii) The construction is analogous to the one performed in the proof of Lemma \ref{lem:AkMeas}, let us sketch it briefly. Clearly, for every $\delta_1, \vare_1>0$, the following subset ${\mathcal D}_{\delta_1, \vare_1}\subset X\times \MM_{C(.)}$
$${\mathcal D}_{\delta_1, \vare_1}:=\bigcap_{i \in \N} \bigcup_{0<r<\delta_1} \bigcup_{x \in X} \{x\}\times \left[B_{1/i}\left((X,\sfd_r, \mm^x_r, x)\right)  \cap \left( \bigcup_{Y\in \MM_{C(.)}:\diam(Y)> \vare_1} B_{\delta_1} (Y\times \R^k) \right)  \right] \quad \text{ is Borel }. $$
Therefore $\Pi_1({\mathcal D}_{\delta_1, \vare_1})\subset X$ is analytic and, since $\{x \in A_k:\bar{\delta}(x, \vare_1)\geq \delta_1\}=A_k\setminus \Pi_1({\mathcal D}_{\delta_1, \vare_1})$, the thesis follows.
\end{proof}

\subsection{Rectifiability of $\RCD^*(K,N)$-spaces}\label{SS:Rect}
The goal of this section is to prove that given an $\RCD^*(K,N)$-space $(X,\sfd,\mm)$, every $A_k$ defined in \eqref{eq:defAk} is $k$-rectifiable, the rectifiability Theorem \ref{thm:rect} will then  follow by Lemma \ref{lem:AkEq}.

To this aim, fix $\bar{x}\in A_k$. By Theorem \ref{lem:ConstrU}, for every $\vare_2>0$ there exists a rescaling $(X,\tilde{\sfd}, \tilde{\mm}, \bar{x}):=(X,r^{-1} \sfd, \mm^{\bar{x}}_r, \bar{x})$ for some $0<r<<1$, with the following property: 
 there exist $R>10$ and points $\{p_i,q_i\}_{i=1,\ldots, k}\subset \partial B^{\tilde{\sfd}}_{R}(\bar{x}), \{p_i+q_j\}_{1\leq i <j \leq k} \in B^{\tilde{\sfd}}_{2R}(\bar{x})\setminus B^{\tilde{\sfd}}_R(\bar{x})$ such that   
\begin{equation} \label{eq:B10e2}
\sum_{i=1}^k  \sup_{B^{\tilde{\sfd}}_{10}(\bar{x})}    e_{p_i,q_i} +   \fint _{B^{\tilde{\sfd}}_{10}(\bar{x})} \left( \sum_{i=1}^k  |D e_{p_i,q_i}|^2 + \sum_{1\leq i< j\leq k}   \left|D \left( \frac{\tilde{\sfd}^{p_i}+\tilde{\sfd}^{p_j}} {\sqrt 2} - \tilde{\sfd}^{p_i+p_j}\right) \right|^2 \right) \, \d \tilde{\mm}   \leq \varepsilon_2\quad.
\end{equation}
Consider the maximal function $M^k:B^{\tilde{\sfd}}_9(\bar{x})\to \R^+$ defined by 
\begin{equation}\label{eq:defMk}
M^k(x):= \sup_{0<r<1}  \fint _{B^{\tilde{\sfd}}_{r}(x)} \left( \sum_{i=1}^k  |D e_{p_i,q_i}|^2 + \sum_{1\leq i< j\leq k}   \left|D \left( \frac{\tilde{\sfd}^{p_i}+\tilde{\sfd}^{p_j}} {\sqrt 2} - \tilde{\sfd}^{p_i+p_j}\right) \right|^2 \right) \, \d \tilde{\mm}.
\end{equation}

\begin{lemma} \label{lem:Mk}
For every rescaling $(X,\tilde d,\tilde\mm,\bar x)$ and $\vare_1>0$ the subset  $\{x\in B_9^{\tilde{\sfd}}(\bar{x}) : M^k(x)>\vare_1\}$ is Borel. Moreover for every $\vare_1>0$ there exists $\vare_2>0$ such that if the rescaling $(X,\tilde \sfd,\tilde\mm,\bar x)$ satisfies \eqref{eq:B10e2} then 
\begin{equation} \label{eq:Mke1}
\tilde \mm(\{x \in  B_9^{\tilde{\sfd}}(\bar{x}) : M^k(x)>\vare_1\})\leq \vare_1 \quad.
\end{equation}
\end{lemma}

\begin{proof}
The first claim is trivial since $M^k$ is Borel (it is lower semicontinuous as supremum of  continuous functions). For the second claim, observe that  $(X,\tilde{\sfd}, \tilde{\mm})$ is $\RCD^*(\min(K,0),N)$ thanks to property \eqref{eq:cdinv} so $B_{10}^{\tilde{\sfd}}(\bar{x})$ is doubling with constant depending just on $K$ and $N$. But then \eqref{eq:Mke1} follows by the continuity of the maximal function operator from $L^1$ to weak-$L^{1}$holding  in doubling spaces (for the proof see for instance \cite[Theorem 2.2]{Hein}), which gives $\mm(\{M^k>t\}) \leq \frac{C(K,N)}{t} \vare_2 $.
\end{proof}

Now for a fixed rescaling $(X,\tilde \sfd,\tilde\mm,\bar x)$ and any $\vare_1,\delta_1>0$, let us define the sets
\begin{eqnarray}
U^k_{\vare_1}(\bar x,r) = U^k_{\vare_1}&:=&\{x \in B_9^{\tilde{\sfd}}(\bar{x}) \cap A_k \text{ such that } M^k(x)\leq \vare_1 \} \label{eq:defUke1}  \quad,\\
U^k_{\vare_1, \delta_1}(\bar x,r)=U^k_{\vare_1, \delta_1}&:=&\{x \in B_9^{\tilde{\sfd}}(\bar{x}) \cap A_k \text{ such that } M^k(x)\leq \vare_1 \text{ and } \bar{\delta}(x, \vare_1)\geq \delta_1 \} \quad, \label{eq:defUke1}
\end{eqnarray}
where the map $x \to \bar{\delta}(x, \vare_1)$ was defined in \eqref{eq:defbardelta}. Thanks to (i) of Lemma \ref{lem:delta} we know that $\bar{\delta}(x,\vare_1)>0$ for every $x \in A_k$ and $\vare_1>0$, so 
\begin{equation}\label{eq:Ukbigcup}
U^k_{\vare_1}(\bar x,r) = U^k_{\vare_1}=\bigcup_{j \in \N} U^k_{\vare_1, \frac{1}{j}}\quad . 
\end{equation}
Therefore, to establish the rectifiability of $U^k_{\vare_1}$,  it is enough to  prove that $U^k_{\vare_1, \frac{1}{j}}$ is rectifiable for every $\vare_1>0$. This is our next claim, which is the heart of the proof of the rectifiability of $\RCD^*(K,N)$-spaces. The idea is to ``bootstrap'' to smaller and smaller scales  the excess estimates initially given by Theorem \ref{lem:ConstrU} by using the smallness of the  Maximal function and then to convert these excess estimates into estimates  on the Gromov-Hausdorff distance with $\R^k$ thanks to the Almost Splitting Theorem \ref{thm:AlmSplit} combined with  Lemma \ref{lem:delta}. The conclusion will follow by observing that GH closeness at arbitrary small scales via the same map implies biLipschitz equivalence. 

 \begin{theorem}[Rectifiabilty of  $U^k_{\vare_1, \delta_1}$, of $U^k_{\vare_1}$ and measure estimate] \label{thm:rectUked}
For every $\vare_3>0$ there exist $\delta_1,\vare_1>0$ such that if $(X,\tilde \sfd,\tilde\mm,\bar x)$ is a rescaling satisfying \eqref{eq:B10e2}, where $\epsilon_2>0$ is from Lemma \ref{lem:Mk}, then for every ball $B^{\tilde{\sfd}}_{\delta_1}\subset B_9^{\tilde{\sfd}}(\bar{x})$ of radius $\delta_1$ we have
\begin{equation}\label{eq:rectUked}
B^{\tilde{\sfd}}_{\delta_1}\cap  U^k_{\vare_1, \delta_1} \text{ is $1+\vare_3$-biLipschitz equivalent to a measurable subset of $\R^k$}\quad.
\end{equation}
It follows that $U^k_{\vare_1, \delta_1}$ and, thanks to \eqref{eq:Ukbigcup}, $U^k_{\vare_1}$ are $k$-rectifiable via $1+\vare_3$-biLipschitz maps as well. Moreover 
\begin{equation}\label{eq:measest}
\tilde{\mm}\left((B_9^{\tilde{\sfd}}(\bar{x}) \cap A_k) \setminus U^{k}_{\vare_1}\right) \leq \vare_1.
\end{equation}
\end{theorem}

\begin{proof}
First notice that, thanks to  Lemma \ref{lem:Mk}, Lemma \ref{lem:delta} and Lemma \ref{lem:AkMeas}, the subset  $U^k_{\vare_1, \delta_1}\subset X$ is constructed via a finite combination of intersections and complements of analytic subsets of $X$ so, if we manage to construct a  $1+\vare_3$-biLipschitz map $u$ between $B^{\tilde{\sfd}}_{\delta_1}\cap  U^k_{\vare_1, \delta_1}$ and   a subset $E\subset \R^k$, $E$ will be automatically expressible as a finite combination of intersections and complements of analytic subsets of $\R^k$. Moreover the measure estimate \eqref{eq:measest} readily follows by  Lemma \ref{lem:Mk} and the definition of $ U^{k}_{\vare_1}$ in \eqref{eq:defUke1} . Therefore it is enough to prove  \eqref{eq:rectUked}, i.e.  we have to construct such a map $u$.

We start by fixing  a $\varepsilon_4\in (0,1)$  such that  
\begin{equation}\label{eq:defeps4}
\max \left\{1+\vare_4\, , \, \frac{1}{1-\vare_4} \right\} \leq 1+\vare_3.
\end{equation}
Combining the Almost Splitting Theorem \ref{thm:AlmSplit} with Lemma \ref{lem:delta} and the very definition \eqref {eq:defUke1} of $U^k_{\vare_1, \delta_1}$, we infer that for every $\vare_4>0$ there exist $\vare_1,\delta_1>0$ small enough such that if $(X,\tilde{\sfd}, \tilde{\mm}, \bar{x})$ satisfies \eqref{eq:B10e2}, where $\epsilon_2$ is from Lemma \ref{lem:Mk}, then for every $x \in  U^k_{\vare_1, \delta_1}$ and for every $0<r\leq 2 \delta_1$ it holds
\begin{equation}\label{eq:DXRk}
\DC \left( \left(X, r^{-1} \tilde{\sfd}, \tilde{\mm}^x_r, x \right), \left( \R^k ,\sfd_{\R^k},\LL_k, 0^k)  \right) \right) \leq  \vare_4 \quad . 
\end{equation}
Moreover the  GH $\vare_4$-quasi isometry map $u_{x,r}:B^{r^{-1}\tilde{\sfd}}_{1}(x) \to \R^k$ is  given by 
$$u^i_{x,r}(\cdot):=r^{-1}\left( \tilde{\sfd}(p_i,\cdot)- \tilde{\sfd}(p_i,x)\right), \; i=1,\ldots,k.$$
This means that for every $0<r \leq 2 \delta_1$ and every $y_1,y_2 \in B^{r^{-1}\tilde{\sfd}}_{1}(x)$  it holds
$$ \left| \sqrt{\sum_{i=1}^k  \left(u^i_{x,r} (y_1)- u^i_{x,r}(y_2) \right)^2} -  r^{-1} \tilde{\sfd} (y_1,y_2)  \right|  \leq  \vare_4 \quad ,$$
which implies, after rescaling by $r$, that for every $0<r \leq 2 \delta_1$ and every $y_1,y_2 \in B^{\tilde{\sfd}}_{r}(x)$  it holds
\begin{equation} \label{eq:uixr}
\left| \sqrt{\sum_{i=1}^k  \left( \tilde{\sfd}(p_i,y_1)- \tilde{\sfd}(p_i,y_2) \right)^2} -   \tilde{\sfd} (y_1,y_2)  \right| = r  \left| \sqrt{\sum_{i=1}^k  \left(u^i_{x,r} (y_1)- u^i_{x,r}(y_2) \right)^2} -  r^{-1} \tilde{\sfd} (y_1,y_2)  \right|  \leq r \vare_4\quad.
\end{equation}
Hence, calling $u: B^{\tilde{\sfd}}_9(\bar{x})\to \R^k$ the map $u^i(\cdot):=\tilde{\sfd}(p_i,\cdot)-\tilde{\sfd}(p_i,\bar{x})$ with $i=1,\ldots,k$,  for every $x_1,x_2 \in U^k_{\vare_1, \delta_1}$ with $\tilde{\sfd}(x_1,x_2)\leq 2 \delta_1$ the above estimate \eqref{eq:uixr} ensures that
$$\left|  | u(x_1)-u(x_2)|_{\R^k} - \tilde{\sfd}(x_1,x_2) \right| \leq \vare_4 \tilde{\sfd}(x_1,x_2) \quad, $$
which gives
\begin{equation}\label{eq:uLip}
(1-\vare_4) \,  \tilde{\sfd}(x_1,x_2) \leq  | u(x_1)-u(x_2)|_{\R^k} \leq  (1+\vare_4) \, \tilde{\sfd}(x_1,x_2) \quad.
\end{equation}
This is to say the map $u:B^{\tilde{\sfd}}_{\delta_1}\cap  U^k_{\vare_1, \delta_1}$ is $(1+\vare_3)$-biLipschitz to its image in $\R^k$, in virtue of \eqref{eq:defeps4}. 
\end{proof}

To finish the rectifiability let $\{x_\alpha\} \subset A_k$ be a countable dense subset. Notice that such  a subset exists since $X$ is locally compact.  Let us denote the sets
\begin{align}\label{eq:defRke}
R_{k,\vare}:=\bigcup_{\alpha,j\in \N} U^k_\vare(x_\alpha,j^{-1})\quad,
\end{align}
where   $U^k_\vare(x_\alpha,j^{-1})$ was defined in \eqref{eq:defUke1}. 
It is clear from Theorem \ref{thm:rectUked} that for $\epsilon(N,K)$ sufficiently small the set $R_{k,\vare}$ is rectifiable, since it is a countable union of such sets.  We only need to  see that $\mm(A_k\setminus R_{k,\vare})=0$ via a standard measure-density argument.  

\begin{theorem}[$k$-rectifiability of $A_k$]\label{thm:Akrect}
Let $(X,\sfd,\mm)$ be an $\RCD^*(K,N)$-space and let $A_k$ be defined in \eqref{eq:defAk}. Then there exists $\bar{\vare}=\bar{\vare}(K,N)>0$ such that, for every $1\leq k \leq N$ and $0<\vare\leq \bar{\vare}$, one has that 
$$\mm(A_k\setminus R_{k,\vare})=0 \quad,$$
where $R_{k,\vare}$ is the $k$-rectifiable set  defined in   \eqref{eq:defRke}.
\end{theorem}

\begin{proof}
If by contradiction $\mm(A_k\setminus R_{k,\vare})>0$ then there exists an $\mm$-density point $\bar{x}\in A_k$ of   $A_k\setminus R_{k,\vare}$, i.e.
\begin{equation} \label{eq:mdens}
\lim_{r\downarrow 1} \frac{\mm\left((A_k\setminus R_{k,\vare}) \cap B^{\sfd}_r(\bar{x})  \right)}{\mm(B^{\sfd}_r(\bar{x}))}=1 \quad. 
\end{equation}
Note that by applying  Theorem  \ref{lem:ConstrU},   for any $\vare_2>0$ and for some $0<r\leq \bar{r}(\bar{x},\vare_2)$ sufficiently small we have  that \eqref{eq:B10e2} holds. Therefore, by taking $\vare_2>0$ sufficiently small, for every $j\geq \bar{j}(\bar{x},\vare_2, \vare)$ large enough, there exists $x_\alpha$  sufficiently close to $\bar x$ such that 
$$\mm^{\bar{x}}_{j^{-1}}\left(B^{\sfd}_{j^{-1}}(\bar{x})\cap \Big(U^k_{\vare}(\bar{x},j^{-1})\setminus U^k_{\vare}(x_{\alpha},j^{-1}) \Big)\right)\leq \vare \quad,$$
and, recalling the measure estimate \eqref{eq:measest}, we infer 
$$
\mm^{\bar{x}}_{j^{-1}}\left(\big(B^{\sfd}_{j^{-1}}(\bar{x})\cap  A_k\big) \setminus U^k_{\vare}(x_{\alpha},j^{-1}) \right)\leq 2 \vare.
$$
But now, from the very definition \eqref{eq:normalization}  of the rescaled measure $\mm^{\bar{x}}_{j^{-1}}$ and from  the measure  doubling property ensured  by the $\RCD^*(K,N)$ condition, we have that
$$\frac{ \mm \left(\big(B^{\sfd}_{j^{-1}}(\bar{x})\cap  A_k\big) \setminus U^k_{\vare}(x_{\alpha},j^{-1}) \right)}{\mm\left(B_{j^{-1}}^{\sfd}(\bar{x})\right)}\leq C(K,N) \; \mm^{\bar{x}}_{j^{-1}}\left(\big(B^{\sfd}_{j^{-1}}(\bar{x})\cap  A_k\big) \setminus U^k_{\vare}(x_{\alpha},j^{-1}) \right) \leq 2 C(K,N)  \, \vare \leq \frac{1}{2}$$
for  $\vare \leq \frac{1}{4C(K,N)}$. Since by definition $U^k_{\vare}(x_{\alpha},j^{-1}) \subset R_{k, \vare}$, the last inequality clearly contradicts \eqref{eq:mdens} for $j$ large enough. 
\end{proof}

\subsection{$\mm$-a.e. uniqueness of tangent cones}\label{SS:UniqTC}

The $k$-rectifiability of $A_k$ establishes immediately that for $\mm$-a.e. $x \in A_k$ the tangent cone is unique and isomorphic to the euclidean space $\R^k$; the $\mm$-a.e. uniqueness of tangent cones of $\RCD^*(K,N)$-spaces expressed in Theorem \ref{thm:UniqTC} will then follow from Lemma \ref{lem:AkEq}. For completeness sake we include the argument:
  
 \begin{theorem}
 Let $(X,\sfd,\mm)$ be an $\RCD^*(K,N)$-space and let $A_k$ be defined in \eqref{eq:defAk}, for  $1\leq k \leq N$. Then, for $\mm$-a.e. $x \in A_k$ the tangent cone is unique and $k$-dimensional euclidean, i.e.
 \begin{equation}
 \Tan(X, \sfd, \mm, x) = \left\{ \left(\R^k, \sfd_E, \LL_k, 0^k\right) \right\} \quad.
 \end{equation}  
 \end{theorem}
 
 \begin{proof}
 Let  $S_n\subset X$ be defined by
 $$S_n:=\left\{x \in X \, : \,  \exists (Y,\sfd_Y,\mm_Y,y)\in \Tan(X,\sfd,\mm,x) \text{ with } \DC\left((Y,\sfd_Y,\mm_Y,y),(\R^k,\sfd_{\R^k},\LL_k,0^k) \right)>\frac{1}{n}\right\} \quad. $$
 Observe that $S_n\subset X$ is analytic since it can be written as projection of a Borel subset of $X\times \MM_{C(.)}$:
 $$S_n=\Pi_1\left[ {\mathcal A} \cap \left( \bigcup_{x \in X} \{x\} \times \left( \MM_{C(.)} \setminus \bar{B}_{1/n}\big((\R^k,\sfd_{\R^k},\LL_k,0^k)\big) \right)  \right) \right]\quad,$$
 where ${\mathcal A}\subset X\times \MM_{C(.)}$ is the Borel subset (see the proof of Lemma \ref{lem:AkMeas}) defined in  \eqref{eq:defcalA}.
 In order to get the thesis it is clearly enough to prove that $\mm(A_k\cap S_n)=0$, for every $n \in \N$.  If by contradiction for some $n\in \N$ one has $\mm(A_k\cap S_n)>0$,  then there exists an $\mm$-density point $\bar{x}\in A_k$ of $A_k\cap S_n$, i.e.
 \begin{equation}\label{eq:mdensAkSn} 
 \lim_{r\downarrow 0} \frac{ \mm\left(A_k\cap S_n \cap B^{\sfd}_{r}(\bar{x})\right)}{\mm\left(B_r^{\sfd}(\bar{x})\right)}=1\quad .
 \end{equation}
 Repeating the first part of the proof of Theorem  \ref{thm:rectUked} we get that for  $\vare_3=\frac{1}{2n}$ and for every  $\vare_1,\delta_1>0$ (to be fixed later depending  just on $K$ and $N$) there exists $r_0=r_0(\bar{x},\vare_3,\vare_1,\delta_1)>0$ such that  the rescaled space $(X,r_0^{-1}\sfd, \mm^{\bar{x}}_{r_0}, \bar{x})$ has a  subset $U^k_{\vare_1,\delta_1}$ satisfying
 \begin{equation}\label{eq:Uke1d1e1}
 \mm^{\bar{x}}_{r_0} \left(B^{\sfd}_{r_0}(\bar{x})\setminus  U^k_{\vare_1,\delta_1}\right)\leq  2 \vare_1 \quad,
 \end{equation}
 and such that, for every $x \in B^{\sfd}_{r_0}(\bar{x})\cap  U^k_{\vare_1,\delta_1}$, one has
 \begin{equation}\label{eq:Xrr0Rk}
 \DC \left( \Big(X,(rr_0)^{-1}\sfd, \mm^{\bar{x}}_{rr_0}, \bar{x}\Big), \Big(\R^k,\sfd_{\R^k},\LL_k,0^k\Big)  \right)  \leq \frac{1}{2n} \quad \text{for every } \; 0<r\leq 2\delta_1\quad .
 \end{equation}
 The last property \eqref{eq:Xrr0Rk} implies that,   for every $x \in B^{\sfd}_{r_0}(\bar{x})\cap  U^k_{\vare_1,\delta_1}$, one has $\Tan(X, \sfd, \mm, x) \subset B_{\frac{1}{2n}} \Big((\R^k,\sfd_{\R^k},\LL_k,0^k)\Big)$ so, by the very definition of $S_n$, that $S_n\cap B^{\sfd}_{r_0}(\bar{x})\cap  U^k_{\vare_1,\delta_1}=\emptyset$ or, in other terms, $A_k\cap S_n \cap B^{\sfd}_{r_0}(\bar{x})\subset  B^{\sfd}_{r_0}(\bar{x})\setminus  U^k_{\vare_1,\delta_1}$. But now, from the very definition \eqref{eq:normalization}  of the rescaled measure $\mm^{\bar{x}}_{r_0}$ and from  the measure  doubling property ensured  by the $\RCD^*(K,N)$ condition, we have that
$$\frac{ \mm\left(A_k\cap S_n \cap B^{\sfd}_{r_0}(\bar{x})\right)}{\mm\left(B_{r_0}^{\sfd}(\bar{x})\right)}\leq C(K,N) \; \mm^{\bar{x}}_{r_0} \left(A_k\cap S_n \cap B^{\sfd}_{r_0}(\bar{x})\right) \leq C(K,N) \;  \mm^{\bar{x}}_{r_0} \left( B^{\sfd}_{r_0}(\bar{x})\setminus  U^k_{\vare_1,\delta_1} \right)  \leq 2 C(K,N) \, \vare_1 \leq \frac{1}{2}$$
for  $\vare_1\leq \frac{1}{2C(K,N)}$, where we used \eqref{eq:Uke1d1e1}. For $r_0>0$ small enough this clearly contradicts \eqref{eq:mdensAkSn} .
 \end{proof}
 
 \section*{Appendix: an $\varepsilon$-regularity result}
 In this appendix we isolate the following $\varepsilon$-regularity result which was proved along the lines of the rectifiability theorem, since it may be useful for future developments of the theory of $\RCD^{*}(K,N)$-spaces. For the reader's convenience we give a self-contained argument.
 
 \begin{theorem}\label{thm:epsreg1}
  For every  $N\in (1,\infty)$ and $\varepsilon\in (0,1)$ there exists $\delta=\delta(\varepsilon, N)>0$ with the following property. Let $(X,\sfd,\mm)$ be an $\RCD^{*}(-\delta, N)$-space and assume that  for some $\bar{x} \in X$ it holds
 $$\sfd_{mGH} \left( \Big(B_{1/\delta}(\bar{x}) , \sfd, \mm \Big), \Big(B_{1/\delta}(0),\sfd_{\R^k},\LL_k\Big)  \right)  \leq \delta,  $$
 then there exists a Borel subset $U\subset B_{1}(\bar{x})$ such that  $\mm(B_{1}(\bar{x})\setminus U))\leq \varepsilon$ and $U$ is $(1+\varepsilon)$-biLipschitz to a subset of $\R^{k}$. 
 \end{theorem}
 
 \begin{proof}
 Let $\vare_{1}>0$ small to be fixed later, depending just on $\vare$ and $N$ via the Almost Splitting Theorem \ref{thm:AlmSplit}.
 Following verbatim the proof of Theorem \ref{lem:ConstrU}  we obtain that, if  $\delta>0$ is chosen small enough,  there exist points $p_{i}, q_{i} \in \partial B_{\frac{1}{2\delta}}(\bar{x})$ for $i=1,\ldots, k$,  and $\{p_{i} + p_{j}\}_{1\leq i<j \leq k} \subset  B_{\frac{3}{4\delta}}(\bar{x}) \setminus  B_{\frac{1}{2\delta}}(\bar{x})$ such that
 \begin{equation}
\sum_{i=1}^{k} \sup_{B_{\frac{1}{\delta^{1/3}}}(\bar{x})} e_{p_{i},q_{i}}+  \fint _{B_{\frac{1}{\delta^{1/3}}}(\bar{x})} \left( \sum_{i=1}^k  |D e_{p_i,q_i}|^2 + \sum_{1\leq i< j\leq k}   \left|D \left( \frac{\sfd_{p_i}+ \sfd_{p_j}} {\sqrt 2} - \sfd_{p_i+p_j}\right) \right|^2 \right) \, \d \mm \leq \vare_{1}^{2},
 \end{equation}
where $\sfd_{p_i}(\cdot):=\sfd(p_i, \cdot)$ and the excess $e_{p_i,q_i}$ is defined by $e_{p_i,q_i}(\cdot):=\sfd_{p_i}(\cdot)+\sfd_{q_i}(\cdot)-\sfd(p_i,q_i)$.   
Consider the maximal function $F:B_1(\bar{x})\to \R^+$ defined by 
\begin{equation}\label{eq:defMkEPSREG}
F(x):= \sup_{0<r<\frac{1}{2\delta^{1/3}}}  \fint _{B_{r}(x)} \left( \sum_{i=1}^k  |D e_{p_i,q_i}|^2 + \sum_{1\leq i< j\leq k}   \left|D \left( \frac{\sfd_{p_i}+ \sfd_{p_j}} {\sqrt 2} - \sfd_{p_i+p_j}\right) \right|^2 \right) \, \d \mm.
\end{equation}
From the continuity of the maximal function operator from $L^{1}$ to $L^{1}$-weak  (see the proof of Lemma \ref{lem:Mk} for the details), we get that
\begin{equation} \label{eq:Mke1EPSREG}
 \mm(\{x \in  B_1(\bar{x}) : F(x)>\vare_{1} \})\leq C(N) \, \frac{\vare_{1}^{2}}{\vare_{1}} = C(N) \, \vare_{1}.
\end{equation}
Let 
$$
U:=\{x \in  B_1(\bar{x}) : F(x)\leq \vare_{1} \},
$$
and notice that, if we choose $\vare_{1}\in (0, \vare]$, we get the desired measure estimate
\begin{equation}\label{eq:estmBminusUEPSREG}
\mm(B_{1}(\bar{x})\setminus U) \leq \vare.
\end{equation}
 Fix now $\varepsilon_2=\vare_{2}(\vare)\in (0,1)$  such that  
\begin{equation}\label{eq:defepsEPSREG}
\max \left\{1+\vare_2\, , \, \frac{1}{1-\vare_2} \right\} \leq 1+\vare.
\end{equation}
By  the Almost Splitting Theorem \ref{thm:AlmSplit}, if $\vare_{1}=\vare_{1}(\vare_{2},N)=\vare_{1}(\vare,N)>0$  and $\delta>0$ are chosen small enough, then  for every $x \in  U$ and $r \in (0,1]$  it holds
\begin{equation}\label{eq:DXRkEPSREG}
\DC \left( \left(X, r^{-1} \sfd, \mm^x_r, x \right), \left( \R^k ,\sfd_{\R^k},\LL_k, 0^k)  \right) \right) \leq  \vare_2 \quad . 
\end{equation}
Moreover the  GH $\vare_2$-quasi isometry map $u_{x,r}:B^{r^{-1}\sfd}_{1}(x) \to \R^k$ is  given by 
$$u^i_{x,r}(\cdot):=r^{-1}\left( \sfd(p_i,\cdot)- \sfd(p_i,x)\right), \; i=1,\ldots,k.$$
This means that for every $r \in (0,1)$ and every $y_1,y_2 \in B^{r^{-1}\sfd}_{1}(x)$  it holds
$$ \left| \sqrt{\sum_{i=1}^k  \left(u^i_{x,r} (y_1)- u^i_{x,r}(y_2) \right)^2} -  r^{-1} \sfd (y_1,y_2)  \right|  \leq  \vare_2 \quad ,$$
which implies, after rescaling by $r$, that for every $r \in (0,1]$ every $y_1,y_2 \in B_{r}(x)$  it holds
\begin{equation} \label{eq:uixrEPSREG}
\left| \sqrt{\sum_{i=1}^k  \left( \sfd(p_i,y_1)- \sfd(p_i,y_2) \right)^2} -   \sfd (y_1,y_2)  \right| = r  \left| \sqrt{\sum_{i=1}^k  \left(u^i_{x,r} (y_1)- u^i_{x,r}(y_2) \right)^2} -  r^{-1} \sfd (y_1,y_2)  \right|  \leq r \vare_2\quad.
\end{equation}
Hence, calling $u: B_1(\bar{x})\to \R^k$ the map $u^i(\cdot):=\sfd(p_i,\cdot)-\sfd(p_i,\bar{x})$ with $i=1,\ldots,k$,  for every $x_1,x_2 \in U$, the above estimate \eqref{eq:uixrEPSREG} ensures that
$$\left|  | u(x_1)-u(x_2)|_{\R^k} - \sfd(x_1,x_2) \right| \leq \vare_2 \sfd(x_1,x_2) \quad, $$
which gives
\begin{equation}\label{eq:uLipEPSREG}
(1-\vare_2) \,  \sfd(x_1,x_2) \leq  | u(x_1)-u(x_2)|_{\R^k} \leq  (1+\vare_2) \, \sfd(x_1,x_2) \quad.
\end{equation}
This is to say the map $u: U\to \R^{k}$ is $(1+\vare)$-biLipschitz to its image in $\R^k$, in virtue of \eqref{eq:defepsEPSREG}; recalling also \eqref{eq:estmBminusUEPSREG} the proof is complete.
 \end{proof}
 By a simple rescaling argument we get the next variant of the $\vare$-regularity Theorem \ref{thm:epsreg1}.
\begin{theorem}\label{thm:epsreg2}
For every  $N\in (1,\infty)$ and $\varepsilon\in (0,1)$ there exists $\delta=\delta(\varepsilon, N)>0$ with the following property. Let $(X,\sfd,\mm)$ be an $\RCD^{*}(-1, N)$-space and assume that  for some $\bar{x} \in X$ it holds
 $$\sfd_{mGH} \left( \Big(B_1(\bar{x}) , \sfd, \mm \Big), \Big(B_1(0),\sfd_{\R^k},\LL_k\Big)  \right)  \leq \delta,  $$
 then there exists a Borel subset $U\subset B_{\delta}(\bar{x})$ such that  $\mm(B_{\delta}(\bar{x})\setminus U)\leq \varepsilon \, \mm(B_{\delta}(\bar{x}))$ and $U$ is $(1+\varepsilon)$-biLipschitz to a subset of $\R^{k}$. 
\end{theorem}

 \def\cprime{$'$}

\


\begin{thebibliography}{10}
\bibitem{AG90}
{\sc U.~Abresch and D.~Gromoll},
  \emph{On complete manifolds with nonnegative Ricci curvature}, J. Amer. Math. Soc., Vol. 3, (1990),  355--374.
  
   
 \bibitem{ACD}
{\sc L.~Ambrosio, M.~Colombo, S.~Di Marino},
  \emph{Sobolev spaces in metric measure spaces: reflexivity and lower semicontinuity of slope}, Variational Methods for Evolving Objects, Adv. Stud. Pure
Math., Vol. 67, Mathematical Society of Japan, Tokio, (2015).
    
\bibitem{AGMR2012}
{\sc L.~Ambrosio, N.~Gigli, A.~Mondino and T.~Rajala},
  \emph{Riemannian Ricci curvature lower bounds in metric spaces with $\sigma$-finite measure},  Trans. Amer. Math. Soc., Vol. 367, Num. 7,  (2015), 4661--4701.
   

\bibitem{Ambrosio-Gigli-Savare08}
  {\sc L.~Ambrosio, N.~Gigli and G.~Savar\'e},
  {\em Gradient flows in metric spaces and in the space of probability measures}, 
  Lectures in Mathematics ETH Z\"urich, Birkh\"auser Verlag, Basel, second~ed., 2008.


\bibitem{AGS11a}
\leavevmode\vrule height 2pt depth -1.6pt width 23pt,
 {\em Calculus and heat flow in
  metric measure spaces and applications to spaces with {R}icci bounds from
  below},  Invent. Math., Vol. 195, Num. 2,  (2014), 289--391.
  
  
\bibitem{AGS11b}
\leavevmode\vrule height 2pt depth -1.6pt width 23pt,
 {\em Metric measure
  spaces with {R}iemannian {R}icci curvature bounded from below},  Duke Math. Journ., Vol. 163, Num. 7,  (2014),  1405--1490.


\bibitem{AGS12}
\leavevmode\vrule height 2pt depth -1.6pt width 23pt,
 {\em Bakry-\'{E}mery
  curvature-dimension condition and {R}iemannian {R}icci curvature bounds},
  Ann. Probab., Vol. 43, Num. 1,  (2015), 339--404.
  
   \bibitem{AH-N}
{\sc L.~Ambrosio and S.~Honda}, {\em {N}ew stability results for sequences of metric measure spaces with uniform Ricci bounds from below}. Preprint, arXiv:1605.07908, 2016.


  
   \bibitem{AMS2013} 
  {\sc L.~Ambrosio, A.~Mondino and G.~Savar\'e}, 
	{\em  Nonlinear diffusion equations and curvature conditions in
metric measure spaces}, Preprint  arXiv:1509.07273.

 \bibitem{AMSLocToGlob}
\leavevmode\vrule height 2pt depth -1.6pt width 23pt,
  {\em On the Bakry-\'Emery condition, the gradient estimates and the Local-to-Global property of $RCD^*(K,N)$ metric measure spaces},  J. Geom.  Anal.,
Vol.  26, (2016), 24--56.



\bibitem{BS2010} 
  {\sc K.~Bacher and K.-T.~Sturm},
  {\em Localization and Tensorization Properties of the Curvature-Dimension Condition for Metric Measure Spaces},
  J. Funct. Anal. \textbf{259} (2010), 28--56.

\bibitem{BakryEmery_diffusions}
{\sc D.~Bakry and M.~Emery},
{\em Diffusions hypercontractives}. Seminaire de Probabilites XIX, Lecture Notes in Math., Springer-Verlag, New York. 1123 (1985), 177--206.

\bibitem{BakryLedoux}
{\sc D.~Bakry and M.~Ledoux},
{\em A logarithmic Sobolev form of the Li-Yau parabolic inequality}. Rev. Mat. Iberoam.,
Vol. 22,  Num. 2, (2006), 683--702. 

\bibitem{BuragoBuragoIvanov}
{\sc D.~Burago, Y.~Burago and S.~Ivanov},
{\em A course in metric geometry},
Graduate Studies in Mathematics, Amer. Math. Soc., (2001).


\bibitem{Cavalletti12}
{\sc F.~Cavalletti},
{\em Decomposition of geodesics in the Wasserstein space and the globalization property},
Geom. Funct. Anal., Vol. 24, (2014), 493--551.    


\bibitem{CM16}
{\sc F.~Cavalletti and A.~Mondino},
{\em Optimal maps in essentially non-branching spaces},
preprint arXiv:1609.00782, to appear in Comm. Cont. Math., DOI: 10.1142/S0219199717500079
 
 \bibitem{CS12}
{\sc F.~Cavalletti and K.~T. ~Sturm},
{\em  Local curvature-dimension condition implies measure-contraction property,},
J. Funct. Anal.,  Vol. 262, (2012),  5110--5127.
 
 
 
 
  \bibitem{Cheeger97}
{\sc J.~Cheeger},
{\em Differentiability of Lipschitz functions on metric measure spaces,} Geom.
Funct. Anal., Vol. 9, (1999), 428--517.

 
 \bibitem{CC96}
{\sc J.~Cheeger and T.~Colding},
{\em Lower bounds on Ricci curvature and the almost rigidity of warped products},
 Ann. of Math., Vol.{144}, (1996), 189--237.

\bibitem{CC97}
\leavevmode\vrule height 2pt depth -1.6pt width 23pt,
 {\em On the structure of spaces with {R}icci
  curvature bounded below {I}}, J. Diff. Geom., Vol. 45, (1997), 406--480.

\bibitem{CC00a}
\leavevmode\vrule height 2pt depth -1.6pt width 23pt, 
{\em On the structure of
  spaces with {R}icci curvature bounded below {II}}, J. Diff. Geom., Vol.54, (2000),
  13--35.

\bibitem{CC00b}
\leavevmode\vrule height 2pt depth -1.6pt width 23pt, 
{\em On the structure of
  spaces with {R}icci curvature bounded below {III}}, J. Diff. Geom., Vol. 54,
  (2000), 37--74.

\bibitem{ChGr}
{\sc J.~Cheeger and D.~Gromoll},
{\em The splitting theorem for manifolds of nonnegative Ricci curvature}, J. Diff. Geometry, Vol. 6, (1971/72),  119--128.

\bibitem{Col96a}
 {\sc T.~Colding},
 {\em Shape of manifolds with positive Ricci curvature},
 Invent. Math., Vol. 124, (1996), 175--191.

\bibitem{Col96b}
 \leavevmode\vrule height 2pt depth -1.6pt width 23pt,
 {\em Large manifolds with positive Ricci curvature},
 Invent. Math., Vol. 124, (1996), 193--214

\bibitem{Col97}
 \leavevmode\vrule height 2pt depth -1.6pt width 23pt,
 {\em {R}icci curvature and volume convergence},
  Ann. of Math., Vol. 145, (1997), 477--501.
 
 
 \bibitem{CN}
{\sc T.~Colding, A.~Naber },
{\em Sharp H\"older continuity of tangent cones for spaces with a lower Ricci curvature bound and applications}, Annals of Math., Vol. 176, (2012).
 
\bibitem{CN13} 
\leavevmode\vrule height 2pt depth -1.6pt width 23pt,
{\em Characterization of tangent cones of noncollapsed limits
with lower Ricci bounds and applications},   Geom. Funct. Anal., Vol. 23, (2013), 134--148.

 \bibitem{EKS2013}
  {\sc M.~Erbar, K.~Kuwada and K.-T.~Sturm},
  {\em On the equivalence of the entropic curvature-dimension condition and {B}ochner's inequality on metric measure spaces}, Invent. Math., Vol. 201, Num. 3,  (2015), 993--1071.
  

\bibitem{GaMo}
{\sc  N.~Garofalo and A.~Mondino}, 
{\em Li-Yau and Harnack type inequalities in $\RCD^*(K,N)$ metric measure spaces},  Nonlinear Analysis: Theory, Methods \& Applications, Vol. 95, (2014), 721--734.

\bibitem{Gigli12}
{\sc N.~Gigli}, {\em On the differential structure of metric measure spaces and
  applications},  Mem. Amer. Math. Soc., Vol. 236, Num. 1113,  (2015).  

\bibitem{GigliSplitting}
\leavevmode\vrule height 2pt depth -1.6pt width 23pt,
{\em The splitting theorem in non-smooth context}, preprint arXiv:1302.5555, (2013).

\bibitem{GigliSplittingSur}
\leavevmode\vrule height 2pt depth -1.6pt width 23pt,
{\em An overview on the proof of the splitting theorem in non-smooth context}, Anal. Geom. Metr. Spaces, Vol. 2, (2014), 169--213. 


\bibitem{GMS2013}
{\sc N.~Gigli, A.~Mondino and G.~Savar\'e},
{\em Convergence of pointed non-compact metric measure spaces and stability of Ricci curvature bounds and heat flows},  Proc. London Math. Soc.,  Vol. 111, Num. 5, (2015), 1071--1129. 

\bibitem{GMR2013} 
{\sc N.~Gigli, A.~Mondino and T.~Rajala},
{\em Euclidean spaces as weak tangents of infinitesimally Hilbertian metric measure spaces with Ricci curvature bounded below}, J. Reine Angew. Math., Vol. 705,  (2015), 233--244. 

\bibitem{GiMo12} 
{\sc N.~Gigli and S.~Mosconi},
{\em The Abresch-Gromoll inequality in a non-smooth setting}, Discrete Contin. Dyn. Syst., Vol. 34, Num. 4,  (2014), 1481--1509.

\bibitem{Hein} 
{\sc  J.~Heinonen},
{\em Lectures on Analysis on Metric spaces}, Universitext, Springer,  (2001).

\bibitem{LD2011}
 {\sc E.~Le~Donne},
 {\em Metric spaces with unique tangents},
 Ann. Acad. Sci. Fenn. Math., Vol. 36, (2011), 683--694.

\bibitem{LY86}
{\sc P.~Li and S.T.~Yau}, {\em On the parabolic kernel of the Schr\"odinger operator},  Acta Math., Vol.156, (1986), 153--201.

\bibitem{Lott-Villani09}
{\sc J.~Lott and C.~Villani}, 
{\em Ricci curvature for metric-measure spaces
  via optimal transport}, Ann. of Math. \textbf{169} (2009), 903--991.

\bibitem{P1987}
 {\sc D.~Preiss},
 {\em Geometry of measures in $\R^n$: distribution, rectifiability, and densities},
 Ann. of Math., Vol. 125 (1987), 537--643.

\bibitem{R2011}
 {\sc T.~Rajala},
 {\em Local Poincar\'e inequalities from stable curvature conditions on metric spaces},
 Calc. Var. Partial Differential Equations, Vol. 44,  (2012), 477--494.
 
 \bibitem{R2013}
\leavevmode\vrule height 2pt depth -1.6pt width 23pt,
 {\em Failure of the local-to-global property for $\CD(K,N)$ spaces}, Ann. Sc. Norm. Super. Pisa Cl. Sci., Vol. 16, (2016), 45--68. 


\bibitem{Savare13}
{\sc G.~Savar\'e}, {\em Self-improvement of the {Bakry-\'Emery condition and
  Wasserstein contraction of the heat flow in RCD$(K,\infty)$} metric measure
  spaces},   Discrete Contin. Dyn. Syst., Vol. 34, Num. 4, (2014), 1641--1661.
 
  
\bibitem{Sturm96}
{\sc K.T.~Sturm}, {\em Analysis on local Dirichlet spaces III. The parabolic Harnack inequality},  J. Math. Pures Appl., Vol. 75 (1996), 273-297.  

\bibitem{Sturm06I}
\leavevmode\vrule height 2pt depth -1.6pt width 23pt,
{\em On the geometry of
  metric measure spaces. {I}}, Acta Math. \textbf{196} (2006), 65--131.

\bibitem{Sturm06II}
\leavevmode\vrule height 2pt depth -1.6pt width 23pt,
 {\em On the geometry of
  metric measure spaces. {II}}, Acta Math. \textbf{196} (2006), 133--177.

\bibitem{Villani09}
{\sc C.~Villani}, {\em Optimal transport. Old and new}, vol.~338 of Grundlehren
  der Mathematischen Wissenschaften, Springer-Verlag, Berlin, 2009.
\end{thebibliography}
\end{document}